\documentclass[11pt]{article}

\usepackage{amsfonts}
\usepackage{amsthm}
\usepackage{amscd}
\usepackage{amssymb}
\usepackage{graphics}
\usepackage{amsmath}
\usepackage[english]{babel}
\usepackage{color}
\usepackage[usenames,dvipsnames]{xcolor}
\usepackage[colorlinks=true,linkcolor=NavyBlue,urlcolor=RoyalBlue,citecolor=PineGreen,linktocpage=true]{hyperref}
\usepackage{hyperref}
\usepackage{bbm}
\usepackage{yfonts}
\usepackage{mathrsfs}
\usepackage{upgreek}
\usepackage{epsfig}
\usepackage{graphicx}
\usepackage{enumerate}
\usepackage{url}
\usepackage[T1]{fontenc}
\usepackage[latin9]{inputenc}
\usepackage{blkarray}

\usepackage{geometry}
\DeclareMathAlphabet{\mathpzc}{OT1}{pzc}{m}{it}

\numberwithin{equation}{section}

\usepackage{titlesec}
\titleformat{\subsection}[runin]{\normalfont\bfseries}{\thesubsection.}{.5em}{}[.]\titlespacing{\subsection}{0pt}{2ex plus .1ex minus .2ex}{.8em}
\titleformat{\subsubsection}[runin]{\normalfont\itshape}{\thesubsubsection.}{.3em}{}[.]\titlespacing{\subsubsection}{0pt}{1ex plus .1ex minus .2ex}{.5em}
\titleformat{\paragraph}[runin]{\normalfont\itshape}{\theparagraph.}{.3em}{}[.]\titlespacing{\paragraph}{0pt}{1ex plus .1ex minus .2ex}{.5em}

\newtheorem{theorem}{Theorem}[section]
\newtheorem{lemma}[theorem]{Lemma}
\newtheorem{corollary}[theorem]{Corollary}
\newtheorem{proposition}[theorem]{Proposition}

\newtheorem{claim}[theorem]{Claim}

\theoremstyle{definition}

\theoremstyle{remark}
\newtheorem{remark}[theorem]{Remark}
\newtheorem*{remark*}{Remark}

\newcommand{\BC}{{\mathbb{C}}}

\newcommand{\BE}{{\mathbb{E}}}

\newcommand{\BG}{{\mathbb{G}}}

\newcommand{\BN}{{\mathbb{N}}}

\newcommand{\BP}{{\mathbb{P}}}

\newcommand{\BS}{{\mathbb{S}}}

\newcommand{\BZ}{{\mathbb{Z}}}

\newcommand{\FM}{{\mathfrak{M}}}
\newcommand{\FN}{{\mathfrak{N}}}

\newcommand{\Fp}{{\mathfrak{p}}}

\newcommand{\FX}{{\mathfrak{X}}}

\newcommand{\Fe}{{\mathfrak{e}}}
\newcommand{\Fq}{{\mathfrak{q}}}

\newcommand{\CE}{{\mathcal{E}}}
\newcommand{\CF}{{\mathcal{F}}}
\newcommand{\CG}{{\mathcal{G}}}

\newcommand{\CI}{{\mathcal{I}}}

\newcommand{\CL}{{\mathcal{L}}}

\newcommand{\CR}{{\mathcal{R}}}
\newcommand{\CS}{{\mathcal{S}}}

\newcommand{\CW}{{\mathcal{W}}}

\newcommand{\tbD}{{\textbf{D}}}

\newcommand{\tbP}{{\textbf{P}}}

\newcommand{\tba}{{\textbf{a}}}

\newcommand{\tbe}{{\textbf{e}}}

\newcommand{\tbl}{{\textbf{l}}}

\newcommand{\tbp}{{\textbf{p}}}
\newcommand{\tbq}{{\textbf{q}}}
\newcommand{\tbr}{{\textbf{r}}}

\newcommand{\tbu}{{\textbf{u}}}
\newcommand{\tbv}{{\textbf{v}}}
\newcommand{\tbw}{{\textbf{w}}}

\newcommand{\tbz}{{\textbf{z}}}

\newcommand{\SV}{{\mathscr{V}}}

\newcommand{\tri}{{\vartriangleleft}}
\newcommand{\trieq}{{\trianglelefteq}}

\newcommand{\ind}{{\mathbbm{1}}}

\newcommand{\sfV}{{\mathsf{V}}}

\newcommand{\ep}{{\varepsilon}}

\newcommand{\si}{{\sigma}}
\newcommand{\lam}{{\lambda}}

\usepackage{tikz}

\usepackage{ifthen}

\usetikzlibrary{calc}
\usetikzlibrary{petri}
\usepackage{pgf}
\usepackage{pgffor}

\usetikzlibrary{arrows,shapes,positioning}
\usetikzlibrary{decorations.markings}

 \usepackage{cleveref}
\makeatother

\tikzstyle arrowstyle=[scale=1]
\tikzstyle directed=[postaction={decorate,decoration={markings,
    mark=at position .65 with {\arrow[arrowstyle]{stealth}}}}]
\tikzstyle reverse directed=[postaction={decorate,decoration={markings,
    mark=at position .65 with {\arrowreversed[arrowstyle]{stealth};}}}]
    \tikzstyle{placec}=[circle, draw=black,fill=white, inner sep=2pt]
     \tikzstyle{placer}=[rectangle, draw=black,fill=white, inner sep=3pt]

\newcommand{\RN}[1]{
  \textup{\uppercase\expandafter{\romannumeral#1}}
}

\newcommand{\Tr}{\textrm{Tr}}
\newcommand{\beq}{\begin{equation}}
\newcommand{\eeq}{\end{equation}}

\newcommand{\E}{\mathbb{E}}
\newcommand{\eps}{\varepsilon}

\newcommand{\bbC}{\mathbb{C}}

\newcommand{\D}{\mathbf{D}}
\newcommand{\bbN}{\mathbb{N}}
\newcommand{\C}{\mathbb{C}}

\newcommand{\caI}{\mathcal I}
\newcommand{\caS}{\mathcal {S}}

\newcommand{\bfz}{\mathbf{z}}
\newcommand{\bfD}{\mathbf{D}}
\newcommand{\bfw}{\mathbf{w}}

\newcommand{\frN}{\mathfrak N}
\newcommand{\frf}{\mathfrak f}

\newcommand{\dg}{\dagger}

\newcommand{\frh}{\mathfrak h}
\newcommand{\frK}{\mathfrak K}
\newcommand{\frn}{\mathfrak n}

\newcommand{\la}{\lambda}
\newcommand{\frl}{\mathfrak l}

\newcommand{\frr}{\mathfrak r}
\newcommand{\frL}{\mathfrak L}
\newcommand{\frR}{\mathfrak R}

\newcommand{\caW}{\mathcal W}

\newcommand{\fre}{\mathfrak{e}}
\newcommand{\frp}{\mathfrak{R}}
\newcommand{\frQ}{\mathfrak{L}}

\begin{document}

\title{Eigenvector correlations in the complex Ginibre ensemble}
\author{Nicholas Crawford\thanks{supported by Israel Science Foundation grant number 1692/17}  ~ and~ Ron Rosenthal\thanks{supported by Israel Science Foundation grant number 771/17}}

\maketitle


\begin{abstract}
	The complex Ginibre ensemble is an $N\times N$ non-Hermitian random matrix over $\BC$ with i.i.d. complex Gaussian entries normalized to have mean zero and variance $1/N$. Unlike the Gaussian unitary ensemble, for which the eigenvectors are distributed according to Haar measure on the compact group $U(N)$, independently of the eigenvalues, the geometry of the eigenbases of the Ginibre ensemble are not particularly well understood. In this paper we systematically study  properties of eigenvector correlations in this matrix ensemble. In particular, we uncover an extended algebraic structure which describes their asymptotic behavior (as $N$ goes to infinity). Our work extends previous results of Chalker and Mehlig \cite{CM1}, in which the correlation for pairs of eigenvectors was computed. 
 \end{abstract}

\section{Introduction}

	For $N\in\BN$, let $M_N=(M_{ij})_{i,j=1}^{N}$ denote a random $N\times N$ matrix sampled from the Ginibre ensemble. That is, $M_N$ is an $N\times N$ matrix with complex entries $M_{ij}$, where $(M_{ij})_{i,j\geq 1}$ are independent complex Gaussian random variables with mean $0$ and variance $1/N$. Let us denote the probability measure of such a complex Gaussian variable by $\kappa_N$
\[
	\kappa_N(\textrm{d}z)= \pi^{-1}\exp(-N|z|^2)N\textrm{d}^2 z,\quad \forall z\in\BC\,,
\]
and by $\BP$ the joint law of the random variables $(M_{ij})_{i,j\geq 1}$. 

With $\BP$-probability $1$, the matrix $M_N$ is diagonalizable and has $N$ distinct eigenvalues, denoted $(\lambda_i)_{i=1}^N$. The eigenvalue (and singular value) distribution is explicitly computable as a determinantal processes. In particular, the spectrum and $k$-point eigenvalue distributions of $M_N$ converge in $\BP$-probability (and also almost surely), in the limit $N\rightarrow \infty$, to the uniform measure on the unit disc $\textbf{D}_1:=\{z\in\BC ~:~ |z|<1\} \subset \mathbb C$ (resp. $\D_1^k$). Furthermore, like its better-known Hermitian counterpart, the Gaussian unitary ensemble, both the eigenvalues and singular values display universal behavior with respect to the variation of the distribution (under modest analytic assumptions on the distribution) of individual entries of the matrix, see \cite{Gin,Gir,Gir94,bai1997,TV,gotze,TV10,BYY14,BYY14b,Y14,AEK18}. 

Since the matrix $M_N$ is generically non-Hermitian with respect to the natural inner product on $\C^N$, one can associate to $(\lam_i)_{i=1}^N$ two sets of bases of eigenvectors for $\C^N$, a basis of `right' eigenvectors $(\tbr_i)_{i=1}^n$ and a basis of `left' eigenvectors $(\tbl_i)_{i=1}^N$. In the natural coordinate system defining $M_N$, the $\tbr_i$'s are column vectors and $M_N  \tbr_i= \lambda_i \tbr_i$ while the $\tbl_i$'s are row vectors and $\tbl_i M_N=\lambda_i \tbl_i$. Given the $\tbr_i$'s, the $\tbl_i$'s are normalized so that 
\begin{equation}\label{E:Norm}
	 \tbl_i \cdot \tbr_j=\delta_{i j}\,.
\end{equation} 

	A point worth keeping in mind  throughout the paper is that properly speaking, $\tbr_i\in \BC^N$ while $\tbl_i\in (\BC^N)^*$. Hence we will use the convention $ \tbl_i(k):=\tbl(e_k)$, where $e_k$ is the column vector with $1$ in the $k$'th component and $0$ elsewhere.  With this convention $ \tbl_i \cdot \tbr_j =\sum_k \tbl_i(k)\tbr_j(k)$ \emph{without} any complex conjugation. Using these bases, we have the spectral decomposition $M_N=\sum_{1\leq i\leq N}\lambda_iQ_i$, where for $1\leq i\leq N$, we introduced the notation $Q_i=\tbr_i\tbl_i$. 

Unlike what transpires in the Hermitian setting, the eigenvectors $\tbr_i$ are strongly correlated with the eigenvalues $\lambda_i$. This has a number of interesting consequences. For example, there is no simple description for the analogue of Dyson's Brownian motion of the Ginibre ensemble, since one has to track the evolution of both eigenvalues and eigenvectors. 

In this paper, we study of the \textit{geometry} of the eigenbases {$(\tbr_i)_{i=1}^N$ and $(\tbl_i)_{i=1}^N$}.  We shall provide further motivation for this below but, for now, let us simply say that we find related questions intrinsically interesting. For example, what are the typical angles between distinct eigenvectors? What is the volume of the parallelepiped determined by the eigenvectors corresponding to eigenvalues near the deterministic parameters $\nu_1, \cdots \nu_\ell$ for some fixed $\ell\in \bbN$ as $N\rightarrow \infty$? How does the minimal angle between eigenvectors behave as a function of $N$? 

Here, we focus on the computation of the \emph{$2\ell$-point correlation functions} for the collection of eigenvectors corresponding to the deterministic eigenvalues $\nu_1, \dotsc, \nu_{2\ell}$.  
To explain what we mean,  let $\tbr_i^{\dg}\in (\BC^N)^*$ (respectively $\tbl_i^\dg \in \BC^N$) denote the row (column) vector obtained by conjugate transposing  $\tbr_i$ ($\tbl_i$). When $\ell=1$, the natural quantity to study is, formally,
\begin{equation}\label{eq:2-point-def}
	\sum_{i, j=1}^N \delta(\nu_1-\lambda_i) \delta(\nu_2-\lambda_j) \tbr_i^{\dg}\cdot  \tbr_j \,,
\end{equation}
where $\tbr_i^{\dg}\cdot  \tbr_j:=\sum_k \overline{\tbr_i(k)}  \tbr_j(k)$. Aside from the minor technical issue that one must formalize the meaning of the delta functions, the expression itself is actually not well defined due to the following symmetry. For every choice of non-zero complex numbers $(c_i)_{i=1}^{N}$, the pair of bases $(c_i \tbr_i)_{i=1}^N, (c_i^{-1}\tbl_i)_{i=1}^N$ is just as good as $( \tbr_i)_{i=1}^N, (\tbl_i)_{i=1}^N$. From a physical perspective, see \cite{CM1,CM2}, it is natural in such a situation to focus on quantities invariant under this symmetry. Let $\CI_{2\ell, N}$ be the set of all $2\ell$-tuples of integers in $[N]$. For every $\eps>0$, $N\in \bbN$, and every $\pmb{\nu}\in \textbf{D}_1^{2\ell}$, define
\begin{equation}\label{eq:finite_correlation_function}
\begin{aligned}
	\widehat{\rho}_{2\ell, N, \eps}(\pmb{\nu}):\hspace{-2pt}&= \eps^{-4\ell}N^{-1} \sum_{I\in\CI_{2\ell, N}} \phi_{\eps}(\|\pmb{\nu}-\lambda_I\|_\infty) \prod_{m=1}^{\ell}  (\tbr_{I_{2m-1}}^{\dg}\cdot \tbr_{I_{2m}})( \tbl_{I_{2m}}\cdot \tbl_{I_{2m+1}}^{\dg})\\
	&= \eps^{-4\ell}N^{-1} \sum_{I\in\CI_{2\ell, N}} \phi_{\eps}(\|\pmb{\nu}-\lambda_I\|_\infty) \mathrm{Tr}(Q_{I_1}^\dag Q_{I_2}Q_{I_3}^\dag \cdots Q_{I_{2\ell}})\,,
\end{aligned}
\end{equation}
where $\lambda_I=(\lambda_{I_m})_{m\in I}$, the function $\phi_\varepsilon$ denotes the indicator function of the interval $(-\varepsilon,\varepsilon)$, and for $I=(I_1,\ldots,I_{2\ell})$, we have implemented the cyclic notation $I_{m+2\ell}=I_{m}$. 

Our goal is to show the existence of the limit 
\[
	\rho_{2\ell}(\pmb{\nu}):=\lim_{\eps \rightarrow 0} \lim_{N\rightarrow \infty}\E[\widehat{\rho}_{2\ell, N, \eps}(\pmb{\nu})]\,,
\]  
and then explore some of its basic properties as a function of the spectral parameters $\pmb{\nu}$. Let us note that another way to interpret this quantity is to compute the conditional expectation of $\mathrm{Tr}(Q_1^*Q_2Q_3^*\cdots Q_{2\ell})$ conditioned on $(\lambda_1,\ldots,\lambda_{2\ell})=\pmb{\nu}$. For technical convenience we work with \eqref{eq:finite_correlation_function}, although we expect both definitions to lead to the same correlation functions on the macroscopic scale.

In fact, the case $\ell=1$ has already been computed by Chalker and Mehlig \cite{CM1,CM2,CM3}, for which they obtained the beautiful formula 
\begin{equation}\label{eq:CM}
	{\rho_2}(\nu_1,\nu_2)=-\frac{1-\overline{\nu_1}\nu_2}{|\nu_1-\nu_2|^4}\,,\qquad\qquad  \forall \nu_1,\nu_2\in\tbD_1\,.
\end{equation}

\begin{remark*}
	Note that in \cite{CM1}, this formula appears with a factor of $\pi^{-2}$. The origin of this factor is due to a different normalization. 
\end{remark*}

To illustrate our main result for the next order correlation function (when $\ell=2$), we prove that 
\[
	\rho_4(\nu_1,\nu_2,\nu_3,\nu_4)= \frac{1}{(\nu_2-\nu_4)(\overline{\nu_1}-\overline{\nu_3})}\Big[{\rho}_2(\nu_1,\nu_2){\rho}_2(\nu_3,\nu_4)-{\rho}_2(\nu_1,\nu_4){\rho}_2(\nu_3,\nu_2)\Big]\,.
\]

On the basis of this formula one can already begin to see some of the general structure structure that we uncover. First, note that $\rho_4$ can be expressed in terms of products of $\rho_2$'s with coefficients that are rational functions in $(\nu_{2k})_k$ and $(\overline{\nu_{2k-1}})_k$. In fact, one can interpret $\nu_2-\nu_4$ as a Vandermonde determinant (and similarly $\overline{\nu_1}-\overline{\nu_3}$). This turns out be a general feature of the correlation functions; In the general case, the correlation functions can be expressed as a linear combination of products of $\rho_2$'s with a specific pairing rule that we explain below. Moreover, after multiplying by a product of Vandermonde determinants, the coefficients of each product of $\rho_2$'s are polynomials in $(\nu_{2k})_k$ and $(\overline{\nu_{2k-1}})_k$.

\subsection{Motivation}

Chalker and Mehlig were motivated to compute \eqref{eq:CM} after considering the following problem: Let $M_N$ and $M'_N$ be a pair of independent random matrices distributed according to the Ginibre ensemble and  interpolate from one to the other via $M_N(\theta)=\cos(\theta)M_N+ \sin(\theta) M_N'$. How do the eigenvalues $\la_i(\theta)$ vary with $\theta$? Note that for any fixed $\theta$, the matrix $M_N(\theta)$ has the same distribution as $M_N$. However, examining the velocities of eigenvalues, Chalker and Mehlig found that $\E[|\partial_{\theta} \lambda_i|^2|\lambda_i=z]\approx 1- |z|^2$. Here the square of the length of $Q_i$ in the Hilbert-Schmidt norm naturally appears; $\E[|\partial_{\theta} \lambda_i|^2|\lambda_i=z]=\E[\mathrm{tr}(Q_i^\dag Q_i)|\lambda_i=z]/N$. It turns out that this quatity is asymptotic to $1-|z|^2$.  By way of comparison, for the (self adjoint) Gaussian Unitary Ensemble with the same normalization on the matrix entries, it is known that $\E[\partial_{\theta} \lambda_i^2]=O(1/N)$, see \cite{Wilk}. This indicates a strong instability in the spectrum of a non-Hermitian random matrix which cannot be captured by the typical studies of eigenvalues alone.

This instability turns out to be common in physical and numerical problems involving non-Hermitian matrices. For example, it is the origin of computational issues related to matrix inversion \cite{Grcar} and motivated the study of \emph{pseudo-spectra}, that is, approximate eigenvalues and eigenvectors, when faced with a non-normal matrix. The latter concept has found application in numerous settings (see the wonderful book \cite{Tref3} for a survey). Let us mention a few examples we find interesting, although the discussion is a bit off our main topic. 

First, the instability of eigenvalues and eigenvectors of non-normal operators (and the success of pseodspectra in detecting it) was connected to the onset of turbulence in certain fluid flows at Reynolds numbers lower than what might naively be expected based on linearized stability analysis \cite{Tref1, Tref2}. 

A second interesting example in which non-normality and stability of eigenvalues and eigenvectors plays a significant role is the Hatano-Nelson model \cite{HN,Goldshield,Tref3}. This is an Anderson type model in one dimension in which the propagator is a tilted Laplacian, breaking time reversal symmetry. According to numerical results presented in \cite{Tref3}, the eigenvalues and eigenvectors of this model display rather strong sensitivity to boundary conditions, periodic versus Dirichlet. Remarkably however, the pseudo-spectra of this operator seems to be relatively insensitive to boundary conditions. 

As a third and final example regarding the significance of this instability, we mention recent work by Fyodorov and Savin regarding inverse lifetimes of resonance states in open quantum systems \cite{FS12}. If the Hamiltonian of the corresponding closed system is perturbed, the inverse lifetimes of the resonance states shift. The authors show that  at weak coupling between the system and environment, the magnitudes of these shifts are predominantly due to the non-orthogonality of the resonance states. Moreover, considering the limit in which the number of resonance states tends to infinity, they use random matrix techniques to predict the statistics for these shifts. Subsequent experimental work \cite{GKL14} actually confirmed these random matrix theory approximations. This last fact is particularly significant from our point of view: By analogy with the Hermitian setting, in which the local statistics of eigenvalue correlations of GUE is expected to be universal over a large class of  models both random and deterministic, one might hope to use non-Hermitian random matrices to study large dissipative quantum systems. The work \cite{GKL14} provides proof of concept for this hope.

Another point of view one might take, which makes the quantities we study quite natural, is that Chalker and Mehlig  computed the two point, or spin-spin, correlation function of a statistical mechanical system: As is well known, the eigenvalues of the Ginibre ensemble form a system of free fermions, or, equivalently, a determinantal point process. In this interpretation, the associated eigenvectors should not be forgotten as they provide an additional spin structure for this system. 

There was a bit of followup work after \cite{CM1,CM2}, c.f. \cite{Poles1,DBM}, but little attention has been paid by the mathematics community until rather recently. Walters and Starr \cite{WS}, extended the calculation of $\rho_2$ to spectral parameters at the boundary of $\D_1$ and also computed the asymptotics of mixed moments of $\widehat{\rho}_{2,N,\varepsilon}$ with powers of the random matrix $M_N$. Bourgade and Dubach \cite{BD} and Fyodorov \cite{F} characterized the \emph{distribution} of the "self-overlap", namely, the length square of a particular $Q_i$, which also represents the local condition number for $M_N$. It would be interesting to extend their results to a "full counting statistics" for all $2\ell$-point correlations of the type considered in the present paper. 

We focus exclusively on eigenvector correlations for the complex Ginibre ensemble.  As we detail below, there seems to be an interesting algebraic structure underlying our computations which we have only partially uncovered. It would therefore be worthwhile to explore the analogous correlations for other non-Hermitian ensembles, e.g., the real and quaternionic cases.  

\subsection{Some basic notions and the main result}\label{subsec:basic_notions}

	For a finite set $A\subset \BN$ denote by $\CS_A$ the permutation group on $A$, and define $\CS=\biguplus_{A\subset \BN,~ 0\leq |A|<\infty}\CS_A$. Here, $\biguplus$ denotes the \emph{disjoint} union of these collections of permutations, that is, we do not identify permutations which have the same fixed points, so the cycle $(1, 2, 3)\in \CS_{[3]}$ is not the same as $(1,2,3)(4)\in \CS_{[4]}$ nor the same as the cycle $(1,2,4)\in \CS_{\{1,2,4\}}$. The set $\CS_\emptyset$  contains a unique permutation on the from the empty set to itself, which we denote by $\emptyset$. For $\si\in\CS$ we denote by $\SV(\sigma)$ the set of vertices on which $\si$ acts and by $|\si|$ the number of cyclic permutation in $\si$. 
	
It follows from our formalism that, when computing \eqref{eq:finite_correlation_function}, we also compute correlations corresponding to product of cycles and hence correlation functions associated with general permutations in $\CS$. In order to study correlation functions associated with general permutations, we change our notation slightly. For $\sigma\in\CS$, let $\CI_{\sigma}=\CI_{\sigma}^N$ be the collection of tuples indexed by $\SV(\sigma)$, taking values in $[N]$. Then, for $\si\in\CS$ and $\tbu,\tbv\in \bfD_1^{\SV(\sigma)}$, set 
\begin{multline}\label{eq:finite_correlation_function1}
	\widehat{\rho}_{N,\eps}(\sigma; \tbu,\tbv):\hspace{-0.7pt}=\eps^{-4|\SV(\sigma)|}N^{-|\sigma|}
	 \sum_{I,J\in \CI_{\sigma}} \phi_{\eps}(\| \tbu-\lambda_{I}\|_\infty)\phi_{\eps}(\| \tbv-\lambda_{J}\|_\infty) \times\\
	\prod_{j\in \SV(\sigma)}  (\tbl_{J_{\sigma^{-1}(j)}}\cdot \tbl_{I_{j}}^{\dg}) (\tbr_{I_j}^{\dg} \cdot \tbr_{J_{j}})\, ,
\end{multline}

Note that when $\sigma$ has more than one cycle in it, $\widehat{\rho}_{N,\ep}(\si)$ is rescaled by a higher power of $N$, namely $N^{-|\sigma|}$, corresponding to the total number of cycles in the correlation function. Relating our new notation to \eqref{eq:finite_correlation_function}, by taking $C_\ell$ to be the cyclic permutation on $[\ell]$ and defining $\pmb{\nu}\in\tbD_1^{2\ell}$ via the relation $\nu_{2j-1}=v_j$ and $\nu_{2j}=u_j$, one obtains $\widehat{\rho}_{N,\varepsilon}(C_\ell;\tbu,\tbv)=\widehat{\rho}_{2\ell,N,\varepsilon}(\pmb{\nu})$. 

For $m\in\BN$ and $\tbu,\tbv\in \tbD_1^{m}$, define 
\begin{align*}
	& \mathrm{Dist}(\tbu,\tbv) = \mathrm{Dist}(\tbu,\tbv;\partial\tbD_1)\\
	&\qquad :=  \min_{\alpha,\beta\in [m]} \{|u_\alpha-v_\beta|\}  \wedge \min_{\alpha,\beta\in [m],~\alpha\neq \beta}\{|u_\alpha-u_\beta|,|v_\alpha-v_\beta|\}\wedge \min_{\alpha\in [m]}\{1-|u_\alpha|, 1-|v_\alpha|\}\,. 
\end{align*}

Our first theorem proves the existence of the limiting correlation function in the macroscopic scale of separation, i.e., $\mathrm{Dist}(\tbu,\tbv)>0$ and a factorization property it satisfies. 

\begin{theorem}\label{T:1}
For every $\si\in\CS$ and every $\tbu,\tbv\in\tbD_1^{\SV(\sigma)}$ such that $\mathrm{Dist}(\tbu,\tbv)>0$, the limit
\[
	\rho(\sigma; \tbu,\tbv):=\lim_{\eps \rightarrow 0} \lim_{N\rightarrow \infty}\E[\widehat{\rho}_{N,\eps}(\sigma; \tbu,\tbv)]
\]
exists. Moreover, if $\sigma=\{\CL_k\}_{k=1}^{|\si|}$ where $\CL_k$ are the cycles of $\sigma$, then
\[
	{\rho}(\sigma; \tbu,\tbv)= \prod_{k=1}^{|\si|}{\rho}(\CL_k; \tbu|_{\SV(\CL_k)}, \tbv|_{\SV(\CL_k)}).
\]
\end{theorem}
This factorization property only occurs with macroscopic separation of the spectral parameters. If one considers the same objects on the scale $|u_i-v_j|=O(N^{-1/2})$, this result is no longer true, as we show in another forthcoming paper \cite{CR2}. Theorem \ref{T:1} implies that it suffices to understand the limiting correlation function for cyclic permutations. In order to study these we introduce some additional definitions. Let $\si, \tau\in\CS$ be two cyclic loops such that $\SV(\si)\subseteq \SV(\tau)$. We say that $\si$ is a sub-loop of $\tau$ if the map $\tau\circ \si^{-1}$ has at most one non-fixed point (as a map from $\SV(\si)$ to $\SV(\tau)$). Any permutation $\si\in\CS$ induces an orientation on $\SV(\si)$. In particular, for $m\geq 3$ and $\alpha_1,\ldots,\alpha_m\in\SV(\si)$ one can check whether $(\alpha_1,\ldots,\alpha_m)$ belong to the same cycle in $\si$ and appear on it with the prescribed ordering. Given two disjoint subloops $\pi_1$ and $\pi_2$ of a cyclic permutation $\si$, we say that they are crossing if there exists $\alpha\in \SV(\pi_1)$ and $\beta\in\SV(\pi_2)$ such that $(\alpha,\pi_1(\alpha),\beta,\pi_2(\beta))$ is not the ordering of these vertices in $\si$. Otherwise we say that $\pi_1$ and $\pi_2$ are non-crossing. There is a natural partial ordering on permutations in $\CS$ defined as follows. For $\si,\tau\in\CS$, say that $\sigma \trianglelefteq \tau$ if $\SV(\sigma) \subset \SV(\tau)$, every loop of $\sigma$ is a sub-loop of some loop in $\tau$ and all pairs of loops of $\si$ are non-crossing with respect to $\tau$. Finally, for $A\subset \BN$ and $\tbv\in\BC^A$, let $\mathsf{V}_A(\tbv)$ be the Vandermonde determinant $\prod_{\alpha, \beta\in A,\, \alpha< \beta}(v_{\beta}-v_{\alpha})$.  

\begin{figure}
  \centering
    \includegraphics[width=0.6\textwidth]{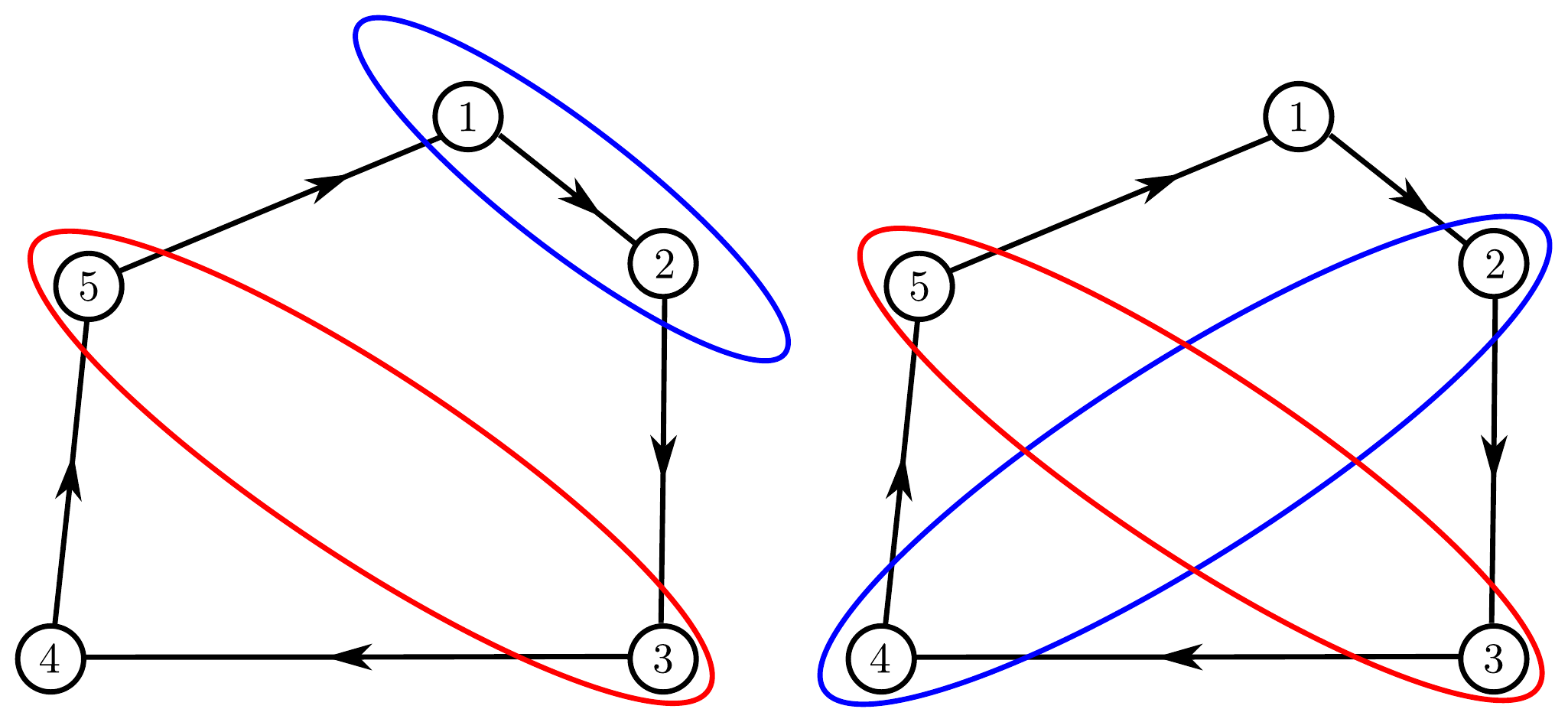}
     \caption{In black the cycle $\si=(1,2,3,4,5)$. On the left two subcycles $\pi_1=(1,2)$ and $\pi_2=(3,5)$ which are noncrossing. On the right, two subcycles $\pi_1=(2,4)$ and $\pi_2=(3,5)$ which are crossing with respect to $\si$. The sequence $(2,4,3,5)$ does not the order of the vertices in $\si$. \label{fig:crossing}}
\end{figure}

\begin{theorem}\label{T:2}
	There are two families of polynomials $(\frp_{\pi}, \frQ_{\pi})_{\pi \in \CS}$, 
with $\frp_{\pi}$ and $\frQ_\pi$ being homogeneous polynomials from $\BC^{\SV(\pi)}\times \BC^{\SV(\pi)}$ to $\BC$ of degree $\binom{|\SV(\pi)|-1}{2}$ in the first and second sequence of variables, 	such that for every permutation $\si\in\CS$ and every $\tbu,\tbv\in\tbD_1^{\SV(\si)}$
\[
	\rho(\si;\tbu,\tbv) =\hspace{-8pt}\sum_{\substack{\pi\trieq \si \\ \SV(\pi)=\SV(\si)}} \frac{\frp_{\pi}(\overline{\tbu}|_{\SV(\pi)},\tbv|_{\SV(\pi)})\frQ_{\si\circ \pi^{-1}}(\overline{\tbu}|_{\SV(\pi)},\pi^{-1}(\tbv|_{\SV(\pi)}))}{\mathsf{V}_{\SV(\si)}(\overline{\tbu})^2\mathsf{V}_{\SV(\si)}({\tbv})^2}\prod_{\alpha\in \SV(\pi)}\rho_2(u_\alpha,v_{\pi^{-1}(\alpha)})\,,
\]
where $\rho_2$ given by \eqref{eq:CM}.
\end{theorem}

\begin{remark*}
	In spite of the factorization stated in Theorem \ref{T:1}, note that Theorem \ref{T:2} suggests, and this is born out by its proof, that even the cycle correlation function cannot be disentangled from the correlation functions corresponding to more complicated permutations. 
\end{remark*}

We have computed these polynomials (as well as the associated correlation functions) for $\sigma=(1,2)$, see Section \ref{sec:The_matrix_frN}, and also in the case $\si=(1,2,3)$ and $\sigma=(1, 2, 3, 4)$ (unpublished). We think it would be an interesting combinatorial problem to give a closed expressions for them, but have not completely achieved this goal for the time being. 

Nevertheless, we finish this section by describing some structural properties the polynomials satisfy and providing a recurrence equation for their computation. To this end let us introduce a matrix $\frN$, indexed by the elements of $\CS$, which plays a crucial role in our paper. For $\si,\tau\in\CS$ we say that $\si\preceq\tau$ if $\SV(\si)\subset\SV(\tau)$, every loop of $\si$ is a subloop of $\tau$ and all but at most one of the loops of $\tau$ are also loops in $\si$. Notice that $\si\preceq \tau$ implies $\si\trianglelefteq\tau$ but the converse is not necessarily true. Given $\si,\tau\in\CS$ such that $\si\preceq \tau$  define $\SV_{nf}(\si;\tau)$ to be the set of the \emph{non-fixed points} of $\tau\circ \sigma^{-1}$ as map from $\SV(\si)$ to $\SV(\tau)$ and let $\widehat{\SV}_{nf}(\si;\tau) = \SV_{nf}(\si;\tau)\cup (\SV(\tau)\setminus \SV(\si))$.

Then we define the functions $\frn:\CS \times \prod_{\tau\in\CS}\BC^{\SV(\tau)}\times \BC^{\SV(\tau)}\to\BC$ and $\frh:\prod_{\tau\in\CS}\BC^{\SV(\tau)}\times \BC^{\SV(\tau)}\to \BC$ by 
\begin{equation}\label{the_function_Fn}
	\frn_{\sigma,\tau}(\tbu,\tbv) = \frac{1}{\pi}\int_{\D_1}\prod_{\alpha\in \widehat{\SV}_{nf}(\sigma;\tau)} \frac{1}{(\overline{\nu}-\overline{u_{\alpha}})} \frac{1}{(\nu-v_{\sigma^{-1}(\alpha)})} \textrm{d}^2 \nu
\end{equation}
and
\begin{equation}\label{eq:the_function_h}
	\frh_\si(\tbu,\tbv) = \sum_{\alpha\in\SV(\si)}h(u_\alpha,v_{\sigma^{-1}(\alpha)})\,, 
\end{equation}
where for $u,v\in\tbD_1$, we define\footnote{Note that $\partial_{u}\partial_{\overline{v}}h(u,v) = \rho_2(u,v)$. 
}
\[
	h(u, v)=\frac{1}{\pi}\int_{\D_1} \frac{1}{(\overline{\nu}-\overline{u})(\nu-v)} \textrm{d}^2 \nu = 
	\log\Big(\frac{1-\overline{u}v}{|u-v|^2}\Big)\,.
\]

The matrix $\frN\equiv \frN(\tbu,\tbv)$ is then defined by
\begin{equation}\label{eq:The_Matrix_N}
	\frN_{\sigma, \tau}=\begin{cases}
		\frh_{\si}(\tbu,\tbv) & \text{ if }\sigma=\tau, \\
		\frn_{\sigma, \tau}(\tbu,\tbv) & \text{ if }\sigma\prec \tau,\\
		0 & \text{ otherwise}.
\end{cases}
\end{equation}

Note that $\frN$ is upper triangular since $\preceq$ is a partial order, and therefore its eigenvalues are the diagonal entries $\frh_{\si}(\tbu,\tbv)$. 

The equations for the eigenvectors of $\frN$ can be written recursively. If $\frl_\pi$, respectively $\frr_{\pi}$, denote the left (respectively right) eigenvectors, then 
\begin{align}\label{eq:EV}
	&[\frh_{\si}-\frh_{\tau}]\frl_\si(\tau)= \sum_{\si\trieq \pi \prec\tau} \frl_{\si}(\pi) \frN_{\pi, \tau},\qquad \forall \si, \tau\in\CS\,,\\
	&[\frh_{\tau}-\frh_{\si}]\frr_\tau(\si)= \sum_{\si\tri \pi \preceq\tau} \frN_{\si, \pi} \frr_{\tau}(\pi),\qquad \forall \si, \tau\in\CS\,,
\end{align}
subject to the initial conditions $\frl_{\si}(\si)=\frr_{\si}(\si)=1$ for all $\si\in\CS$. For $|\SV(\si)|, |\SV(\tau)|\leq 2$ we can compute $\frl_\si(\tau), \frr_\tau(\si)$ directly from the eigenvector equation. Let $\Delta=\mathsf{V}_{12}(\tbu)^{-1}\mathsf{V}_{12}(\tbv)^{-1}$.  Then they are expressed in the following matrices (with rows indexed by the subscript):

\hspace{10pt}

\noindent
\[
\begin{array}{c|rrrrr}
\frl_\si(\tau)&\varnothing&(1) & (2) &(1)(2) &(12)\\
\hline
\varnothing& 1& -1& -1&1 &0\\
(1) & 0 & 1& 0 &-1 & 0\\
(2)  &0 & 0 & 1& -1& 0\\
(1)(2) & 0 & 0 &0 & 1 &   \Delta\\
(12) &  0 & 0&0 &0 &1
\end{array}
\quad
\begin{array}{c|rrrrr}
\frr_\tau(\sigma)&\varnothing&(1) & (2) &(1)(2) &(12)\\
\hline
\varnothing& 1& 1& 1&1 & - \Delta\\
(1) & 0 & 1& 0 &1 & - \Delta\\
(2)  &0 & 0 & 1& 1& - \Delta\\
(1)(2) & 0 & 0 &0 & 1 &  - \Delta\\
(12) &  0 & 0&0 &0  &1
\end{array}
\]
\noindent
Note in particular that these matrices are inverses of one another (as should be the case since they are matrices of left and right eigenvectors of the same matrix). Without some more sophisticated analysis of the equations \eqref{eq:EV}, adding even $1$ further vertex makes the calculation prohibitively time consuming. \textit{A priori} it is NOT clear that $\frl_\si(\tau), \frr_\tau(\si)$ are rational functions in $\bar{\tbu}, \tbv$.  The fact that they are provides Theorem \ref{T:2} with a somewhat miraculous quality not apparent at first sight. Some further properties are stated in the next theorem. The reader may also consult Section \ref{sec:The_matrix_frN} for a complete account of the analysis, which contains additional structural properties of $\frl_\si(\tau), \frr_\tau(\si)$.

\begin{theorem}\label{thm:T3}
The polynomials $\frp_\si, \frQ_\si$ appearing in Theorem \ref{T:2} are respectively 
\[
	\frp_{\sigma}(\overline{\tbu},\tbv)=\mathsf{V}_{\SV(\si)}(\overline{\tbu})\mathsf{V}_{\SV(\si)}({\tbv})\frr_{\sigma}(\varnothing;\tbu,\tbv)
\]
and
\[
	\frQ_{\sigma}(\overline{\tbu},\tbv)=\mathsf{V}_{\SV(\si)}(\overline{\tbu})\mathsf{V}_{\SV(\si)}({\tbv})\frl_{\texttt{I}_{\SV(\si)}}(\sigma;\tbu,\tbv)\,,
\]  
Moreover, for all $\sigma\in\CS$ and $i,j\in\SV(\si)$ such that $i,j$ belong to a common cycle of $\si$ of length $\geq 3$, both polynomials $\frp_{\sigma}$ and $\frQ_{\sigma}$ vanish whenever $(u_i, v_i)=(u_j, v_j)$. 
\end{theorem}


\section{Rewriting the correlation functions}
 
The first order of business is to recast $\widehat{\rho}_{N, \eps}(\si;\tbu,\tbv)$ in terms of resolvents and contour integrals, which are substantially more amenable to analysis than the expression in \eqref{eq:finite_correlation_function1}. 

\subsection{Upper triangulating $M_N$}
 
While it is NOT possible in general to diagonalize $M_N$ with a unitary transformation, it is possible  to bring $M_N$ to upper triangular form using unitary matrices. Appendix A.33 of \cite{Mehta} presents a useful set of coordinates implementing this change of variables. In summary, \cite{Mehta} shows that one can find a (random) unitary transformation $U$ under which $M_N$ can be represented by an upper triangular matrix $T=T_N$ whose diagonal entries are the eigenvalues of $M_N$ and whose off diagonal entries (above the diagonal) are i.i.d. (also independent of the eigenvalues) complex Gaussian random variables with mean zero and variance $1/N$.

Using the notation $\lambda_j=T_{jj}$ for the eigenvalues, the law for $T$ under $\BP$ is given explicitly by
\begin{equation}\label{E:Tmarg}
	\BP(\textrm{d}T)\propto \prod_{1\leq i<j\leq N}|\lambda_i-\lambda_j|^2 \prod_{j=1}^{N} \kappa_N(\mathrm{d}\lambda_j) \prod_{1\leq i<j\leq N}\kappa_N(\mathrm{d}T_{ij})\,.
\end{equation}

\subsection{Rewriting $\widehat{\rho}_{N,\varepsilon}(\si)$ using resolvents} 

For $z\in\mathbf{D}_1$, let 
\[
	R_N(z)=(T_N-z)^{-1}\,,
\]
be the resolvent of $T_N$ and, for a cyclic permutation $\si\in\CS$ and $\bfz, \bfw \in \textbf{D}_1^{\SV(\si)}$, define
\begin{equation}\label{eq:F_jCL}
	F_N(\si)=F_N(\si;\tbz,\tbw) = \mathrm{Tr}\bigg(\prod_{\alpha\in \SV(\si)}(R_N(z_{\alpha}))^\dag R_j(w_{\alpha})\bigg)\,,
\end{equation}
with the product of the matrices taken according to the orbit of a fixed element of the cycle $\si$. Note that this is well defined due to the invariance of the trace under cyclic permutations. Furthermore, since $T_N$ is obtained from $M_N$ by a unitary transformation $F_{N}(\si;\bfz, \bfw)$ is also the corresponding trace of a product of resolvents for $M_N$. 

We extend the definition of $F_N$ to general permutations $\si\in\CS$ by setting 
\begin{equation}\label{eq:F_jCD}
	F_N(\si)= F_N(\si;\tbz,\tbw) = \prod_{\CL\in\si} F_N(\CL;\tbz|_{\SV(\pi)},\tbw|_{\SV(\pi)}),\qquad \forall \si\in\CS\quad \forall \tbz,\tbw\in\tbD_1^{\SV(\si)}\,,
\end{equation}
where the product is taken over all cyclic permutations in $\si$ and with the convention that $F_N(\emptyset)=1$.

\begin{remark} \label{rem:on_the_defn_of_F_j} In the last product and in many of the products to follow we observe that each of the constituent factors is a trace of product of matrices. Hence there is no need to order the loops of a permutation. On the other hand, each loop imposes a natural ordering on its vertices and the product of matrices appearing in a loops trace is always taken with respect to this ordering. This convention will be enforced throughout the remainder of this paper. 
\end{remark}

Using the spectral decomposition for $M_N$ in terms of $\tbr_i, \tbl_i$,
\[
\begin{aligned}
	F_N(\si) & = \prod_{\pi\in\si} F_n(\pi) = \prod_{\pi\in\si}\mathrm{Tr}\bigg(\prod_{\alpha\in\SV(\pi)}\sum_{i,j=1}^N \frac{Q_i^\dag}{\overline{\lam_i}-\overline{z_\alpha}}\frac{Q_j}{\lam_j-w_\alpha}\bigg)\\
	&= \sum_{I,J\in \CI_\si}\prod_{\pi\in\si}\mathrm{Tr}\bigg(\prod_{\alpha\in\SV(\pi)}\frac{Q_{I_\alpha}^\dag Q_{J_\alpha}}{(\overline{\lam_{I_\alpha}}-\overline{z_\alpha})(\lambda_{J_\alpha}-w_\alpha)}\bigg)\,.
\end{aligned}
\]

Hence, by Cauchy's Theorem
\begin{equation}\label{eq:Relaton_of_rho_and_F}
	\widehat{\rho}_{N, \eps}(\si;\tbu,\tbv)=\frac{1}{(2\pi i)^{2|\SV(\si)|}N^{|\si|}\varepsilon^{4|\SV(\si)|}}\int_{\caS_\varepsilon^{\SV(\si)}(\tbu,\tbv)}  F_{N}(\si;\bfz, \bfw)\textrm{d}\overline{\bfz} \textrm{d}\bfw \,.
\end{equation}
where the contour integrals over $\caS_\varepsilon^{\SV(\si)}(\tbu,\tbv)$ are clockwise over the conjugate variables  $\overline{z}_\alpha$ along the circle of radius $\varepsilon$ and center $u_\alpha$ and counterclockwise over the variables $w_\alpha$ along the circle of radius $\varepsilon$ and center $v_\beta$ for every $\alpha\in\SV(\si)$.

Due to the above presentation for $\widehat{\rho}_{N,\varepsilon}(\si)$, our main object of interest in the following is the quantity $\E[F_{N}(\si;\bfz,\bfw)]$.

\section{A recursive equation via a diagrammatic expansion}\label{S:Rec}

In order to compute the expectation of $\E[F_{N}(\si;\bfz,\bfw)]$, we shall first compute $\E[F_{N}(\si)|\lambda_j: j\in [N]]$, in the terms of the eigenvalues $(\lambda_j)_{j\in [N]}$. This is achieved by successively integrating out the column vectors $\tbv_j=(T_{pj})_{1\leq p< j}$ using Schur's complement formula. The method described below for computing $\E[F_{N}(\si)|\lambda_j: j\in [N]]$ leads to the study conditional expectations of arbitrary correlation functions for the $j$ by $j$ minor corresponding to $\lambda_1,\ldots,\lambda_j$ and $\tbv_1,\ldots,\tbv_j$. 

For $1\leq j\leq N$, let $R_j(z)=(T_j-z)^{-1}$ be the resolvent of $T_j$ and extend the definition of the functions $F_j(\si)$ in an appropriate way. 

Let us discuss the iterative procedure for their computation. To every permutation $\si\in\CS$ one can associate a natural directed graph (digraph) whose vertex set is $\SV(\si)$ and whose edge set, denoted $\CE(\si)$, is the set of ordered vertices $(v,w)\in \SV(\si)\times \SV(\si)$ such that $\si(v)=w$. The resulting digraph is composed of finitely many loops (including loops of length $1$) on the vertex set $\SV(\si)$. Denoting by $\BG$ the set of all such digraphs, one can verify that the mapping $\si\mapsto (\SV(\si),\CE(\si))$ is a bijection from $\CS$ to $\BG$. 

Given a permutation $\si$ and $F\subset \CE(\si)$, define $F^o = \{v ~:~ (v,w)\in F\}$ and $F^i=\{w ~:~ (v,w)\in F\}$. Furthermore, for every pair of permutations $\si,\si'$ set $\CE(\si,\si')=\CE(\si)\cap\CE(\si')$. Finally, for $2\leq j\leq N$, define the sigma-algebra $\CF_j = {\Large \si}(\tbv_k,\lambda_i ~:~ 2\leq k\leq j,\, i\in [N])$ and for $z\in\BC$ and $j\in [N]$ denote 
\[ 
	a_j(z) = (\lambda_j-z)^{-1}\,.
\]

Our main recursive identity computes the action of taking expectation of $F_j$ with respect to $\tbv_j$ as the application of a transfer matrix $A_j$ to $F_{j-1}$. The key point is that $A_j$ only depends on the eigenvalue $\lambda_j$.
\begin{proposition}\label{prop:recurssive_identity} For every $2\leq j\leq N$ and $\si\in\CS$
	\begin{equation}\label{eq:recurrsive_identity_0}
		\BE[F_j(\si)|\CF_{j-1}] =  \sum_{\si'\in\CS}F_{j-1}(\si') A_j(\si',\si) \,,	
	\end{equation}		
	where for any pair of permutations $\si,\si'\in\CS$ such that $\SV(\si')\subset \SV(\si)$
	\begin{align}\label{eq:recurssive_identity_1}
	A_j(\si',\si) &= N^{-\SV(\si')}\prod_{\alpha\in\SV(\si)} a_j(w_\alpha)\overline{a_j(z_\alpha)}\cdot \prod_{e=(e-,e^+)\in\CE(\si,\si')}\Big(1+\frac{N}{\overline{a_j(z_{e^-})}a_j(w_{e^+})}\Big)\nonumber\\
	&=\sum_{F\subset \CE(\si,\si')} N^{-|\SV(\si')|+|F|}\hspace{-10pt}\prod_{\alpha\in \SV(\si)\setminus F^o} a_j(w_\alpha)\prod_{\alpha\in\SV(\si)\setminus F^i}\overline{a_j(z_\alpha)}\,,
\end{align}
and $A_j(\si',\si)=0$ otherwise. 
\end{proposition}


\subsection{Proof of Proposition \ref{prop:recurssive_identity}}

By Schur's complement formula, for $2\leq j\leq N$, we can write 
\[
	R_{j}(z):=\Bigg(\begin{array}{cc}
		R_{j-1}(z) & -(\lambda_j-z)^{-1}	R_{j-1}(z)\tbv_j \\
		0 & (\lambda_j-z)^{-1}\\
	\end{array}
	\Bigg)\,.
\]
Therefore, with the abbreviation $a_j(z)=(\lambda_j-z)^{-1}$, 
\begin{equation}\label{eq:expansion_of_products_of_Green_functions}
\begin{aligned}
	& R_j(z)^\dag R_j(w) \\
	& \qquad = \Bigg(\begin{array}{cc}
		R_{j-1}(z)^\dag R_{j-1}(w)  & -a_j(w)R_{j-1}(z)^\dag R_{j-1}(w)\tbv_j \\
		-\overline{a_j(z)}\tbv_j^\dag R_{j-1}(z)^\dag R_{j-1}(w) & \overline{a_j(z)}a_j(w)+ \overline{a_j(z)}a_j(w)\tbv_j^\dag R_{j-1}(z)^\dag R_{j-1}(w)\tbv_j\\
	\end{array}
	\Bigg)\,.
\end{aligned}
\end{equation}

Recall that for $2\leq j\leq N$, we defined $\CF_j = {\Large \si}(\tbv_k,\lambda_i ~:~ 2\leq k\leq j,\, i\in [N])$ and that $(\tbv_j)_{2\leq j\leq N}$ are independent random vectors whose entries are i.i.d. complex Gaussians with mean zero and variance $1/N$. Since $\tbv_j$ is independent of $R_{j-1}$, computation of $\BE[F_j(\si)|\CF_{j-1}]$ reduces to the expansion of $F_{j}(\si)$ in terms of $(F_{j-1}(\si'))_{\si'\in\CS}$ and the Gaussian vectors $\tbv_j, \tbv_j^{\dag}$ followed by integration over these later. 

We now develop a diagrammatic language to efficiently interpret the meaning and size of the terms appearing in this expansion of $F_j(\si)$ according to \eqref{eq:expansion_of_products_of_Green_functions}. Fix $\pi\in\CS$ and assume it is a cycle. Treating each of the two terms in the lower right block in \eqref{eq:expansion_of_products_of_Green_functions} separately, the term $F_j(\pi)$ can be written as a sum over products of the form $\mathrm{Tr}\big(\prod_{\alpha\in\SV(\pi)} X_{\alpha}\big)$, where 
\[
	X_{\alpha} \in \bigg\{\begin{array}{l} 
R_{j-1}(z_\alpha)^\dag R_{j-1}(w_\alpha),\quad  -a_j(w_\alpha)R_{j-1}(z_\alpha)^\dag R_{j-1}(w_\alpha)\tbv_j,\quad \overline{a_j(z_\alpha)}a_j(w_\alpha) \\
		-\overline{a_j(z_\alpha)}\tbv_j^\dag R_{j-1}(z_\alpha)^\dag R_{j-1}(w_\alpha),\quad  \overline{a_j(z_\alpha)}a_j(w_\alpha)\tbv_j^\dag R_{j-1}(z_\alpha)^\dag R_{j-1}(w_\alpha)\tbv_j	
	\end{array}\bigg\}\,.
\]
To each such product we associate a set of vertices, labeled by elements from $\SV(\pi)$, of five types. The vertices are decorated with zero or two half edges (each of which is either dotted or solid) according to the following rule:
\[
\begin{array}{cccc}
(\RN{1}) &R_{j-1}(z_\alpha)^\dag R_{j-1}(w_\alpha) &\qquad \qquad &  \begin{tikzpicture}[roundnode/.style={circle, draw=black, fill=white, very thick, minimum size=2mm}] 
							\draw[very thick, directed] (0, 0) -- (1, 0);
							\draw[very thick, directed] (1, 0) -- (2, 0);
							\node[roundnode,placec] at (1,0) {$\alpha$};
							\end{tikzpicture}\\
&&&\\
(\RN{2}) & -a_j(w_\alpha)R_{j-1}(z_\alpha)^\dag R_{j-1}(w_\alpha)\tbv_j &\qquad\qquad & 							
\begin{tikzpicture}[roundnode/.style={circle, draw=black, fill=white, very thick, minimum size=2mm}] 
							\draw[very thick, directed] (0, 0) -- (1, 0);
							\draw[very thick, directed, dotted] (1,0) -- (2, 0);
							\node[roundnode,placec] at (1,0) {$\alpha$};
							\end{tikzpicture}\\	
&&&\\
(\RN{3}) & -\overline{a_j(z_\alpha)}\tbv_j^\dag R_{j-1}(z_\alpha)^\dag R_{j-1}(w_\alpha) & \qquad\qquad & 
\begin{tikzpicture}[roundnode/.style={circle, draw=black, fill=white, very thick, minimum size=2mm}] 
							\draw[very thick, directed] (1, 0) -- (2, 0);
							\draw[very thick, directed, dotted] (0,0) -- (1, 0);
							\node[roundnode,placec] at (1,0) {$\alpha$};
							\end{tikzpicture}\\	
&&&\\
(\RN{4}) & \overline{a_j(z_\alpha)}a_j(w_\alpha)\tbv_j^\dag R_{j-1}(z_\alpha)^\dag R_{j-1}(w_\alpha)\tbv_j & \qquad\qquad & 							
\begin{tikzpicture}[roundnode/.style={circle, draw=black, fill=white, very thick, minimum size=2mm}] 
							\draw[very thick, directed,dotted] (0, 0) -- (1, 0);
							\draw[very thick, directed,dotted] (1, 0) -- (2, 0);
							\node[roundnode,placec] at (1,0) {$\alpha$};
							\end{tikzpicture}\\	

&&&\\
(\RN{5}) & \overline{a_j(z_\alpha)}a_j(w_\alpha) &\qquad\qquad & 
\begin{tikzpicture}[roundnode/.style={circle, draw=black, fill=white, very thick, minimum size=2mm}] 						
							\node[roundnode,placer] at (1,0) {$\alpha$};
							\end{tikzpicture}\\	
\end{array}
\]

The nemonic behind these associations is as follows. The vertex $\alpha\in\SV(\pi)$ is decorated with a circle whenever the term $R_{j-1}(z_\alpha)^\dag R_{j-1}(w_\alpha)$  is present in the expansion of the product of $R_j$'s and with a square whenever \emph{only} the factor $\overline{a_j(z_\alpha)}a_j(w_\alpha)$ appears (i.e. case $(\RN{5})$). Additionally a dotted half edge going into the vertex $\alpha$ is associated with the term $-\overline{a_j(z_\alpha)}\tbv_j^\dag$, a dotted half edge going out of the the vertex $\alpha$ is associated with the term $-a_j(w_\alpha)\tbv_j$ and every thick half edge is associated with the term $1$. 

Due to the definition of matrix product and trace, decorated vertices contributing to $\mathrm{Tr}\big(\prod_{\alpha\in \SV(\pi)} R_j^\dag(z_\alpha)R_j(w_\alpha)\big)$ must fit together according to the following rules: 
\begin{enumerate}
\item Vertices of type $(\RN{1})$ or $(\RN{3})$ can only be followed (in the sense of the loop $\pi$) by a vertex of type $(\RN{1})$ or $(\RN{2})$. 
\item Vertices of type $(\RN{1})$ or $(\RN{2})$ can only be ahead (in the loop) of a vertex of type $(\RN{1})$ or $(\RN{3})$. 
\end{enumerate}
This means that a vertex with a solid half-edge going out can only be followed by a vertex with a solid half-edge going in and vice versa. Similarly, 
\begin{enumerate}
  \setcounter{enumi}{2}
\item Vertices of type $(\RN{2}), (\RN{4})$ or $(\RN{5})$ can only be followed (in the loop) by a vertex of type $(\RN{3}), (\RN{4})$ and $(\RN{5})$
\item Vertices of type $(\RN{3}), (\RN{4})$ or $(\RN{5})$ can only be ahead (in the loop) of a vertex of type $(\RN{2}), (\RN{4})$ and $(\RN{5})$
\end{enumerate}
Thus a circularly decorated vertex with a dotted half edge going out can only be followed by a circularly decorated vertex with a dotted half edge going in or by a square vertex. A circular vertex with an incoming dotted half edge  can only appear after a circular vertex with an outgoing dotted half-edge or after a square vertex. A square vertex can only be followed by (appear after) a square vertex or by a circular vertex with a incoming and$\backslash$or outgoing dotted half edge. A diagram associated to a product $\prod_{\alpha\in\SV(\pi)}X_{\alpha}$ which satisfies the four conditions above is called {\it compatible} with the loop $\pi$. 

Going back to general diagrams, given $\si\in \CS$, the product $F_j(\si)=\prod_{\pi\in\si}F_j(\pi)$ can be written as a sum over terms of the form $\prod_{\pi\in\si}\mathrm{Tr}\prod_{\alpha\in\SV(\pi)}X_{\alpha}$. The only terms which contribute to $F_j(\si)$ have the property that for each $\pi\in \si$ the product $\prod_{\alpha\in\SV(\pi)}X_{\alpha}$ corresponds to a diagram compatible with $\pi$. In this case we say that $(\prod_{\alpha\in\SV(\pi)}X_{\alpha})_{\pi\in\si}$ is compatible with $\si$ (or $\si$-compatible) and associate with it the decorated vertices from each of its loops. 

Given a permutation $\si$ and a $\si$-compatible family of products $(\prod_{\alpha\in\SV(\pi)}X_{\alpha})_{\pi\in\si}$, we generate a partial-digraph from the decorated vertices, denoted $\Gamma(\si)$, by gluing every outgoing solid half arrow associated with a given vertex with the incoming solid half arrow associated with its successor vertex in the permutation, where its successor is the vertex obtained by applying $\si$ to its label. See Figure \ref{fig:partial_graph} for an illustration. This is well defined since the vertices of the different loops are disjoint. We denote by $\CG(\si)$ the set of partial graphs obtained from $\si$-compatible products.

\begin{figure}
  \centering
    \includegraphics[width=0.7\textwidth]{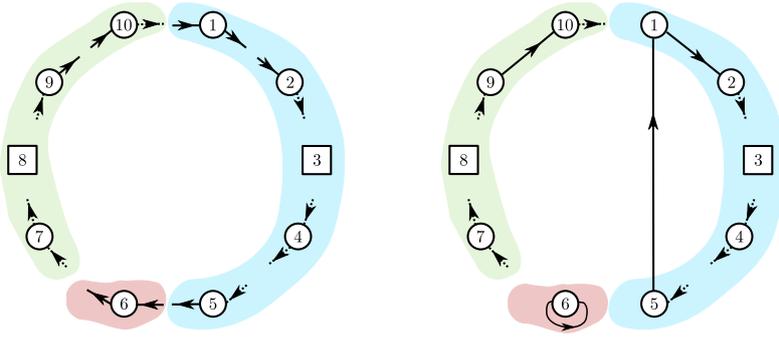}
     \caption{An illustration of a compitable set of vertices (on the left) and the partial graph associated with it (on the right). The original diagram is composed of three loops $\{(1,2,3,4,5),(6),(7,8,9,10)\}$ appearing in the illustration in the colors blue, red and green.\label{fig:partial_graph}}
\end{figure}

Using the fact that compatible products are in bijection with partial digraphs we obtain 
\begin{equation}
	F_j(\si) =  \sum_{\substack{(\prod_{\alpha\in\SV(\pi)}X_{\alpha})_{\pi\in\si}\\ \text{is } \si-\text{compatible}}}\prod_{\pi\in\si}\mathrm{Tr}\big(\prod_{\alpha\in\SV(\pi)}X_{\alpha}\big) = \sum_{\Gamma\in\CG(\si)}\Theta(\Gamma)\,, 
\end{equation}
where for $\Gamma\in\CG(\si)$ we denote by $\Theta(\Gamma)$ the term $\prod_{\pi\in\si}\mathrm{Tr}\big(\prod_{\alpha\in\SV(\pi)}X_{\alpha}\big)$ for the $\si$-compatible product associated with the partial graph $\Gamma$. 

Next we discuss the interpretation, in terms of partial graphs, of integrating over $\tbv_j=(T_{pj})_{1\leq p< j}$. We begin with some elementary combinatorial observations. Let $\Gamma$ be a partial graph associated with the diagram $\si$, and for $\FX\in \{\RN{1},\RN{2},\RN{3},\RN{4},\RN{5}\}$ denote by $\Gamma_\FX$ the set of vertices of type $\FX$ in $\Gamma$. Also, denote $\Gamma_{out} = \Gamma_{\RN{2}}\cup \Gamma_{\RN{4}}$ and $\Gamma_{in} = \Gamma_{\RN{3}}\cup \Gamma_{\RN{4}}$. Note that in a partial graph obtained from a compatible product $|\Gamma_\RN{2}|=|\Gamma_\RN{3}|$ and therefore $|\Gamma_{in}|=|\Gamma_{out}|$. Also, recall that dotted half edges going into (out of) a vertex in $\Gamma_{in}$ ($\Gamma_{out}$) are in correspondence with appearances of $\tbv_j$ and $\tbv_j^\dag$. Since each appearance of $\tbv_j$ or $\tbv_j^\dag$ in the product can be replaced by a sum over $T_{\cdot j}$ (respectively $\overline{T}_{j\cdot}$) and since those are independent complex Gaussian random variables with mean zero and variance $N^{-1}$, any product in which the number of appearances of $T_{p j}$ is not equal to the number of appearances of $\overline{T}_{jp}$ equals zero in expectation. In other words, 
\begin{equation}\label{eq:computing_expectation_1}
	\BE[F_j(\si)|\CF_{j-1}] = \sum_{\Gamma\in\CG(\si)} \sum_{\tbp,\tbq} \BE[\Theta(\Gamma)^{\tbp,\tbq}]\,,
\end{equation}
where the first sum is over partial graphs $\Gamma$, the second sum is over $\tbp:\Gamma_{out}\to [j-1],\tbq:\Gamma_{in}\to[j-1]$ such that $|\tbp^{-1}(i)|=|\tbq^{-1}(i)|$ for every $i\in [j-1]$ and $\Theta(\Gamma)^{\tbp,\tbq}$ is obtained from $\Theta(\Gamma)$ be replacing $v_\alpha$ in the product by $T_{\tbp(\alpha)j}$ and $v^\dag_\alpha$ by $T_{\tbq(\alpha) j}$ whenever they appear in the term related to the vertex $\alpha$. 

Using the fact that $(T_{pj})_{1\leq p<j}$ are independent Gaussian random variables with mean zero and variance $N^{-1}$ gives

\begin{claim}\label{clm:computing_expectation_1}
	Let $\si\in\CS$ and $\Gamma\in\CG(\si)$. Then, 
\begin{equation}\label{eq:computing_expectation_2}
	\sum_{\tbp,\tbq} \BE[\Theta(\Gamma)^{\tbp,\tbq}] = \sum_{f} \sum_{\tbp}\BE[\Theta(\Gamma)^{\tbp,\tbp f^{-1}}]\,,
\end{equation}		
where the first sum on the right hand side is over all bijections $f:\Gamma_{out}\to \Gamma_{in}$. 
\end{claim}

\begin{proof}
	This follows from Wick's theorem for centered, complex normal random variables. 
\end{proof}

We organize the remainder of the proof of Proposition \ref{prop:recurssive_identity} as a series of combinatorial statements, see Claims  \ref{clm:computing_expectation_2} - \ref{clm:Comb_claim_3}.
Fix a permutation $\si\in\CS$. To every pair $(\Gamma,f)$, where $\Gamma\in\CG(\si)$ and $f$ is a bijection from $\Gamma_{out}$ to $\Gamma_{in}$, we associate a digraph $\si_0=\si_0(\Gamma,f)$ obtained from $\Gamma$ by gluing together the dotted half edge going out of the vertex $\alpha$ with the dotted half edge going into $f(\alpha)\in \Gamma_{in}$, for every $\alpha\in \Gamma_{out}$, and then removing all square vertices, namely vertices of type $(\RN{5})$. See Figure \ref{fig:partial_graph_and_bijection} for an illustration. The resulting digraph is always composed of disjoint loops (including loops with one vertex and one edge) and as such corresponds to a unique element in $\CS$.

\begin{figure}
  \centering
    \includegraphics[width=0.7\textwidth]{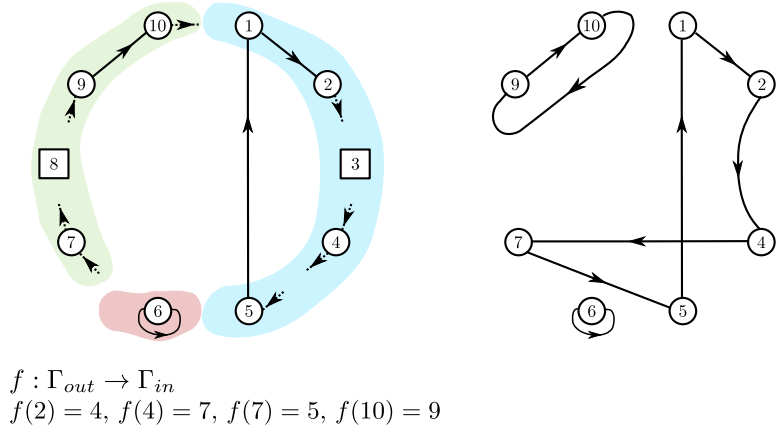}
     \caption{An illustration of a partial digraph and a bijection $\si$ related to it (on the left) together with the digraph associated with it (on the right). \label{fig:partial_graph_and_bijection}}
\end{figure}

\begin{claim}\label{clm:computing_expectation_2}
	Let $\si\in\CS$ be a permutation, $\Gamma\in\CG(\si)$ and $f:\Gamma_{out}\to \Gamma_{in}$ a bijection. Then 
	\begin{equation}
		\sum_{\tbp}\Theta(\Gamma)^{\tbp,\tbp f^{-1}} = N^{-|\Gamma_{out}|}\prod_{\alpha\in \Gamma_{out}\cup \Gamma_{\RN{5}}}a_j(w_\alpha)\prod_{\alpha\in\Gamma_{in}\cup \Gamma_{\RN{5}}}\overline{a_j(z_\alpha)} \cdot F_{j-1}(\si_0(\Gamma,f))\,.
	\end{equation}
\end{claim}

\begin{proof}
	Let us begin by noting that $\sum_{\tbp}\Theta(\Gamma)^{\tbp,\tbp f^{-1}}$ gives us a term of the form $F_{j-1}(\si_0(\Gamma,f))$ which is the corresponding product of traces of resolvents. Furthermore, each pairing of $\tbv_j$ and $\tbv_j^\dag$, which corresponds to a matching in $f$ yields a power of $N^{-1}$.  Since there are $|\Gamma_{out}|=|\Gamma_{in}|$ such matchings,  $F_{j-1}(\si_0(\Gamma,f))$ is multiplied by the total power $N^{-|\Gamma_{out}|}$.  Additionally, there are factors of $a_j$'s originating in the decorated vertices that must be accounted for. Each decorated vertex $\alpha$ of type $(\RN{2})$, $(\RN{4})$ or $(\RN{5})$ yields a factor of $a_j(w_\alpha)$ while each vertex of type $(\RN{3})$, $(\RN{4})$ or $(\RN{5})$ yields a factor of $\overline{a_j(z_\alpha)}$.
\end{proof}

Combining \eqref{eq:computing_expectation_1} together with Claims \ref{clm:computing_expectation_1} and \ref{clm:computing_expectation_2} we conclude 
\begin{align}\label{eq:matrix}
	\BE[F_j(\si)|\CF_{j-1}] &= \sum_{\Gamma\in\CG(\si),f} N^{-|\Gamma_{out}|}\prod_{\alpha\in \Gamma_{out}\cup \Gamma_{\RN{5}}}a_j(w_\alpha)\prod_{\alpha\in\Gamma_{in}\cup \Gamma_{\RN{5}}}\overline{a_j(z_\alpha)} \cdot F_{j-1}(\si_0(\Gamma,f))\nonumber\\
 &=\sum_{\si'\in \CS} F_{j-1}(\si') A_j(\si',\si)\,,
\end{align}
where 
\begin{equation}\label{eq:formula_for_A_j}
	 A_j(\si',\si) = \sum_{\substack{\Gamma\in\CG(\si),f\text{ s.t.}\\ \si_0(\Gamma,f)=\si'}} N^{-|\Gamma_{out}|}\prod_{\alpha\in \Gamma_{out}\cup \Gamma_{\RN{5}}}a_j(w_\alpha)\prod_{\alpha\in\Gamma_{in}\cup \Gamma_{\RN{5}}}\overline{a_j(z_\alpha)}\,.
\end{equation}

Finally, it is left to compare \eqref{eq:formula_for_A_j} with the formula for $A_j(\si',\si)$ appearing in Proposition \ref{prop:recurssive_identity}. 

\begin{claim}
	Let $\si,\si'\in\CS$. There exists $\Gamma\in\CG(\si)$ and a bijection $f:\Gamma_{out}\to \Gamma_{in}$ such that $\si_0(\Gamma,f)=\si'$ if and only if $\SV(\si')\subset\SV(\si)$. Furthermore, assuming $\SV(\si')\subset \SV(\si)$, if $\Gamma\in\CG(\si)$, then there exists a bijection $f:\Gamma_{out}\to\Gamma_{in}$ such that $\si_0(\Gamma,f)=\si'$ if and only if $\Gamma_{\RN{5}}=\SV(\si)\setminus\SV(\si')$. 
\end{claim}	

\begin{proof}	
	If $\alpha\in \SV(\si')\setminus \SV(\si)$, then $\alpha\notin \Gamma$ for every $\Gamma\in\CG(\si)$ which implies that $\alpha\notin \si_0(\Gamma,f)$ and in particular $\si_0(\Gamma,f)\neq \si'$. In the other direction, assume that $\SV(\si')\subset\SV(\si)$. Let $\Gamma$ be the partial graph obtained from $\si$ be choosing all vertices in $\SV(\si)\setminus\SV(\si')$ to be of type $(\RN{5})$ and all remaining vertices to be of type $(\RN{4})$. One can verify that in this case $\Gamma\in \Gamma(\si)$ and $\Gamma_{out}=\Gamma_{in}=\SV(\si')$. Choosing $f$ to be the bijection induced from the cycle structure of $\si'$ one obtains $\si_0(\Gamma,f)=\si'$ as required. 
	
	Furthermore, the vertices in the permutation $\si_0(\Gamma,f)$ are exactly $\SV(\si)\setminus \Gamma_{\RN{5}}$ and therefore, given two permutations $\si,\si'$ such that $\SV(\si')\subset\SV(\si)$, a partial graph $\Gamma\in\CG(\si)$ can satisfy $\si_0(\Gamma,f)=\si'$ for some bijection $f$ only if and $\SV(\si')\subset \SV(\si)$ and $\Gamma_{\RN{5}} = \SV(\si)\setminus \SV(\si')$. 
\end{proof}

\begin{claim}\label{clm:Comb_claim_3}
	Let $\si,\si'\in\CS$ be two permutations such that $\SV(\si')\subset\SV(\si)$ and let $\Gamma\in\CG(\si)$ such that $\Gamma_{\RN{5}}=\SV(\si)\setminus\SV(\si')$. There exists a bijection $f:\Gamma_{out}\to\Gamma_{in}$ such that $\si_0(\Gamma,f)=\si'$ if and only if the set of directed edges in $\Gamma$ is contained in $\CE(\si,\si')$. Furthermore, if the last condition holds, then such a bijection is unique. 
\end{claim}

\begin{proof}
	Denote by $\CE(\Gamma)$ the set of oriented edges in the partial graph $\Gamma$. From the definition of partial graphs associated with a diagram we know that $\CE(\Gamma)\subset\CE(\si)$ for every $\Gamma\in\CG(\si)$. Furthermore, from the definition of the digraph $\si_0(\Gamma,f)$ obtained from $\Gamma$ with the help of the bijection $f$, it follows that $\CE(\Gamma)\subset\CE(\si')$. Hence the condition is necessary. 
	
	Conversely, assume next that $\Gamma\in\CG(\si)$ and that $\CE(\Gamma)\subset\CE(\si,\si')$. Since $\si'$ is composed of disjoint oriented loops and since $\Gamma$ can be thought of as an oriented graph on the same vertex set as $\si'$ with $\CE(\Gamma)\subset\CE(\si')$, there can be at most one way to complete the set of oriented edges of $\Gamma$ in order to create $\si'$, that is, by adding the set of edges $\CE(\si')\setminus\Gamma$. This can be done using the bijection $f:\Gamma_{out}\to\Gamma_{in}$ defined by $f(\alpha)=\beta$, where for $\alpha\in\Gamma_{out}$ we define $\beta$ to be the unique vertex in $\SV(\si')$ such that that $(\alpha,\beta)\in\CE(\si')$. 	
\end{proof}

Combining the last two claims with \eqref{eq:formula_for_A_j} we obtain for $\si,\si'\in\CS$ such that $\SV(\si')\subset\SV(\si)$
\begin{equation}
	A_j(\si',\si) = \hspace{-10pt}\prod_{\alpha\in \SV(\si)\setminus\SV(\si')} \hspace{-15pt}a_j(w_\alpha)\overline{a_j(z_\alpha)} 
		\sum_{F\subset \CE(\si,\si')} N^{-|\SV(\si')|+|F|}\hspace{-5pt}\prod_{\alpha\in \SV(\si')\setminus F^o} a_j(w_\alpha)\prod_{\alpha\in\SV(v')\setminus F^i}\overline{a_j(z_\alpha)}
		\,.
\end{equation}
and $A_j(\si',v)=0$ whenever $\SV(\si')\setminus\SV(\si)\neq \emptyset$. This completes the proof of Proposition \ref{prop:recurssive_identity}.
\qed


\subsection{From the recursive equation to the conditional expectation}
\label{S:CE}
Going back to Proposition \ref{prop:recurssive_identity}, for every fixed choice of $\tbz$ and $\tbw$, we may interpret $F_{j-1}(\si)$ and hence $\BE[F_j(\si)|\CF_{j-1}]$ as random elements of $\BC^{\CS}$. Naturally, we can then view $A_j(\cdot,\cdot)$ as a linear  operator on $\bbC^{\CS}$. In particular, \eqref{eq:recurrsive_identity_0} may then be written in a matrix form $\BE[F_j(\cdot)|\CF_{j-1}]=F_{j-1}A_j(\cdot)$.  Let us observe that the matrix $A_j$ is the specialization, at the value $\lambda_j$, of the matrix valued function $A^{\lambda}:=(A(\si', \si;\lambda,\tbz,\tbw))_{\si,\si\in\CS}$ given by 
\begin{align}\label{eq:A1} 
A(\si', \si;\lambda,\tbz,\tbw) =& 
	\prod_{\alpha\in \SV(\si)\setminus\SV(\si')}(\lambda-w_\alpha)^{-1}(\overline{\lambda}-\overline{z_\alpha})^{-1} \nonumber \\
	&	\times\sum_{F\subset \CE(\si,\si')} N^{-|\SV(\si')|+|F|}\prod_{\alpha\in \SV(\si')\setminus F^o } (\lambda-w_\alpha)^{-1}\prod_{\alpha\in\SV(\si')\setminus F^i}(\overline{\lambda}-\overline{z_\alpha})^{-1}\,,
\end{align}
and by Proposition \ref{prop:recurssive_identity}
\[
	\BE[F_j(\cdot)|\CF_{j-1}]=F_{j-1}A^{\lambda_j}(\cdot)\,.
\]

Repeating the induction procedure and using the tower property of conditional expectations, we conclude that for every $2\leq j\leq N$
\[
	\BE[F_N(\cdot)|\CF_{j-1}] = F_{j-1}A^{\lambda_N}A^{\lambda_{N-1}}\ldots A^{\lambda_j}(\cdot)\,,
\]
and in particular 
\[
	\BE[F_N(\cdot)|\lambda_i ~:~i\in [N]]=\BE[F_N(\cdot)|\CF_{1}] = F_{1}A^{\lambda_2}A^{\lambda_3}\ldots A^{\lambda_N}(\cdot)\,.
\]
Note that once we get down to $F_1$, all resolvent matrices are given by the appropriate scalars $\overline{a_1(z)},a_1(w)$.  Therefore, it is natural to interpret $F_1$ as 
\[
	F_1(\si) = \prod_{\alpha\in\SV(\si)}(\lambda_1-w_\alpha)^{-1}(\overline{\lambda}-\overline{z_\alpha})^{-1} = A^{\lambda_1}(\emptyset,\si).
\]
We therefore obtain the compact formula
\begin{equation}\label{eq:cond_expec_for_F_N}
	\BE[F_N(\si)|\lambda_i ~:~i\in [N]] = \tbe_\emptyset^\dag A^{\lambda_1}A^{\lambda_2}A^{\lambda_3}\ldots A^{\lambda_N}\tbe_{\si},\qquad \forall \si\in\CS\,,
\end{equation}
where $\emptyset$ denotes the empty diagram and $(\tbe_{\si})_{\si\in\CS}$ is the standard basis of column vectors, i.e., $\tbe_\si(\si') = 1$ if $\si'=\si$ and is $0$ otherwise.

\section{Integrating over the eigenvalues}
\label{S:M}\label{S:A}

We now have (at least in principal) expressed all correlation functions of interest in terms of the eigenvalues $(\lambda_i)_{i\in[N]}$. The next step is therefore to take the expectation with respect to the eigenvalues. By \eqref{eq:cond_expec_for_F_N}, we can write this as 
\[
	\BE[F_N(\si)] = \tbe_\emptyset^\dg \BE[A^{\lambda_1}A^{\lambda_2}\ldots A^{\lambda_N}]\tbe_\si\,.
\]

Recall that our main goal is to prove the existence of the limit $\lim_{N\to\infty} \allowbreak\BE[\widehat{\rho}_{N,\varepsilon}(\si)]$. Due to the relation between $\E[\widehat{\rho}_{N,\varepsilon}(\si)]$ and $F_{N}(\si)$ given by \eqref{eq:Relaton_of_rho_and_F}, we wish to take the limit of the normalized function $N^{-|\si|} \BE[F_{N}(\si)]$ as $N$ tends to infinity. 

Let us introduce a normalized version of $A^{\lambda}$ to absorb the factor $N^{-|\si|}$. Let $\Lambda$ be the diagonal $\CS\times\CS$ matrix given by $\Lambda_{\si,\tau} = N^{|\si|}\ind_{\si=\tau}$ and define the matrix $B^{\lambda} = \Lambda A^{\lambda} \Lambda^{-1}$. In other words,
\begin{align}\label{eq:The_matrix_B}
&B^{\lambda}(\si, \tau) = 
	\prod_{\alpha\in \SV(\tau)\setminus\SV(\si)}(\lambda-w_\alpha)^{-1}(\overline{\lambda}-\overline{z_\alpha})^{-1} \nonumber \\
	&\qquad	\cdot\sum_{F\subset \CE(\si,\tau)} N^{-|\SV(\si)|+|F|-|\tau|+|\si|}\prod_{\alpha\in \SV(\si)\setminus F^o} (\lambda-w_\alpha)^{-1}\prod_{\alpha\in\SV(\si)\setminus F^i}(\overline{\lambda}-\overline{z_\alpha})^{-1}\,,
\end{align}
whenever $\SV(\si)\subset \SV(\tau)$ and $0$ otherwise. 

From the definition of $B^{\lambda}$ and the fact that the empty diagram does not contain any loops, we obtain 
\begin{equation}\label{eq:F_in_terms_of_B}
	N^{-|\si|}\BE[F_N(\si)] = \tbe_\emptyset^\dg \BE[\Lambda A^{\lambda_1}A^{\lambda_2}\ldots A^{\lambda_N}\Lambda^{-1}]\tbe_\si
	= \tbe_\emptyset^\dg \BE[B^{\lambda_1}B^{\lambda_2}\ldots B^{\lambda_N}]\tbe_\si\,.
\end{equation}
Hence, it suffices to understand the limit as $N$ tends to infinity of the entries of the matrix $\BE[B^{\lambda_1},\ldots,B^{\lambda_N}]$.

Recall the definition of the distance between vectors of complex numbers.
\begin{align*}
	& \mathrm{Dist}(\tbz,\tbw) = \mathrm{Dist}_\ell(\tbz,\tbw;\partial\tbD_1)\\
	&\qquad =  \min_{\alpha,\beta\in [\ell]} \{|z_\alpha-w_\beta|\}  \wedge \min_{\alpha,\beta\in [\ell],~\alpha\neq \beta}\{|z_\alpha-z_\beta|,|w_\alpha-w_\beta|\}\wedge \min_{\alpha\in [\ell]}\{1-|z_\alpha|, 1-|w_\alpha|\}\,. 
\end{align*}
Throughout the remainder of the paper we use the notation $\sum'$ in order to denote standard summation with the exception that the empty sum equals $1$. 

\begin{theorem}[The matrix $\FN$]\label{thm:limit_exists}
	Fix $\ell\in\BN$ and $\tbz,\tbw\in\BC^\ell$ such that $\mathrm{Dist}(\tbz,\tbw)>0$. Then, for every $\si,\tau\in \CS_\ell :=\biguplus_{A\subset[\ell]}\CS_A$ the sequence $\tbe_\si\BE[B^{\lambda_1}B^{\lambda_2}\ldots B^{\lambda_N}]\tbe_\tau$ converges as $N$ goes to infinity to $\tbe_\si\exp(\FN_\ell(\tbz,\tbw))\tbe_\tau$, where $\FN_\ell$ is the restriction of the matrix $\FN$, see \eqref{eq:THE_MATRIX_N_WELL}, to entries in $\CS_\ell$. In particular, for every fixed $\tbz,\tbw\in\tbD_1^\ell$ such that $\mathrm{Dist}(\tbz,\tbw)>0$ and every $\si\in\CS_\ell$
	\[ 
		\lim_{N\to\infty} N^{-|\si|}\BE[F_N(\si)] = \tbe_\emptyset^\dag \exp(\FN_\ell) \tbe_\si\,.
	\]
The matrix $\FN$ considered above is defined by 
\begin{equation}\label{eq:THE_MATRIX_N_WELL}
\begin{aligned}
		\FN_{\si,\tau} = &\ind_{\SV(\si)\subset\SV(\tau)} \hspace{-20pt}
		\sum_{\substack{F\subset\CE(\si,\tau) \\ |F| = |\SV(\si)|+|\tau|-|\si|-1}}\\
		&\quad 	\sideset{}{'}\sum_{\substack{\alpha\in \SV(\si)\setminus F^{o} \\ \beta\in \SV(\si)\setminus F^{i}}} \bigg(\prod_{\substack{\gamma\in \SV(\tau)\setminus F^o \\ \gamma\neq \alpha}}(w_\alpha-w_\gamma)^{-1} \prod_{\substack{\delta \in \SV(\tau)\setminus F^i \\ \delta\neq \beta}}(\overline{z}_\beta-\overline{z}_\delta)^{-1}\bigg) h(w_\alpha,z_\beta)\,,
\end{aligned}
\end{equation}
	where for $z,w\in\tbD_1$ we define $h(z,w)=\log\big(\frac{1-\overline{z}w}{|z-w|^2}\big)$.
\end{theorem}

\begin{remark*}
	In Section \ref{sec:The_matrix_frN} we will show that the above formula for $\FN$ coincides with the one provided in \eqref{the_function_Fn}-\eqref{eq:The_Matrix_N}.
\end{remark*}

The following weaker notion of distance will be used throughout this section. For $\tbz,\tbw\in \tbD_1^{\ell}$ let 
\[
	\mathrm{dist}(\tbz,\tbw)= \mathrm{dist}_\ell(\tbz,\tbw;\partial\tbD_1)=  \min_{\alpha\in [\ell]} \{|z_\alpha-w_\alpha|, 1-|z_\alpha|, 1-|w_\alpha|)\} \,.
\]

\subsection{Key proposition}

Fix $\ell\in\BN$. Throughout the remainder of this section we make the dependence on $\ell$ implicit. In particular,  with a slight abuse of notation we use $B^{\lambda}$ also to denote the restriction of $B^{\lambda}=\Lambda A^\lambda \Lambda^{-1}$ to $\CS_\ell\times\CS_\ell$. We organize the matrix elements of $B^{\lambda}$ according to powers of $N$, writing
\[
	B^{\lambda} = \sum_{k\in \BZ} N^{-k} B^{\lambda,k}\,,
\]	
where the matrices $B^{\lambda,k}$ depend only on $\tbz,\tbw$, $\ell$ and $\lambda$ and in particular not on $N$. 

Our first observation is that  $B^{\lambda,k}=0$ for $k<0$ and $k>2\ell$. 

\begin{claim}\label{clm:power_of_N}
	For every $\si,\tau\in\CS_\ell$ such that $\si\neq \tau$ and $\SV(\si)\subset\SV(\tau)$,
\[
		|\CE(\si,\tau)|\leq |\SV(\si)| + |\tau| - |\si|-1\,.
\] 
In particular, $B^{\lambda,k}=0$ for every $k<0$ and $k>2\ell$. 
\end{claim}

\begin{proof}
	Since each permutation consists of disjoint collection of loops and since the number of (directed) edges in each loop equals the number of vertices in it, it follows that 
\begin{equation}\label{eq:edges_verices}
	|\CE(\si)|=|\SV(\si)|\,.
\end{equation}

We split the proof into two cases. Assume first that $|\si|< |\tau|$. Using \eqref{eq:edges_verices}, we obtain $|\CE(\si,\tau)| < |\CE(\si)| = |\SV(\si)|\leq |\SV(\si)|+|\tau|-|\si|$, as required. 
	
Assume next that $|\tau|\leq|\si|$ and $\si\neq \tau$. Once again, due to \eqref{eq:edges_verices} it suffices to prove that $|\CE(\si,\tau)|\leq |\CE(\si)| -(|\si|-|\tau|)-1$. To see this, observe that $\SV(\si)\subset\SV(\tau)$ and therefore each loop in $\si$ which is not part of $\tau$ must contain at least two directed edges that do not belong to $\CE(\tau)$. Since the loops are disjoint so are the edges, which implies that there are at least $|\si|-|\tau|+1$ oriented edges in $\CE(\si)$ that do not belong to $\CE(\tau)$, that is $|\CE(\si,\tau)|\leq |\CE(\si)|-(|\si|-|\tau|)-1$ as required. 

Finally, note that for every $F\subset\CE(\si,\tau)$ we have $-|\SV(\si)|+|F|-|\si|+|\tau|\geq -\ell + 0 -\ell +0 = -2\ell$, which implies that $B^{\lambda,k}=0$ for $k>2\ell$. Similarly, $-|\SV(\si)|+|F|-|\si|+|\tau|\leq 0$ and therefore $B^{\lambda,k}=0$ for $k<0$
\end{proof}

The second observation reads
\begin{claim}
	$B^{\lambda,0} = \mathrm{Id}$ for every $\lambda\in\BC$. 
\end{claim}

\begin{proof}
	For the diagonal entries, this follows from the fact that $B^{\lambda}_{\si,\si}=A^{\pmb{\lambda}}_{\si,\si}$ for every $\si\in \CS_\ell$ together with the fact that $A^{\lambda}_{\si,\si}=1$ for every $\si\in\CS_\ell$ as can be seen by examining the explicit expression \eqref{eq:recurssive_identity_1}. Indeed, the possible powers of $N$ in \eqref{eq:recurssive_identity_1} for $\si=\si'$ are $-|\SV(\si)|+|F|\leq -|\SV(\si)|+|\CE(\si)|=0$ for some $F\subset \CE(\si)$ with equality if and only if $F=\CE(\si)$. For the off-diagonal entries this follows from Claim \ref{clm:power_of_N}.
\end{proof}

Using the last two claims we can rewrite the matrix $B^{\lambda}$ as 
\begin{equation}\label{eq:B_to_F}
	B^{\lambda} = \mathrm{Id} + X^{\lambda}\,,
\end{equation}
where $X^{\lambda} = \sum_{k=1}^{2\ell} N^{-k}B^{\lambda,k}$. 

Let us observe that by virtue of the symmetry of \eqref{E:Tmarg}, the product $\E[X^{\lambda_{j_1}}X^{\lambda_{j_2}}\ldots X^{\lambda_{j_m}}]$ depends only on the cardinality $m$ and not on the specific choice of $j_1,\ldots,j_m\in [N]$ as long as $j_1,\ldots,j_m$ are distinct. Therefore,
\begin{equation}\label{eq:Psi}
	\BE[B^{\lambda_1}B^{\lambda_2}\ldots B^{\lambda_N}] = \BE[(\mathrm{Id}+X^{\lambda_1})(\mathrm{Id}+X^{\lambda_2})\ldots (\mathrm{Id}+X^{\lambda_N})]
	 =\sum_{m=0}^N    {N \choose m} \Psi(m)\,,
\end{equation}
where for $0\leq m\leq N$ we introduced the matrix 
\[
	\Psi(m)=\E[X^{\la_1}\dotsc X^{\la_m}]\,.
\]
 
Combining \eqref{eq:F_in_terms_of_B} with \eqref{eq:Psi} we conclude that Theorem \ref{thm:limit_exists} amounts to two estimates. The first evaluates $\Psi(m)$ for a fixed $m$ as $N$ goes to infinity, while the second allows us to conclude that the limit exists as the sum of the term-by-term limits. Both estimations are summarized in the following lemma:

\begin{proposition}\label{L:EstMain} Fix $\ell\in\BN$, $\delta\in (0,1/2)$ and $\tbz,\tbw\in\tbD_1^\ell$ such that $\mathrm{Dist}(\tbz,\tbw)>0$. 
\begin{enumerate}
\item[(1)] For every fixed $m\in \BN$, there exists a constant $\widehat{C}=\widehat{C}(m,\mathrm{Dist}(\tbz,\tbw),\ell,\delta)\in (0,\infty)$ such that 
	\begin{equation}\label{eq:Term:ii}
		\bigg\|{N \choose m}\Psi(m)-\frac{\frN^m}{m!}\bigg\|\leq \frac{\widehat{C}}{N^{1/2-\delta}}\,.
	\end{equation}
\item[(2)] There exists a universal constant $C\in (0,\infty)$ such that for all $m\in \BN$.
	\begin{equation}\label{eq:Term:iii}
		\bigg\|{N \choose m}\Psi(m)\bigg\| \leq \Big(\frac{C}{\mathrm{Dist}(\tbz,\tbw)^{2\ell+2} m}\Big)^m\,.
\end{equation}
\end{enumerate}
\end{proposition}

\begin{proof}[Proof of Theorem \ref{thm:limit_exists} assuming Proposition \ref{L:EstMain}]
	This is a standard $\varepsilon/3$ argument. 
\end{proof}

\vspace{0.5cm}

\subsection{Proof of Proposition \ref{L:EstMain}}

Fix $m\in\BN$. We start with the proof of \eqref{eq:Term:ii}. Using the explicit expression for the entries of the matrix $B^{\lambda}$ and its relation to the matrix $F^{\lambda}$ (see \eqref{eq:The_matrix_B} and \eqref{eq:B_to_F}), for every $m\in\BN$ and $\si,\tau\in \CS_\ell$
\begin{align}\label{eq:proof_of_Prop_eq1}
	\tbe_\si\binom{N}{m}\Psi(m)\tbe_\tau & = \sum_{\substack{\si_i\in\CS_\ell \\ i=1,\ldots,m-1}} \sum_{\substack{F_i\subset \CE(\si_{i-1},\si_i)\\ i=1,\ldots,m-1}}  \binom{N}{m}N^{-\sum_{i=1}^{m}(|\SV(\si_i)|-|F_i|+|\si_{i-1}|-|\si_i|)}\nonumber \\
	& \qquad \quad\times \BE\Big[\prod_{i=1}^m \prod_{\alpha\in \SV(\si_i)\setminus F_i^o}(\lambda_i-w_\alpha)^{-1} \prod_{\beta\in\SV(\si_i) \setminus F_i^i}(\overline{\lambda_i}-\overline{z_\beta})^{-1}\Big]\,,
\end{align}
where we used the notation $\si_0=\tau$ and $\si_m=\si$. 

Recall that we use $\sum'_{\alpha\in I}$ to denote summation over $\alpha\in I$ which is taken to be $1$ if $I$ is empty. 
\begin{lemma}[Partial Fraction Expansion]\label{lem:pfe}
	For any finite index set $\caI$ and any pairwise distinct $|\caI|$-tuple $(w_\alpha)_{\alpha\in \caI}$
\[
	\prod_{\alpha \in \caI }(\lambda-w_{\alpha})^{-1}=\sideset{}{'}\sum_{\alpha \in \caI} (\lambda-w_{\alpha})^{-1} \prod_{\CI\ni\beta\neq \alpha}(w_{\alpha}-w_{\beta})^{-1}\,.
\]
\end{lemma}

For $m\in\BN$ and $\tbw',\tbz'\in\tbD_1^m$, let 
\begin{equation}
	\psi_m (\tbz',\tbw') = \BE\bigg[\prod_{i=1}^{m}(\lambda_i-w'_i)^{-1}(\overline{\lambda}_i-\overline{z'}_i)^{-1}\bigg]\,.
\end{equation}
Assuming that $\mathrm{Dist}(\tbz,\tbw)>0$ and using the partial fraction expansion, we can rewrite \eqref{eq:proof_of_Prop_eq1} as
\begin{align}\label{eq:proof_of_Prop_eq2}
	\tbe_\si\binom{N}{m}\Psi(m)\tbe_\tau  &= \sum_{\substack{\si_i\in\CS_\ell \\ i=1,\ldots,m-1}} \sum_{\substack{F_i\subset \CE(\si_{i-1},\si_i)\\ i=1,\ldots,m-1}}\sideset{}{'}\sum_{\substack{ \alpha_i\in \SV(\si_i)\setminus F_i^o \\ \beta_i\in\SV(\si_i)\setminus F_i^i\\ i=1,\ldots,m-1}} 
	\CR(\tbz,\tbw,\{\si_i,F_i,\alpha_i,\beta_i\}_{i=1}^{m-1},\si,\si)\nonumber\\
	&\qquad\qquad\quad \times \binom{N}{m}N^{-\sum_{i=1}^{m}(|\SV(\si_i)|-|F_i|+|\si_{i-1}|-|\si_i|)}\psi_m(\tbz^{\pmb{\alpha}},\tbw^{\pmb{\beta}})\,,
\end{align}
where for $\pmb{\alpha}=(\alpha_1,\ldots,\alpha_m)$ and $\pmb{\beta}=(\beta_1,\ldots,\beta_m)$, we define $z^{\pmb{\alpha}}_i = z_{\alpha_i}$ ($w^{\pmb{\beta}}_i = w_{\beta_i}$) and 
\begin{align}
		&\CR(\tbz,\tbw,\{\si_i,F_i,\alpha_i,\beta_i\}_{i=1}^{m-1},\si,\tau) \nonumber\\
		&\quad \qquad =	\prod_{i=1}^m\ind_{\SV(\si_i)\subset \SV(\si_{i-1})}\hspace{-15pt}\prod_{\substack{\gamma\in \SV(\si_i)\setminus F_i^o \\ \gamma\neq \alpha_i}}\hspace{-5pt}(w_{\alpha_i}-w_\gamma)^{-1}\hspace{-5pt}\prod_{\substack{\delta\in \SV(\si_i)\setminus F_i^i \\ \delta\neq \beta_i}}(\overline{z_{\beta_i}}-\overline{z_\delta})^{-1}\,.
\end{align}

The following lemma summarizes the required bounds on $\psi_m$ needed to control $\Psi(m)$. Recall that for $z,w\in\tbD_1$, we denote $h(w,z)=\log\big(\frac{1-\bar{z}w}{|z-w|^2}\big)$.
\begin{lemma}\label{lem:bound_on_psi_m} Fix $m\in \BN$. For every $\delta>o$, there exists a constant $C\in (0,\infty)$ such that for every $\tbz',\tbw'\in\tbD_1^m$
\begin{equation}\label{eq:psi_bound_i}
	\bigg|\binom{N}{m}N^{-m}\psi_m(\tbz',\tbw') - \frac{1}{m!}\prod_{i=1}^m h(w'_i,z'_i)\bigg| \leq \frac{C}{\mathrm{dist}^{2m}(\tbz',\tbw')}N^{\delta-\frac12}\,.
\end{equation}
In addition, there exists a universal constant $C\in (0,\infty)$ such that for every $m\in\BN$
\begin{equation}\label{eq:psi_bound_ii}
	\bigg|\binom{N}{m}N^{-m}\psi_m(\tbz',\tbw')\bigg|\leq \bigg(\frac{C}{m\cdot \mathrm{dist}^2(\tbz',\tbw')}\bigg)^m\,.
\end{equation}
\end{lemma}

We postpone the proof of Lemma \ref{lem:bound_on_psi_m} and return to complete the proof of Proposition \ref{L:EstMain}. Fix some $\delta\in(0,1/2)$. By Claim \ref{clm:power_of_N}, the term $\sum_{i=1}^{m}(|\SV(\si_i)|-|F_i|+|\si_{i-1}|-|\si_i|)$ in \eqref{eq:proof_of_Prop_eq2} is at least $m$. Combining \eqref{eq:psi_bound_i} with the fact that $\mathrm{dist}(\tbz^{\pmb{\alpha}},\tbw^{\pmb{\beta}})\geq \mathrm{Dist}(\tbz,\tbw)>0$ for every choice of indexes $\pmb{\alpha},\pmb{\beta}$ we conclude that all summands in \eqref{eq:proof_of_Prop_eq2} satisfying $\sum_{i=1}^{m}(|\SV(\si_i)|-|F_i|+|\si_{i-1}|-|\si_i|)>m$ are of order $C(m,\mathrm{Dist}(\tbz,\tbw))/N$, while all summands in \eqref{eq:proof_of_Prop_eq2} satisfying $\sum_{i=1}^{m}(|\SV(\si_i)|-|F_i|+|\si_{i-1}|-|\si_i|)=m$, are at distance at most $\widehat{C}(m,\mathrm{Dist}(\tbz,\tbw))N^{\delta-1/2}$ from 
\[
	\CR(\tbz,\tbw,\{\si_i,F_i,\alpha_i,\beta_i\}_{i=1}^{m-1},\si,\tau)\frac{1}{m!}\prod_{i=1}^{m}h(z_{\beta_i},w_{\alpha_i})\,,
\]
for appropriate choices of $\{\si_i,F_i,\alpha_i,\beta_i\}_{i=1}^m$. Since the number of summands is finite we conclude that 
\begin{align}
 	& \bigg|\tbe_\si\binom{N}{m}\Psi(m)\tbe_\tau-\frac{1}{m!}(\FN^m)_{\si,\tau}\bigg|
	 =	\bigg|\binom{N}{m}\Psi(m) - \frac{1}{m!}\sum_{\substack{\si_i\in\CS_\ell \\ i=1,\ldots,m-1}} \sum_{\substack{F_i\subset \CE(\si_{i-1},\si_i)\\ |F_i| = |\SV(\si_i)| + |\si_{i-1}|-|\si_i|-1 \\ i=1,\ldots,m-1}}\nonumber\\
	 &\qquad \sum_{\substack{ \alpha_i\in \SV(\si_i)\setminus F_i^o \\ \beta_i\in\SV(\si_i)\setminus F_i^i\\ i=1,\ldots,m-1}} \CR(\tbz,\tbw,\{\si_i,F_i,\alpha_i,\beta_i\}_{i=1}^{m-1},\si,\tau)\frac{1}{m!}\prod_{i=1}^{m}h(z_{\beta_i},w_{\alpha_i})\bigg| \leq \widehat{C}N^{\delta-1/2}\,,
\end{align}
as required. 

Next, we turn to prove \eqref{eq:Term:iii}. Using again the fact that the sum in \eqref{eq:proof_of_Prop_eq2} is finite, it suffices to prove the result for each of the summands separately. Fix $\si=\si_m,\tau=\si_0,\si_i\in\CS_\ell$, $F_i\subset \CE(\si_{i-1},\si_i)$ and $\alpha_i\in\SV(\si_i)\setminus F_i^o$, $\beta_i\in\SV(\si_i)\setminus F_i^i$ for $i=1,\ldots,m$. 
From the definition of $\mathrm{Dist}$,	for every $\tbz,\tbw\in\tbD_1^\ell$ such that $\mathrm{Dist}(\tbz,\tbw)>0$, we have 
\[
	|\CR(\tbz,\tbw,\{\si_i,F_i,\alpha_i,\beta_i\}_{i=1}^{m-1},\si,\tau)|\leq \frac{1}{\mathrm{Dist}(\tbz,\tbw)^{2\ell m}}\,,
\]
which together with \eqref{eq:psi_bound_ii} proves that for every $m\in \BN$, the summand under consideration is bounded by 
\begin{align*}
	&\Big|\CR(\tbz,\tbw,\{\si_i,F_i,\alpha_i,\beta_i\}_{i=1}^{m-1},\si,\tau)\binom{N}{m}N^{-\sum_{i=1}^{m}(|\SV(\si_i)|-|F_i|+|\si_{i-1}|-|\si_i|)}\psi_m(\tbz^{\pmb{\alpha}},\tbw^{\pmb{\beta}})\Big|\\
	\leq &\Big|\frac{1}{\mathrm{Dist}(\tbz,\tbw)^{2\ell m}}\binom{N}{m}N^{-m}\psi_m(\tbz^{\pmb{\alpha}},\tbw^{\pmb{\beta}})\Big| \leq \Big(\frac{C}{\mathrm{Dist}(\tbz,\tbw)^{2\ell+2}m}\Big)^m\,,
\end{align*}
as required. \hfill \qed

\vspace{0.5cm}
\subsection{Proof of Lemma \ref{lem:bound_on_psi_m}}

In order to prove Lemma \ref{lem:bound_on_psi_m} we need to estimate the function $\psi_m$ defined by 
\[
	\psi_m (\tbz',\tbw') = \BE\bigg[\prod_{i=1}^{m}(\lambda_i-w'_i)^{-1}(\overline{\lambda}_i-\overline{z}'_i)^{-1}\bigg]\,.
\]
Recall that the density of the eigenvalues (see \eqref{E:Tmarg}) is given by the kernel
\begin{equation}\label{eq:eigenvalues_kernel}
	\frac{1}{Z_N}\prod_{1\leq i<j\leq N}|\lambda_i-\lambda_j|^2 \prod_{i\in [N]} \kappa_N(\textrm{d}\lambda_i)\,,
\end{equation}
where $Z_N$ is the Selberg integral 
\[
	Z_N = \int_{\BC^N}  \prod_{1\leq i<j\leq N}|\lambda_i-\lambda_j|^2 \prod_{i\in [N]} \kappa_N(\textrm{d}\lambda_i)\,.
\]

We start by rewriting the kernel using a matrix determinant. For $\lambda\in \C$, let $\tbp^{\lambda}$ be the column vector in $\C^N$ with components 
\[
	p^{\lambda}_j=N^{(j-1)/2}\lambda^{j-1}/\sqrt{(j-1)!}
\] 
and let ${\tbp^{\lambda}}^{\dg}$ denote its conjugate transpose row vector. Define the $N\times N$ matrix $K$ by $K_{i, j}=  {\tbp^{\lambda_i}}^\dg\cdot \tbp^{\lambda_j}$ for $i,j\in [N]$. Denoting by $V^{\pmb{\lambda}}$ be the Vandermonde matrix 
\[
	V^{\pmb{\lambda}} = \left(\begin{array}{ccccc}
	 1 & \lambda_1 & \lambda_1^2 & \ldots & \lambda_1^{N-1}\\
	 1 & \lambda_2 & \lambda_2^2 & \ldots & \lambda_2^{N-1}\\
 	 1 & \lambda_3 & \lambda_3^2 & \ldots & \lambda_3^{N-1}\\
 	 \vdots & \vdots & \vdots & \vdots & \vdots \\
 	 1 & \lambda_N & \lambda_N^2 & \ldots & \lambda_N^{N-1}\\
	\end{array}\right)\,,
\]
and by $D$ the diagonal matrix with entries $D_{ii} = N^{(i-1)/2}/\sqrt{(i-1)!}$, one can verify that $K=V^{\pmb{\lambda}}DD^\dag {V^{\pmb{\lambda}}}^\dg$. Consequently, by the Vandermonde determinant formula 
\[
		\prod_{1\leq i<j\leq N}|\lambda_{i}-\lambda_{j}|^{2} = \frac{\prod_{k=1}^{N}(k-1)!}{N^{N(N-1)/2}}\mathrm{Det}(K)\,.
\]
As a result we can rewrite $\psi_m$ as 
\begin{equation}\label{eq:psi_m_explicit}
	\psi_m(\tbz',\tbw') = \frac{1}{\widehat{Z}_N}\int_{\BC^N} \mathrm{Det}(K) \prod_{i=1}^{m}(\lambda_i-w'_i)^{-1}(\overline{\lambda}_i-\overline{z'}_i)^{-1}\prod_{i\in [N]} \kappa_N(\textrm{d}\lambda_i)\,,
\end{equation}
where 
\[
	\widehat{Z}_N = \frac{N^{N(N-1)/2}}{\prod_{k=1}^{N} (k-1)!}Z_N=\int_{\BC^N} \mathrm{Det}(K)\prod_{i\in [N]}\kappa_N(\textrm{d}\lambda_i)\,.
\]

The advantage the renormalized version $\textrm{Det}(K)$ which we use instead of $\prod_{1\leq i<j\leq N}|\lambda_i-\lambda_j|^2=\textrm{Det}(V^{\pmb{\lambda}} {V^{\pmb{\lambda}}}^\dg)$ is the orthonormality property of the vector $\tbp^{\lambda}$ with respect to integration over $\lambda$, namely
\begin{equation}\label{eq:ON}
	\int_{\BC} p^{\lambda}_i\overline{p^{\lambda}_j} \kappa_N(\textrm{d}\lambda) = \delta_{ij},\qquad \forall i,j\in[N]\,.
\end{equation}

We now turn to estimate $\psi_m$ with the help of \eqref{eq:ON}. We start with an explicit calculation for $\widehat{Z}_N$. Using the Leibniz' formula for the determinant and the invariance of the trace under cyclic permutations
\begin{align}\label{eq:Imsotired}
	\textrm{Det}(K) &=\sum_{\pi\in \CS_{[N]}} (-1)^{\mathrm{sgn}(\pi)} \prod_{i=1}^N({\tbp^{\lambda_{\pi(i)}}}^\dg\cdot \tbp^{\lambda_i}) = \sum_{\pi \in \caS_{[N]}} (-1)^{\textrm{sgn}(\pi)} \prod_{\CL\in \pi} \prod_{\alpha\in \CL}( {\tbp^{\lambda_\alpha}}^\dg \cdot \tbp^{\lambda_{\CL(\alpha)}}) \nonumber\\ 
	& = \sum_{\pi \in \caS_{[N]}} (-1)^{\textrm{sgn}(\pi)} \prod_{\CL\in \pi} \mathrm{Tr}\Big(\prod_{\alpha\in \CL}{\tbp^{\lambda_\alpha}}{\tbp^{\lambda_\alpha}}^\dg\Big)\,,
\end{align}
where the product over $\CL$ is taken over all cycles in the cycle decomposition of $\pi$ and in the last expression, the order of product over the matrices ${\tbp^{\lambda_\alpha}}{\tbp^{\lambda_\alpha}}^\dg$ is according to the order in the cycle $\CL$. Recalling that  $|\pi|$ the number of cycles in $\pi$ and using \eqref{eq:ON}, gives
\begin{align*}
	\widehat{Z}_N & = \sum_{\pi\in\CS_{[N]}} (-1)^{\mathrm{sgn}(\pi)}\prod_{\CL\in \pi} \int_{\BC^{|\SV(\CL)|}} \mathrm{Tr}\Big(\prod_{\alpha\in \CL}{\tbp^{\lambda_\alpha}}{\tbp^{\lambda_\alpha}}^\dg\Big) \prod_{\alpha\in \CL}\kappa_N(\textrm{d}\lambda_\alpha)=\sum_{\pi\in\CS_{[N]}} (-1)^{\mathrm{sgn}(\pi)} N^{|\pi|}\\
	&  = \sum_{\pi\in\CS_{[N]}} \prod_{\CL\in \pi} (-1)^{|\CL|-1}N= (-1)^{N}\sum_{k=1}^{N}(-N)^k |\{\pi\in\CS_{[N]} ~:~ |\pi|=k\}|=N!\,,
\end{align*}
where in the last step we used the fact that the generating function for the number of cycles in a permutation in $\CS_{[N]}$ is given by $\sum_{k=1}^{N}|\{\pi\in\CS_{[N]} ~:~ |\pi|=k\}|t^k = t(t+1)(t+2)\cdots (t+N-1)$, see \cite[Proposition 1.3.7]{Sta97}.

Next lets consider the integral $\widehat{Z}_N\psi_{m}(\bfz', \bfw')$, see \eqref{eq:psi_m_explicit},  with respect to the above decomposition of $\mathrm{Det}(K)$.  We start by integrating out $\lambda_N, \cdots, \la_{N-m+1}$.  Let $K^{(m)}$ be the $m$-by-$m$ matrix $(K_{i, j})_{1\leq i, j\leq m}$.  Note that $K^{(m)}$ is NOT the matrix $K$ corresponding to the $m$-by-$m$ Ginibre matrix, as the sums in the inner product of $\tbp^{\lambda}$ and $\tbp^{\lambda_i\,\dg}$ goes all the way to $N$. Using the orthonormality property \eqref{eq:ON}, one can verify that 
\[
	\int_{\C^{N-m}} \textrm{Det}(K) \prod_{i= m+1}^{N} \kappa_N(\textrm{d}\lambda_i)=(N-m)!\,\textrm{Det}(K^{(m)})\,,
\]
and therefore 
\begin{equation}
	{N \choose m}N^{-m}\psi_{m}(\bfz', \bfw')= \frac{N^{-m}}{m!} \int_{\C^m}\hspace{-5pt}\textrm{Det}(K^{(m)}) \prod_{i=1}^m ({\la_i-w_i})^{-1}({\overline{\la_i}-\overline{z_i}})^{-1} \prod_{i=1}^{m}\kappa_N(\textrm{d}\lambda_i)\,.
\end{equation}

For $z,w\in\BC$ such that $z\neq w$, we introduce the kernel $\frK(w,z)=(\frK_{i,j}(w,z))_{i,j\in [N]}$ given by 
\begin{equation}
	\frK_{i,j}(w,z)=\frac{1}{N}\int_\C p^\lambda_i \overline{p^\lambda_j}~ (\lambda-w)^{-1}(\overline{\lambda}-\overline{z})^{-1} \kappa_N(\mathrm{d}\lambda)\,.
\end{equation}

Repeating the argument in \eqref{eq:Imsotired} for the determinant $\mathrm{Det}(K^{(m)})$, we obtain
\begin{align}\label{eq:normalized_psi_m}
	& {N \choose m}N^{-m}\psi_{m}(\bfz', \bfw')	\nonumber\\
	&\qquad  =\frac{1}{m!}\sum_{\pi\in \caS_{[m]}} (-1)^{\textrm{sgn}(\pi)}\prod_{\CL\in \pi}\Tr(\frK(w_i,z_i) \frK(w_{\CL(i)},z_{\CL(i)}) \ldots  \frK(w_{\CL^{|\SV(\CL)|-1}(i)},z_{\CL^{|\SV(\CL)|-1}(i)}))\,.
\end{align}

The next lemma contains the bounds on the kernel $\frK(z,w)$ needed to complete the proof of Lemma \ref{lem:bound_on_psi_m}. 
\begin{lemma}\label{L:pfe}$~$
\begin{enumerate}
\item For every $z,w\in\tbD_1$ such that $|z-w|>0$
\begin{equation}\label{eq:k1}
	 |\frK_{i,j}(z,w)|\leq 4\Big(1+\frac{1}{|z-w|^2N}\Big),\qquad \forall i,j\in [N]\,.
\end{equation}
\item For every $k\geq 2$ and $\tbz, \tbw\in \tbD_1^k$ such that $\mathrm{dist}(\tbz,\tbw)>0$
\begin{equation}\label{eq:k2}
	|\Tr(\frK(w_1, z_1)\cdots \frK(w_k, z_k))|\leq  \bigg(\frac{C}{\mathrm{dist}(\tbz,\tbw)^2}\bigg)^k\cdot \bigg(\frac{\log N}{\sqrt{N}}\bigg)^{k-1}\,,
\end{equation}
where $C\in (0,\infty)$ is some universal constant. 
\item Fix $\delta>0$. There exists a universal constant $C\in (0,\infty)$ such that for every $z,w\in\tbD_1$ 
\begin{equation}\label{eq:k3}
		|\Tr(\frK(w, z))-h(z,w)|\leq \frac{C}{N^{\frac{1}{2}-\delta}|z-w|}\,,
\end{equation}
where $C\in (0,\infty)$ is some universal constant. In particular, 
\begin{equation}\label{eq:k4}
	|\Tr(\frK(w,z))|\leq |h(z,w)| + \frac{C}{N^{\frac{1}{2}-\delta}|z-w|}\leq \frac{C}{|z-w|^2}\,.
\end{equation}
\end{enumerate}
\end{lemma}

We postpone the proof of Lemma \ref{L:pfe} to Appendix \ref{appendix} and turn to complete the proof of Lemma \ref{lem:bound_on_psi_m}. We start with the proof of \eqref{eq:psi_bound_i}. Splitting the sum in \eqref{eq:normalized_psi_m} into $\pi=\mathrm{id}$ and $\pi\in\CS_{[m]}\setminus \{\mathrm{id}\}$ we obtain from Lemma \ref{L:pfe}
\begin{align}\label{eq:bound_bound}
	&\bigg|{N \choose m}N^{-m}\psi_{m}(\bfz', \bfw')-
		\frac{1}{m!}\prod_{i=1}^{m} h(z'_i, w'_i)\bigg|\nonumber\\
	& \qquad \leq \frac{1}{m!}\bigg|\prod_{i=1}^{m}\mathrm{Tr}(\frK(z_i',w_i'))-\prod_{i=1}^m h(z_i',w_i')\bigg|\nonumber\\
	&\qquad +\frac{1}{m!} \sum_{\textrm{id}\neq \pi\in\CS_{[m]}} \prod_{\CL\in \pi}\big|\Tr(\frK(w_i,z_i) \frK(w_{\CL(i)},z_{\CL(i)}) \ldots  \frK(w_{\CL^{|\SV(\CL)|-1}(i)},z_{\CL^{|\SV(\CL)|-1}(i)}))\big|\,.
\end{align}

Using a telescopic sum  we can rewrite the first term on the right hand side as 
\begin{align*}
	&\frac{1}{m!}\bigg|\prod_{i=1}^{m}\mathrm{Tr}(\frK(z_i',w_i'))-\prod_{i=1}^m h(z_i',w_i')\bigg| \\
	\leq &	\frac{1}{m!}\sum_{j=1}^{m}\bigg|\prod_{i=1}^{j-1} h(z_i',w_i')\prod_{i=j}^{m}\mathrm{Tr}(\frK(z_i',w_i'))-\prod_{i=1}^{j} h(z_i',w_i')\prod_{i=j+1}^{m}\mathrm{Tr}(\frK(z_i',w_i'))\bigg|\\
	= &	\frac{1}{m!}\sum_{j=1}^{m}\bigg|\prod_{i=1}^{j-1} h(z_i',w_i')\prod_{i=j+1}^{m}\mathrm{Tr}(\frK(z_i',w_i'))\Big[\mathrm{Tr}(\frK(z_j',w_j')) - h(z_j',w_j')\Big]\bigg|\,.
\end{align*}
and therefore by \eqref{eq:k3} and \eqref{eq:k4},
\[
	\frac{1}{m!}\bigg|\prod_{i=1}^{m}\mathrm{Tr}(\frK(z_i',w_i'))-\prod_{i=1}^m h(z_i',w_i')\bigg|\leq
	\frac{1}{(m-1)!}\bigg(\frac{C}{\mathrm{dist}(\tbz,\tbw)^2}\bigg)^{m}\frac{1}{N^{1/2-\delta}}\,.
\]

Turning to deal with the second term on the right hand side of \eqref{eq:bound_bound},  since there is noting to prove for $m=1$, we assume that $m\geq 2$. By \eqref{eq:k2} and \eqref{eq:k4}, this is bounded from above by 
\begin{align*}
	&\frac{C}{\mathrm{dist}(\tbz,\tbw)^2 m!}\sum_{\textrm{id}\neq \pi\in\CS_{[m]}}\prod_{\CL\in\pi}\cdot \bigg(\frac{C\log N}{\mathrm{dist}(\tbz,\tbw)^2\sqrt{N}}\bigg)^{|\SV(\CL)|-1}\\
	&\qquad =\frac{C}{\mathrm{dist}(\tbz,\tbw)^2 m!}\bigg[ -1 + \prod_{j=1}^{m-1}\bigg(1+\frac{jC\log N}{\mathrm{dist}(\tbz,\tbw)^2\sqrt{N}}\bigg)\bigg]\leq \frac{C\log N}{\mathrm{dist}(\tbz,\tbw)^{2m}  (m-2)!\sqrt{N}}\,,
\end{align*}
where in the last equality we used the generating function for the number of loops in a permutation. Combining the estimation for both terms \eqref{eq:psi_bound_i} follows. 

Tuning to prove \eqref{eq:psi_bound_ii}, a similar computation shows that  
\begin{align*}
	&\bigg|{N \choose m}N^{-m}\psi_{m}(\bfz', \bfw')\bigg|\\
	&\qquad \leq \frac{1}{m!} \sum_{\pi\in\CS_{[m]}} \prod_{\CL\in \pi}\big|\Tr(\frK(w_i,z_i) \frK(w_{\CL(i)},z_{\CL(i)}) \ldots  \frK(w_{\CL^{|\SV(\CL)|-1}(i)},z_{\CL^{|\SV(\CL)|-1}(i)}))\big|\\
	& \qquad \leq \frac{1}{m!} \sum_{\pi\in\CS_{[m]}} \prod_{\CL\in \pi}\bigg(\frac{C}{\mathrm{dist}(\tbz,\tbw)^2}\bigg)^{|\SV(\CL)|}\cdot \bigg(\frac{\log N}{\sqrt{N}}\bigg)^{|\SV(\CL)|-1} \\
	&\qquad = \frac{1}{m!}\bigg(\frac{C}{\mathrm{dist}(\tbz,\tbw)^2}\bigg)^{m}\cdot \prod_{j=1}^{m-1}\bigg(1+\frac{j\log N}{\sqrt{N}}\bigg)\leq \bigg(\frac{C}{m\cdot  \mathrm{dist}(\tbz,\tbw)^2}\bigg)^{m}\,,
\end{align*}
where in the last inequality we used the fact that for sufficiently large $N$ (depending on $m$)
\[
	\prod_{j=1}^{m-1}\bigg(1+\frac{j\log N}{\sqrt{N}}\bigg) \leq \exp\bigg(\sum_{j=1}^{m-1}\frac{j\log N}{\sqrt{N}}\bigg) = \exp\bigg(\binom{m}{2}\frac{\log N}{\sqrt{N}}\bigg)\leq C\,.
\]
This completes the proof of \eqref{eq:psi_bound_ii}. \hfill\qed


\section{The matrix $\frN$ and its eigenvectors.}\label{sec:The_matrix_frN}

In this section we wish to provide further information on the limiting correlation matrix $\exp(\frN)$. In particular: (1) We prove that the formula for $\FN$ obtained in Theorem \ref{thm:limit_exists} and the formula in the introduction \eqref{eq:The_Matrix_N} coincide. (2) We provide an explicit formula for the correlation function $\rho(C_\ell)$ for $\ell=1,2$ and $3$. (3) We introduce a recursive formula for computing the correlation functions.  The main ingredient in obtaining all of the above is the understanding of the algebraic structure originating in the entries of the eigenvectors of the matrix $\FN$

\subsection{Rewriting the matrix $\FN$}
	Our first goal is to provide further insight into the entries of the matrix $\frN$ from Theorem \ref{thm:limit_exists}. In particular we wish to prove that $\FN$ is upper triangular with respect to the partial ordering $\preceq$ and that the formula for its entries given in \eqref{eq:The_Matrix_N} equals the one provided in Theorem \ref{thm:limit_exists}.

We start by rewriting the diagonal entries of $\FN$. Note that in \eqref{eq:THE_MATRIX_N_WELL}, for $\si,\si'\in\CS$, the sume over $F$ runs over all subsets $F\subset \CE(\si)$ of size $|F|=|\SV(\si)|-1$, and that due to the cyclic structure of the digraph associated with the permutation, those sets are exactly the sets containing all but one of the directed edges. Consequently, 
\[
		\FN_{\si,\si} = \sum_{e=(\alpha,\beta)\in\CE(\si)} h(z_\beta,w_\alpha) = \sum_{\alpha\in \SV(\si)}h(z_{\alpha},w_{\si^{-1}(\alpha)})=\frh_\si(\tbz,\tbw)\,,
\]
see \eqref{eq:the_function_h}, for the definition of $\frh_\si$.

Next, we turn to deal with the off-diagonal entries. For $\si,\tau\in\CS$ declare $\si$ to be smaller than $\tau$, denoted $\si\preceq\tau$, if $\SV(\si)\subset \SV(\tau)$ and $|\SV(\si)|+|\tau|-|\si|-1\leq |\CE(\si,\tau)|$. Similarly, we denote $\si\prec \tau$ if $\si\preceq \tau$ and $\si\neq \tau$. Although it is not immediately clear that the definition of $\preceq$ given here is related to the one given in the introduction, we will shortly show that they coincide. It follows from the definition of $\prec$ and the formula for $\FN$ provided in Theorem \ref{thm:limit_exists}, that $\FN_{\si,\tau}\neq0$ implies $\si\preceq \tau$. 

\begin{claim}
	The relation $\prec$ is asymmetric, and therefore for every distinct $\si,\tau\in\CS$ 
\[
	\FN_{\si,\tau}\neq 0 \qquad \Rightarrow \qquad \FN_{\tau,\si}=0\,. 
\]
In particular, the matrix $\FN$ is upper triangular with respect to $\preceq$. 
\end{claim}

\begin{proof}
	Let $\si,\tau\in\CS$ and assume that $\si\preceq \tau$ and $\tau\preceq\si$. It follows from the definition of $\preceq$ that $\SV(\si)=\SV(\tau)$ and that $|\CE(\si,\tau)|\geq  |\SV(\si)|+|\si|-|\tau|-1$ as well as $|\CE(\si,\tau)|\geq |\SV(\si)|+|\tau|-|\si|-1$, and therefore $|\CE(\si,\tau)|= |\SV(\si)|-1$. Since $|\CE(\si,\tau)|\leq |\SV(\si)|-2$ whenever $\si\neq \tau$ (because the edges are directed and form disjoint loops) this implies that $\si=\tau$, as required. 
\end{proof}

Let $\si,\tau\in\CS$ such that $\si\preceq\tau$. Since in this case $\SV(\si)\subset\SV(\tau)$, the definition of the digraphs associated with the permutations implies that 
\[
	\CE(\si,\tau)=\CE(\si)\cap \CE(\tau) = \{(\si^{-1}(\beta),\beta) ~:~ \beta\in\SV(\si) \text{ and }\tau\circ \si^{-1}(\beta)=\beta\}
\]
and therefore $|\SV(\si)|-|\CE(\si,\tau)|=|\SV_{nf}(\si;\tau)|$, where for a pair of diagrams $\si,\tau\in\CS$ such that $\si\preceq\tau$ we introduced the notation
\[
	\SV_{nf}(\si;\tau)=\{\beta\in\SV(\si) ~:~ \beta \text{ is not a fixed point of }\tau\circ\si^{-1}\}\,.
\]

Consequently, $\si\preceq\tau$ if and only if $\SV(\si)\subset\SV(\tau)$ and $|\SV_{nf}(\si;\tau)|= |\si|-|\tau|+1$. 

Next, we distinguish between cycles in $\si$ according to the number of non fixed points in them (with respect to $\tau$). Let $\si,\tau\in\CS$ such that $\si\preceq\tau$. A cycle $\CL\in \si$ is said to be of type $i$ (with respect to $\tau$) for $i\in \{0,1\}$, if $|\SV_{nf}(\si;\tau)\cap \SV(\CL)|=i$. Any other cycle is said to be of type $2$ (with respect to $\tau$), i.e., a cycle $\CL\in \si$ is of type $2$ if $|\SV_{nf}(\si;\tau)\cap \SV(\CL)|\geq 2$. For $i\in \{0,1,2\}$, we denote by $J_i(\si;\tau)$ the number of cycles of type $i$ in $\si$ with respect to $\tau$. 

\begin{claim}\label{clm:prec}
	Let $\si,\tau\in\CS$ such that $\si\preceq \tau$. Then, $J_2(\si;\tau)=0$ and $J_0(\si;\tau)\in \{|\tau|-1,|\tau|\}$, with $J_0(\si;\tau)=|\tau|$ if and only if $\si=\tau$. 
\end{claim}

\begin{proof}
	We fix $\si,\tau\in\CS$ such that $\si\preceq\tau$ and abbreviate $J_i=J_i(\si;\tau)$ for $i\in\{0,1,2\}$. Since every cycle of type $0$ in $\si$ (with respect to $\tau$) is also a cycle in $\tau$, it follows that $J_0\leq |\tau|$ with equality if and only if $\si=\tau$.
	Note that $\sum_{i=0}^2 J_i = |\si|$ and that $\sum_{i=0}^2 iJ_i \leq |\SV_{nf}(\si;\tau)|$ and therefore, the assumption $\si\preceq \tau$, implies $J_0\geq |\tau|-1+J_2$ and in particular that $ J_0\geq |\tau|-1$. Observing that $J_2\geq 1$ implies $J_0\geq |\tau|$ and recalling that $J_0\leq |\tau|$ with equality if and only if $\si=\tau$ and in particular $J_2=0$, the result follows. 
\end{proof}

The last claim provides us with a relatively simple way to describe the relation $\prec$ between permutations. An interval of a  permutation $\tau$ is a sequence of vertices in $\SV(\tau)$, which appear consecutively along the orbit of the permutation. For example if, $\tau=(1,2,3,5,7,4,8)$ then $(1,2,3,5)$, $(3,5,7)$ and $(4,8,1,2)$ are all intervals of $\tau$, but $(1,2,3,4)$ is not. 

\begin{corollary}[the relation $\prec$]\label{cor:perc}
	Let $\si,\tau\in\CS$. Then $\si\preceq\tau$ if and only if all but at most one of the cycles $\CL\in \tau$ are also cycles of $\si$ and $\si$ is obtained from $\tau$ by decomposing the remaining cycle $\CL$ into disjoint intervals of $\CL$ and then for each of the intervals $(v_1,v_2,\ldots,v_m)$, either removing it or declaring it to be a cycle in $\si$ given by $v_i$ to $v_{i+1 \text{ mod }m}$ for every $1\leq i\leq m$. See Figure \ref{fig:prec} for an illustration. 
\end{corollary}

\begin{figure}
  \centering
    \includegraphics[width=0.7\textwidth]{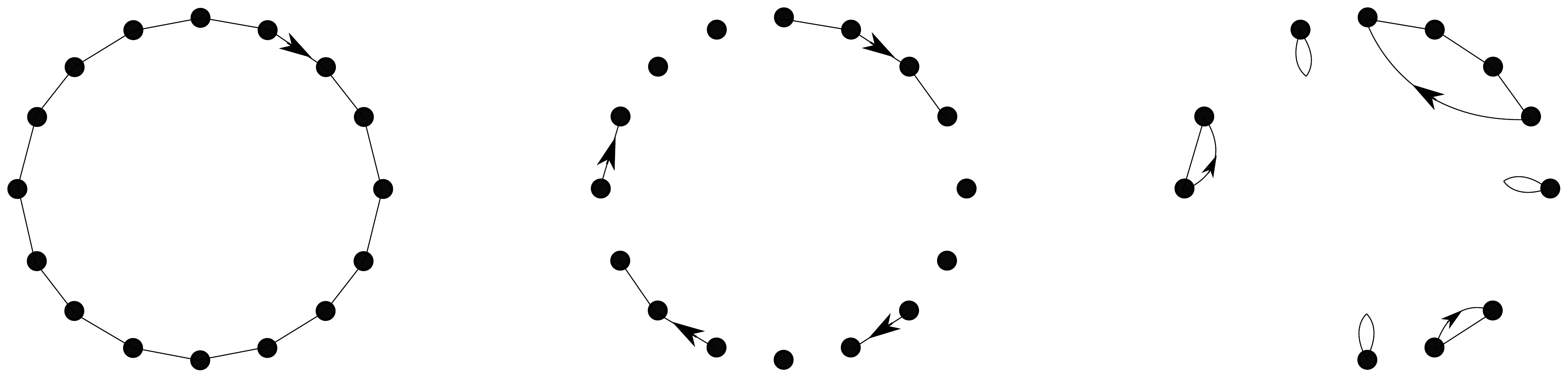}
     \caption{The cycle on the left $\tau$ is decomposed into intervals (middle) and then each of the intervales is either closed into a loop or deleted, thus generating $\si\prec\tau$. If there $\tau$ is composed of several cycles, then the remaining cycles must also be part of $\si$. \label{fig:prec}}
\end{figure}

\begin{proof}
	Assume first that $\si \prec\tau$. By Claim \ref{clm:prec} , we have $J_0(\si;\tau)=|\tau|-1$. Consequently, all but at most one of the cycles of $\tau$ are also cycles of $\si$. Denoting by $\CL$ the remaining cycle of $\tau$, since $\SV(\si)\subset \SV(\tau)$, the vertex set of each additional cycle of $\si$ is contained in $\SV(\CL)$. Furthermore, using Claim \ref{clm:prec} once more, each of these cycles has exactly one vertex in $\SV_{nf}(\si;\tau)$, which implies that it must be an interval of $\CL$. 
	
	In the other direction, if $\si,\tau\in\CS$ satisfy the conditions in Corollary \ref{cor:perc}, then by definition $\SV(\si)\subset \SV(\tau)$. Furthermore, the only edges in $\CE(\si)$ which are not in $\CE(\si,\tau)$ are the ones used to close the intervals in the unique cycle of $\tau$ which is not a cycle of $\si$. This gives one edge for each cycle in $\si$ which is not in $\tau$, and therefore 
\[
	|\CE(\si,\tau)| = |\CE(\si)|-(|\si|-|\tau|+1)=|\SV(\si)|-|\tau|+|\si|-1
\]
which implies that $\si\preceq\tau$. 
\end{proof}

\begin{remark}
	Note that $\prec$ is asymmetric but is not transitive. The relation $\trieq$ from the introduction, see also \eqref{lem:warum}, is a partial ordered set obtained from $\prec$ to guarantee the relation is also transitive. 
\end{remark}

The last corollary implies that for every $\si,\tau\in\CS$, we have $\FN_{\si,\tau}>0$ if and only if $\si\preceq\tau$ in which case 
\begin{align}\label{eq:FN}
	\FN_{\si,\tau} &= \hspace{-10pt}\sideset{}{'}\sum_{\substack{\alpha\in \si^{-1}(\widehat{\SV}_{nf}(\si;\tau))\\ \beta\in \widehat{\SV}_{nf}(\si;\tau)}}h(z_\beta,w_\alpha)\hspace{-10pt}\prod_{\alpha\neq \gamma\in \si^{-1}(\widehat{\SV}_{nf}(\si;\tau))}\hspace{-10pt}(w_\alpha-w_\gamma)^{-1}\prod_{\beta\neq \delta\in \widehat{\SV}_{nf}(\si;\tau)}\hspace{-10pt}(\overline{z}_\beta-\overline{z}_\delta)^{-1}\nonumber\\
&= \frac{1}{\pi}\int_{\D_1}\prod_{\alpha\in \widehat{\SV}_{nf}(\si;\tau)} \frac{1}{(\overline{\nu}-\overline{u_{\alpha}})} \frac{1}{(\nu-w_{\sigma^{-1}(\alpha)})} \textrm{d}^2 \nu\,,
\end{align}
where we introduced the notations 
\[
	\widehat{\SV}_{nf}(\si;\tau) = \SV_{nf}(\si;\tau)\cup (\SV(\tau)\setminus\SV(\si)),
\]
and defined $\si^{-1}(\widehat{\SV}_{nf}(\si;\tau))=\si^{-1}(\SV_{nf}(\si;\tau))\cup (\SV(\tau)\setminus\SV(\si))$.

Recalling the definition of $\frn_{\si,\tau}$, see \eqref{the_function_Fn}, we obtain 
\begin{equation}\label{eq:FN}
	\frN_{\sigma, \tau}=\begin{cases}
		\frh_{\tau}(\tbz,\tbw) & \text{ if }\sigma=\tau, \\
		\frn_{\sigma, \tau}(\tbz,\tbw) & \text{ if }\sigma\prec \tau,\\
		0 & \text{ otherwise}.
\end{cases}
\end{equation}
This completes the proof of the equivalence between the formula for $\FN$ given in Theorem \ref{thm:limit_exists} and the one in \eqref{eq:The_Matrix_N}.


\subsection{Examples}

 \subsubsection{The case $\ell=1$} In this case 
\[
	\frN = \begin{blockarray}{cc}
	 \emptyset & (1)  \\
	 \begin{block}{(cc)}
		0 & h(w_{1},z_{1})\\
		0 & h(w_{1},z_{1})\\
	\end{block}
	\end{blockarray}\,.
\]
and therefore 
\[
		\lim_{N\to\infty} N^{-1}\BE[F_N((1))] = \tbe_\emptyset^\dag \exp(\FN) \tbe_{(1)} = \exp(h(w_1,z_1))-1 = \frac{1-w_1\overline{z_1}}{|w_1-z_1|^2}-1\,.
\]
Using the complex Green theorem, we conclude 
\begin{align}
	{\rho}_{2}(\nu_1,\nu_2) & = \lim_{\varepsilon\to 0}\frac{1}{\pi^{2}\varepsilon^{4}}\int_{\caS_\varepsilon^{2}(\nu_1,\nu_2)}\frac{1-w_1\overline{z_1}}{|w_1-z_1|^2}-1\textrm{d}\overline{z_1} \textrm{d}w_1\nonumber\\
	& = \lim_{\varepsilon\to 0}\frac{1}{\pi^{2}\varepsilon^{4}}\int_{|w_1-\nu_2|<\varepsilon}\int_{|z_1-\nu_1|<\varepsilon}\partial_{z_1}\partial_{\overline{w_1}}\frac{1-w_1\overline{z_1}}{|w_1-z_1|^2} \textrm{d}z_1\wedge \overline{z_1} \textrm{d}w_1\wedge\textrm{d}\overline{w_1} \nonumber\\
		& = -\lim_{\varepsilon\to 0}\frac{1}{\pi^{2}\varepsilon^{4}}\int_{|w_1-\nu_2|<\varepsilon}\int_{|z_1-\nu_1|<\varepsilon}\frac{1-w_1\overline{z_1}}{|w_1-z_1|^4}\textrm{d}z_1\wedge \textrm{d}\overline{z_1} \textrm{d}w_1\wedge\textrm{d}\overline{w_1} \nonumber\\
		& = -\frac{1-\overline{\nu_1}\nu_2}{|\nu_1-\nu_2|^4}\,.
\end{align}

\subsubsection{The case $\ell=2$} In this case 
{\small \[
	\frN =
	\begin{blockarray}{ccccc}
	 \emptyset & (1)  & (2) & (1)(2) & (1,2)\\
	 &&&\\
	 \begin{block}{(ccccc)}
		0 & h(w_1,z_1) & h(w_2,z_2) & 0 & H \\
		0 & h(w_1,z_1) & 0 & h(w_2,z_2) & H\\
		0 & 0 & h(w_2,z_2) &  h(w_1,z_1) & H \\
		0 & 0 & 0 & h(w_1,z_1)+h(w_2,z_2) & H \\
		0 & 0 & 0 & 0 & h(w_1,z_2)+h(w_2,z_1)\\
	\end{block}
	\end{blockarray}\,,
\]}
where we denoted 
\[
	H=(w_1-w_2)^{-1}(\overline{z_1}-\overline{z_2})^{-1}[h(w_1,z_1)+h(w_2,z_2)-h(w_1,z_2)-h(w_2,z_1)].
\]
As a result
\begin{align*}
	&\lim_{N\to\infty} N^{-1}\BE[F_N((1,2))]  = \tbe_\emptyset^\dag \exp(\FN) \tbe_{(1,2)} \\
	&\qquad  = (w_1-w_2)^{-1}(\overline{z_1}-\overline{z_2})^{-1}\bigg[\frac{(1-w_1\overline{z_1})(1-w_2\overline{z_2})}{|w_1-z_1|^2|w_2-z_2|^2}-\frac{(1-w_1\overline{z_2})(1-w_2\overline{z_1})}{|w_1-z_2|^2|w_2-z_1|^2}\bigg]\,.
\end{align*}
Using the complex Green theorem once again, we obtain 
\begin{equation*}
\begin{aligned}
		{\rho}_4(\nu_1,\nu_2,\nu_3,\nu_4)&= \frac{1}{(\nu_2-\nu_4)(\overline{\nu_1}-\overline{\nu_3})}\bigg[\frac{(1-\nu_2\overline{\nu_1})(1-\nu_4\overline{\nu_3})}{|\nu_2-\nu_1|^4|\nu_4-\nu_3|^4}
	-\frac{(1-\nu_2\overline{\nu_3})(1-\nu_4\overline{\nu_1})}{|\nu_2-\nu_3|^4|\nu_4-\nu_1|^4}\bigg]\\
	& =\frac{1}{(\nu_2-\nu_4)(\overline{\nu_1}-\overline{\nu_3})}\Big[{\rho}_2(\nu_1,\nu_2){\rho}_2(\nu_3,\nu_4)-{\rho}_2(\nu_1,\nu_4){\rho}_2(\nu_3,\nu_2)\Big]\,.
\end{aligned}
\end{equation*}

The case by case computation of this correlation functions seems to be complicated in general and so in the next section we shall provide some important structural properties of the matrix $\FN$ which provides further information on the high-order correlation functions. 

%


\subsection{Eigenvectors}\label{sub:eigenbevtors} In order to calculate the correlation function for general permutations and in particular for the cyclic permutation on $[\ell]$, we consider eigenvectors of $\frN$. To this end we fix $\ell\in\BN$ and recall that $\FN_\ell$ is the restriction of $\FN$ to rows and columns of permutations in $\CS_\ell$. In terms of the partial order $\preceq$ described above, $\frN_\ell$ is an upper triangular matrix with eigenvalues $(\frh_{\si})_{\si\in \CS_\ell}$. It has two bases, one of left eigenvectors denoted by $(\frl_\si(\cdot))_{\si\in \CS_\ell}$ and another set of right eigenvectors denoted by $(\frr_\si(\cdot))_{\si\in \CS_\ell}$. We normalize them so that
\[
	\frl_{\si}(\si)=\frr_{\si}(\si)=1,\qquad\forall \si\in\CS_\ell\,.
\]

A priori, the eigenvectors of $\FN_\ell$ depend on $\ell$. However, due to the fact that $\FN$ (and hence $\FN_\ell$) is an upper triangular matrix, common entries of eigenvectors of a permutation $\si$ for different values of $\ell$ are consistent. In other words, for every $\ell_1<\ell_2$ and $\si\in\CS_{\ell_1}$ the left (right) eigenvectors of $\FN_{\ell_1}$ and $\FN_{\ell_2}$ associated with $\si$ coincide on all joint entries, the entries in $\CS_{\ell_1}$. 

If we introduce the matrices  associated with these eigenvectors $L_{\si,\tau}:=\frl_\si(\tau)$ and $R_{\si, \tau}=\frr_\tau(\si)$, then
\[
	\frL=\frR^{-1}\qquad , \qquad  \frN=\frR \frh \frL\qquad ,\qquad \exp(\FN) = \frR\exp(\frh)\frL\,,
\] 
with $\frh$ understood as the diagonal matrix with entries $(\frh_\si)_{\si\in\CS_\ell}$. Therefore, in order to calculate $\exp(\frN)$, it suffices to compute the eigenvectors of $\FN$. We shall work with the left eigenvectors for notational convenience. The proof for the right eigenvectors are similar.

The eigenvector equation $\frh_\si \frl_\si = \frl_\si \frN$ reads
\begin{equation}\label{eq:pre_EV}
	[\frh_{\si}-\frh_{\tau}]\frl_\si(\tau)= \sum_{\pi\neq \tau} \frl_{\si}(\pi) \frN_{\pi, \tau}
	=\sum_{\pi\prec \tau} \frl_{\si}(\pi) \frN_{\pi, \tau},\qquad \forall \si, \tau\in\CS_\ell\,.
\end{equation}

Recall the definition of $\trieq$ from the introduction. The following lemma shows that the summation can be restricted to an even smaller set of permutations. 

\begin{lemma}\label{lem:warum}
	Let $\si,\tau\in\CS$, then $\si\trieq \tau$ if and only if there exists $m\in\BN$ and $\si=\pi_0\preceq \pi_1\prec\pi_2\ldots \pi_m = \tau$. In particular, $\trieq$ is reflexive and transitive, if $\si,\tau\in\CS_\ell$, then $\frl_\si(\tau)\neq 0$ if and only if $\si\trieq\tau$ and the eigenvector equation can be written as 
\begin{equation}\label{eq:EV}
	[\frh_{\si}-\frh_{\tau}]\frl_\si(\tau)= \sum_{\si\trieq \pi \prec\tau} \frl_{\si}(\pi) \frN_{\pi, \tau},\qquad \forall \si, \tau\in\CS_\ell\,.
\end{equation}
\end{lemma}

\begin{proof}
	Let $\si,\tau\in\CS_\ell$. If $\ell_\si(\tau)\neq 0$, then by the eigenvector equation \eqref{eq:pre_EV}, one can find $m\in\BN$ and a sequence $\si=\pi_0,\pi_1,\pi_2,\ldots, \pi_m=\tau$ such that $\frN_{\pi_{j-1},\pi_{j}}\neq 0$ for all $1\leq j\leq m$, namely $\pi_{j-1}\preceq \pi_j$ for all $1\leq j\leq m$. The other direction is proved similarly via induction on $m$. The proof of \eqref{eq:EV} follows by induction on pairs $\si,\tau$ with respect to   $\prec$. 
\end{proof}

\begin{lemma}\label{lem:exist_uniq}
	The family of equations in \eqref{eq:EV} has a unique solution subject to the boundary condition $ \frl_{\si}(\si)=1$ for all $\si\in\CS_\ell$.
\end{lemma}

\begin{proof}
	Since for a generic choice of $\tbz,\tbw\in\BC^\ell$, it holds that $(\frh_\si)_{\si\in\BS_\ell}$ are all distinct, it follows that all the all eigenvalues of $\FN_\ell$ are distinct and hence all eigenvectors are unique up to multiplication by a scalar. The scalars are fixed by the choice of normalization $\frl_\si(\si)=1$ for $\si\in\BS_\ell$ and hence so are the eigenvectors. 
\end{proof}

Having proved the existence and uniqueness of the solution, we turn to discuss its properties. 
%

\begin{lemma}[Tensorial property of eigenvector components] \label{L:Tens}
	Let $\si,\tau\in\CS_\ell$ such that $\si\trieq \tau$ and suppose that $\tau=\{\CL_k\}_{k=1}^{|\tau|}$ are the cycles of $\tau$. Then, there exists $\si_1,\ldots,\si_{|\tau|}\in\CS_\ell$ such that $\si=\biguplus_{k=1}^{|\tau|}\si_k$ and $\si_k\trieq\CL_k$ for $1\leq k\leq |\tau|$. Furthermore, 
\begin{equation}\label{eq:Tens}
	\frl_{\si}(\tau; \bfz, \bfw)=\prod_{k=1}^{|\tau|} \frl_{\si_k}(\CL_k; \tbz|_{\SV(\CL_k)}, \tbw|_{\SV(\CL_k)})\,.
\end{equation}

In particular, if $\si\prec\tau$ and $\CL$ is the unique cycle of $\tau$ that does not belong to $\si$, then $\si|_{\SV(\CL)\cap \SV(\si)}$ is a permutation satisfying $\si|_{\SV(\CL)\cap \SV(\si)}\prec \CL$ and 
\begin{equation}
	\frl_{\si}(\tau; \bfz, \bfw)=\frl_{\si|_{\SV(\CL)\cap \SV(\si)}}(\CL; \tbz|_{\SV(\CL)}, \tbw|_{\SV(\CL)})\,.
\end{equation}
\end{lemma}

\begin{proof}
	For $1\leq k\leq |\tau|$, define $\si_k = \si|_{\SV(\CL_k)\cap \SV(\si)}$. By Lemma \ref{lem:warum}, $(\si_k)_{k=1}^{|\tau|}$ are all permutations such that $\si = \biguplus_{k=1}^{|\tau|}\si_k$ and $\si_k\trieq \CL_k$ for $1\leq k\leq |\tau|$. Similarly, for every $\pi\in\CS_\ell$ such that $\si\trieq\pi\prec\tau$ we have $\si_k\trieq\pi_k\preceq \CL_k$ for every $1\leq k\leq |\tau|$, where we denoted $\pi_k = \pi|_{\SV(\CL_k)\cap \SV(\pi)}$. Also, from the definition of the relation $\prec$, we must have a unique $1\leq k_0\leq|\tau|$ such that $\pi_{k_0}\prec \tau_{k_0}$ and $\pi_k=\tau_k$ for every $k\neq k_0$. In particular, recalling the definition of $\FN$, see \eqref{eq:FN}, we obtain for every $\pi\in\CS_\ell$ satisfying $\si\trieq \pi \prec \si$ that 
\[
	\FN_{\pi,\tau} = \FN_{\pi_{k_0},\CL_{k_0}}\,.
\]
	
From the eigenvector equation \eqref{eq:EV} and the last observation
\[
		[\frh_{\si}-\frh_{\tau}]\frl_{\si}(\tau) = \sum_{\si\trieq\pi\prec \tau}\frl_{\si}(\pi)\FN_{\pi,\tau} = \sum_{k=1}^{|\tau|}\sum_{\si_k\trieq \pi \prec \CL_k}\frl_\si(\{\CL_j\}_{j\neq k}\cup \pi)\FN_{\pi,\CL_k}
\]
Using an induction argument over pairs of permutations with respect to the relation $\prec$, we obtain for every $1\leq k\leq |\tau|$
\[
	\frl_\si(\{\CL_j\}_{j\neq k}\cup \pi) = \frl_{\si_k}(\pi)\cdot \prod_{k\neq j=1}^{|\tau|}\frl_{\si_j}(\CL_j)\,,
\]
and therefore 
\[
\begin{aligned}
~		[\frh_{\si}-\frh_{\tau}]\frl_{\si}(\tau) &= \sum_{\si\trieq\pi\prec \tau}\frl_{\si}(\pi)\FN_{\pi,\tau} = \sum_{k=1}^{|\tau|}\prod_{k\neq j=1}^{|\tau|}\frl_{\si_j}(\CL_j)\sum_{\si_k\trieq \pi \prec \CL_k}\frl_{\si_k}(\pi)\FN_{\pi,\CL_k}\\
		& =\sum_{k=1}^{|\tau|}\prod_{k\neq j=1}^{|\tau|}\frl_{\si_j}(\CL_j)\cdot [\frh_{\si_k}-\frh_{\tau_k}]\frl_{\si_k}(\tau_k)\\
		& =\bigg(\sum_{k=1}^{|\tau|}[\frh_{\si_k}-\frh_{\tau_k}]\bigg)\cdot \prod_{j=1}^{|\tau|}\frl_{\si_j}(\CL_j)\,.
\end{aligned}
\]
Noting that the definition of $\frh$ implies 
\[
	[\frh_{\si}-\frh_{\tau}] = \sum_{k=1}^{|\tau|}[\frh_{\si_k}-\frh_{\tau_k}]\,,
\]
the result follows. 
\end{proof}


Next, we argue that it is enough to calculate $\frl_{\texttt{I}_A}(C_\ell)$ for some subset $A\subset [\ell]$, where $\texttt{I}_A$ is the identity permutation on $A$. We start with a few additional notations. 
For $A\subseteq B\subseteq [\ell]$, define $U_A^B : \CS_A \to \CS_B$ by 
\[
	U^B_A(\pi)(i) = \begin{cases}
		\pi(i) & i\in A\\
		i & i\in B\setminus A
	\end{cases}\qquad \forall \pi\in\CS_A\,.
\]	

Similarly, for $A\subseteq B\subseteq [\ell]$, define $L_A^B : \{\pi\in \CS_B ~:~ \pi(i)=i \text{ for all }i\in B\setminus A\}\to \CS_A$ by 
\[
	L_A^B(\pi)(i)=\pi(i),\qquad \forall i\in A \text{ and }\pi\in \CS_B \text{ as above.}
\]

\begin{lemma}[Recursive property of eigenvector components]\label{lem:Recurrence}
	For every $\si,\tau\in\CS_\ell$ such that $\si\trieq \tau$ denote $\si_\tau = U_{\SV(\si)}^{\SV(\tau)}(\si)$. Then, for every $\tbz,\tbw\in \BC^{\SV(\tau)}$
\[
	\frl_{\si}(\tau; \bfz, \bfw)=\frl_{\texttt{I}_{\SV(\si)}}(\tau\circ \si_\tau^{-1}; \bfz, \si_\tau^{-1}(\bfw))\,.
\]
\end{lemma}

\begin{proof}
	We fix $\si\in\CS_\ell$ and prove the result by induction on $\tau\in\CS_\ell$ such that $\si\trieq \tau$ with the order taken with respect to the relation $\prec$. For the base of the induction, $\tau=\si$, we have from the normalization
\[
	\frl_\si(\tau,\tbz,\tbw) = 1 = 	\frl_{\texttt{I}_{\SV(\si)}}(\texttt{I}_{\SV(\si)},\tbz,\tbw) = \frl_{\texttt{I}_{\SV(\si)}}(\tau\circ \si_\tau^{-1},\tbz,\si_\tau^{-1}(\tbw))\,.
\]

Next, let $\tau\in\CS_\ell$ such that $\si\trieq\tau$ and assume the result holds for all $\pi\in\CS_\ell$ such that $\pi \prec \tau$. Then,
\[
\begin{aligned}
	~[\frh_\si-\frh_\tau]\frl_\si(\tau;\tbz,\tbw) &= \sum_{\si\trieq\pi\prec \tau}\frl_{\si}(\pi;\tbz,\tbw)\FN_{\pi,\tau}(\tbz,\tbw)\\
	& = \sum_{\si\trieq\pi\prec \tau}\frl_{\texttt{I}_{\SV(\si)}}(\pi\circ \si_\pi^{-1};\tbz,\si_\pi^{-1}(\tbw))\FN_{\pi,\tau}(\tbz,\tbw)
\end{aligned}
\]
Noting that $\frl_{\texttt{I}_{\SV(\si)}}(\pi\circ \si_\pi^{-1};\tbz,\si_\pi^{-1}(\tbw))$ is in fact only a function of $\tbz|_{\SV(\pi)}$ and $\tbw|_{\SV(\pi)}$ it follows that $\frl_{\texttt{I}_{\SV(\si)}}(\pi\circ \si_\pi^{-1};\tbz,\si_\pi^{-1}(\tbw))=\frl_{\texttt{I}_{\SV(\si)}}(\pi\circ \si_\pi^{-1};\tbz,\si_\tau^{-1}(\tbw))$. Furthermore, from the definition $\FN$, see \eqref{eq:FN}, for $\si\trieq \pi\prec \tau$, we have 
\[
	\FN_{\pi,\tau}(\tbz,\tbw) = \FN_{\pi\circ \si_\pi^{-1},\tau\circ \si_\tau^{-1}}(\tbz,\si_{\tau}^{-1}(\tbw))\,.
\]	

Combining all of the above, we conclude that
\[
\begin{aligned}
	~[\frh_\si-\frh_\tau]\frl_\si(\tau;\tbz,\tbw) &= \sum_{\si\trieq\pi\prec \tau}\frl_{\texttt{I}_{\SV(\si)}}(\pi\circ \si_\pi^{-1};\tbz,\si_\tau^{-1}(\tbw))\FN_{\pi\circ \si_\pi^{-1},\tau\circ \si_\tau^{-1}}(\tbz,\si_{\tau}^{-1}(\tbw))\\
	& = \sum_{\texttt{I}_{\SV(\si)}\trieq\pi\prec \tau\circ \si_\tau^{-1}}\frl_{\texttt{I}_{\SV(\si)}}(\pi;\tbz,\si_\tau^{-1}(\tbw))\FN_{\pi,\tau\circ \si_\tau^{-1}}(\tbz,\si_{\tau}^{-1}(\tbw))\\
	& =[\frh_{\texttt{I}_{\SV(\si)}}(\tbz,\si_\tau^{-1}(\tbw))-\frh_{\tau\circ \si_\tau^{-1}}(\tbz,\si_\tau^{-1}(\tbw))]\frl_{\texttt{I}_{\SV(\si)}}(\tau\circ \si_\tau^{-1};\tbz,\si_\tau^{-1}(\tbw))\,.
\end{aligned}
\]
Noting that $\frh_{\texttt{I}_{\SV(\si)}}(\tbz,\si_\tau^{-1}(\tbw)) = \frh_{\si}(\tbz,\tbw)$ and $\frh_{\tau\circ \si_\tau^{-1}}(\tbz,\si_\tau^{-1}(\tbw))=\frh_\tau(\tbz,\tbw)$ the result follows. 
\end{proof}

\begin{lemma}\label{lem:up_down}
	For every $A\subseteq B\subseteq [\ell]$, every $\si\in \CS_A$ and every $\tau\in \CS_B$, 
\[
	\frl_\si(\tau)=\begin{cases}
(-1)^{|B|-|A|} \frl_\si(\pi) & \text{if }\exists \pi \in \CS_A \text{ such that }\tau = U_A^B(\pi)\\
0 & \text{ otherwise}
\end{cases}\,.
\] 
\end{lemma}

\begin{proof}
	Fix $A\subseteq B\subseteq [\ell]$ and $\si\in\CS_A$. We split the proof into the statement regarding the permutations $\tau$ such that $\tau = U_A^{\SV(\tau)}(\pi)$ for some $\pi\in \CS_A$, and the ones for which no such $\pi$ exists. 
	
Starting with the former, we prove the statement by double induction: First on the size of $B\setminus A$ and second on the permutation $\tau$ (with respect to the ordering $\prec$). If $|B\setminus A|=0$, then $A=B$ and there is nothing to prove. Next, let $A\subset B\subseteq [\ell]$ such that $|B\setminus A|=1$ and denote by $i$ the unique element in $B\setminus A$. 

We start an induction on $\tau\in \CS_B$ such that $\si\trieq \tau$ (with respect to the ordering $\prec$). The base of the induction is given by $\tau = U_A^B(\si)$ for which the eigenvector equation \eqref{eq:EV} yields
\[
\begin{aligned}
- h(z_i, w_i)\frl_{\si}(U_A^B(\si))
	& = \big[\frh_\si -\frh_{U_A^B(\si)}\big] \frl_\si(U_A^B(\si)) 
	\\
	& =\sum_{\si\trieq \pi \prec U_A^B(\si)} \frl_{\si}(\pi) \FN_{\pi,U_A^B(\si)}	= \frl_\si(\si)\FN_{\si,U_A^B(\si)}	=h(z_i,w_i)\,,
\end{aligned}
\]
which shows that $\frl_\si(U_A^B(\si)) = -1 = -\frl_\si(\si)$, due to the chosen normalization of the eigenvectors. Next assume that for every permutation in $\CS_B$ which is in the image of $I_A^B$ and is smaller than $\tau\in \CS_B$ (with respect to $\prec$) the statement holds. By the eigenvector equation \eqref{eq:EV}, for every $\tau\in\CS_A$
\[
\begin{aligned}
\big[\frh_{\si}-\frh_{\tau}- h(z_i, w_i)\big]\frl_{\si}(U_A^B(\tau))
	& = \big[\frh_\si -\frh_{U_A^B(\tau)}\big] \frl_\si(U_A^B(\tau)) 
	=\sum_{\si\trieq \pi \prec U_A^B(\tau)} \frl_{\si}(\pi) \FN_{\pi,U_A^B(\tau)}	\\	
	& = \sum_{\substack{\si\trieq \pi \prec U_A^B(\tau) \\ i\notin \SV(\pi)}} \frl_{\si}(\pi) \FN_{\pi,U_A^B(\tau)} + \sum_{\substack{\si\trieq \pi \prec U_A^B(\tau) \\ i\in \SV(\pi)}} \frl_{\si}(\pi) \FN_{\pi,U_A^B(\tau)}\,.
\end{aligned}
\]
	
Note that the only $\pi\in \CS_A$ satisfying $\pi \prec U_A^B(\tau)$ and $i\notin \SV(\pi)$ is $\pi=\tau$ and therefore, the first sum equals $\frl_\si(\tau) \FN_{\tau,U_A^B(\tau)} = \frl_\si(\tau)h(z_i,w_i)$. As for the second sum, since $\pi\prec U_A^B(\tau)$ and $i\in\SV(\pi)$ we must have that $i$ is also an isolated vertex of $\pi$ and therefore that $L_A^B(\pi)$ is well defined and $L_A^B(\pi)\prec \tau$. Due to \eqref{eq:FN}, we have $\FN_{\pi,U_A^B(\tau)}=\FN_{L_A^B(\pi),\tau}$ and by the induction assumption $\ell_\si(\pi) = -\ell_\si(L_A^B(\pi))$. Therefore
\[
\begin{aligned}
	&\big[\frh_{\si}-\frh_{\tau}- h(z_i, w_i)\big]\frl_{\si}(U_A^B(\tau))\\
	&\qquad   = \frl_{\si}(\tau)h(z_i, w_i)- 	 \sum_{\si \trieq L_A^B(\pi) \prec \tau}\frl_{\si}(L_A^B(\pi))\frN_{L_A^B(\pi),\tau}
 = \frl_\si(\tau)\frh(z_i,w_i) - [\frh_\si-\frh_\tau]\ell_\si(\tau) \\
 &\qquad = [h_\si-h_\tau-h(z_i,w_i)]\ell_\si(\tau)\,,
\end{aligned}
\]
and therefore that $\frl_\si(U_A^B(\tau))=-\frl_\si(\tau)$. This complete the induction over $\tau$ and thus the proof for the case $|B\setminus A|=1$. Noting that for $A\subset B\subset C$ we have $U_A^C(\tau) = U_B^C(U_A^B(\tau))$, one can repeat the previous argument adding one vertex at a time, thus completing the induction over $|B\setminus A|$, and thus the proof of the first part. 

Next, we turn to deal with these $\tau$ for which there is no $\pi\in\CS_A$ such that $\tau = I_A^{\SV(\tau)}(\pi)$. As before, this is done by induction on the permutation with respect to the relation $\prec$. Let $\tau$ be a permutation with $A\subset \SV(\tau)$ and denote $B=\SV(\tau)$. Since $\tau$ is not of the form $I_A^B(\pi)$ for some $\pi\in\CS_A$, it follows that, there exists $i\in B\setminus A$ such that $\tau(i)\neq i$. Furthermore, due to the first part, we can assume without loss of generality that the set of fixed points of $\tau$ is contained in $A$.

The eigenvector equation reads
\[
	[\frh_{\si}-\frh_{\tau}]\frl_\si(\eta)
	= \sum_{\si\trieq \pi \prec\tau} \frl_{\si}(\pi) \frN_{\pi, \tau}
\]  
By the induction assumption, for every $\si\trieq \pi\prec \tau$ such that $\exists j\in \SV(\pi)$ and $\pi(j)\neq j$, we have $\frl_\si(\pi)=0$. Thus we can restrict the sum to $E_{\si,\tau} = \{\si\trieq \pi\prec \tau ~:~ \pi(j)= j~\forall j \in \SV(\pi) \setminus A\}$. In other words
\[
	[\frh_{\si}-\frh_{\tau}]\frl_\si(\tau)=\sum_{\pi\in E_{\si,\tau}}  \frl_{\si}(\pi) \frN_{\pi, \tau} = \sum_{\eta\in \CS_A}\sum_{\substack{\pi\in E_{\si,\tau}\\ \pi=U_A^{\SV(\pi)}(\eta)}}\frl_\si(\pi)\FN_{\pi,\tau}.
\]
The key point is that for every $\eta\in \CS_A$ the values of $\FN_{\pi,\tau}$ for $\pi\in E_{\si,\tau}$ such that $\pi=U_A^{\SV(\pi)}(\eta)$ are all the same. Indeed, from the properties of the relation $\prec$ (see Corollary \ref{cor:perc}) and the fact that $\tau$ has no fixed points in $B\setminus A$, each vertex in $\pi=U_A^{\SV(\pi)}(\eta)$ is either in a cycle of $\eta$ (which is also a cycle in $\pi$ and $\tau$, a vertex from $B\setminus A$ that belongs to $\tau$ and not to $\pi$ or a vertex from $B\setminus A$ which belongs to $\tau$ and to $\pi$, in which case it is a fixed point of $\pi$ but not of $\tau$. In all the cases above, we obtain that $\widehat{\SV}_{bf}(\pi;\tau)=B\setminus A$, and thus that $\FN_{\si,\tau}$ are all the same. 

Denoting by $K_{\eta,\tau}$ the common value $\FN_{\pi,\tau}$ for $\pi\in E_{\si,\tau}$ such that $\pi = U_A^{\SV(\pi)}(\eta)$, we conclude that
\[
\begin{aligned}
~	[\frh_\si-\frh_\tau]\frl_\si(\tau) &= \sum_{\eta\in\CS_A} K_{\eta,\tau} \sum_{\substack{\pi\in E_{\si,\tau}\\ \pi=U_A^{\SV(\pi)}(\eta)}}\frl_\si(\pi)
	= \sum_{\eta\in\CS_A} K_{\eta,\tau} \sum_{\substack{\pi\in E_{\si,\tau}\\ \pi=U_A^{\SV(\pi)}(\eta)}}(-1)^{|\SV(\pi)|-|A|}\frl_\si(\eta)\\
	& = \sum_{\eta\in\CS_A} K_{\eta,\tau}\frl_\si(\eta) \sum_{A\subseteq C\subseteq B}(-1)^{|C|-|A|}
\end{aligned}
\]
Noting that 
\[
	\sum_{A\subseteq C\subseteq B}(-1)^{|C|-|A|} = \sum_{k=0}^{|B\setminus A|}\binom{|B\setminus A|}{k}(-1)^k = (-1+1)^{|B\setminus A|} = 0\,,
\]
the result follows. 
\end{proof}

We now combine the last three lemmas. Given $\si,\tau\in \CS_\ell$ such that $\si\trieq \tau$, by Lemma \ref{lem:up_down}, we have $\frl_\si(\tau)=0$ unless $\tau= U_{\SV(\si)}^{\SV(\tau)}(\pi)$ for some $\pi\in \CS_{\SV(\si)}$ which satisfy $\si\trieq \pi$. 

Given a cycle $\CL\in\pi$, denote by $\si_\CL = \si|_{\SV(\CL)}$ the restriction of $\si$ to the vertices of $\CL$. Since $\si\trieq\pi$, each of the functions $\si_\CL$ is a permutation satisfying $\si_\CL\trieq \CL$ and $\SV(\si_\CL) = \SV(\CL)$. Hence, by Lemma \ref{L:Tens} 
\begin{align}\label{eq:well_ok}
	(-1)^{|\SV(\tau)|-|\SV(\si)|}\frl_\si(\tau;\tbz,\tbw) &= \frl_\si(\pi;\tbz,\tbw) 
	= \prod_{\CL\in\pi} \frl_{\texttt{I}_{\SV(\si_\CL)}}(\CL\circ \si_\CL^{-1};\tbz|_{\SV(\CL)},\si_\CL^{-1}(\tbw|_{\SV(\CL)}))\nonumber\\
	&= \prod_{\CL\in\pi} \prod_{\CL'\in \CL\circ \si_\CL^{-1}}\frl_{\texttt{I}_{\SV(\CL')}}(\CL';\tbz|_{\SV(\CL')},(\si_\CL^{-1}(\tbw|_{\SV(\CL)})|_{\SV(\CL')}))\,.
\end{align}

Since each of the terms on the right hand size is (up to reindexing) of the form $\frl_{\texttt{I}_{[\ell]}}(C_l;\tbz,\tbw)$ for some choice of $\ell\in\BN$ and $\tbz,\tbw\in\BC^\ell$, the only equations that are left to be solved are (for every $\ell\in\BN$)
\begin{align}&\label{E:fun}
[\frh_{\mathtt{I}_{[\ell]}}-\frh_{C_\ell}]\frl_{\mathtt{I}_{[\ell]}}(C_\ell;\bfz, \bfw)= \sum_{\mathtt{I}_{[\ell]}\trieq \sigma\prec C_\ell } \frl_{\mathtt{I}_{[\ell]}} (\sigma;\bfz, \bfw) \frN_{\sigma, C_\ell},
\end{align}


Recall that for $A\subset [\ell]$ and $\bfz\in \C^A$, we denote by $\mathsf{V}_A(\bfz)$ the Vandermonde determinant $\prod_{\substack{\alpha, \beta\in A\\ \alpha< \beta}}(z_{\beta}-z_{\alpha})$ and for a permutation $\si\in\CS$ and $\tbz\in\BC^{\SV(\si)}$, we define $\mathsf{V}_\si(\tbz) = \prod_{\CL\in\si}\mathsf{V}_{\SV(\CL)}(\tbz|_{\SV(\CL)})$

\begin{lemma}\label{T:Comp1}
	For every $\ell\in\BN$ and $\si\in \CS_\ell$, the function $(\tbz,\tbw)\mapsto \frl_{\mathtt{I}_{[\ell]}}(\si; \bfz, \bfw)$, for $\tbz,\tbw\in \BC^{\SV(\si)}$ is a rational function in the variables $\overline{\tbz}$ and $\tbw$ (and in particular does not depend on $\tbz$ and $\overline{\tbw}$). Moreover $\mathsf{V}_{\si}(\bar{\bfz})\mathsf{V}_{\si}(\bfw) \frl_{\textrm{I}_{\SV(\si)}}(\si; \bfz, \bfw)$ is a homogeneous polynomial of degree ${|\SV(\si)|-1 \choose  2}$ in $\bar{\bfz}$ and $\bfw$. In addition, for $\ell\geq 3$, the polynomial $\mathsf{V}_{[\ell]}(\bar{\bfz})\mathsf{V}_{[\ell]}(\bfw) \frl_{\textrm{I}_{[\ell]}}(C_\ell; \bfz, \bfw)$ vanishes on the complex hyperplanes
\[
	\{(\bfz, \bfw)\in(\BC^{[\ell]})^2 ~:~ (z_\alpha, w_{\alpha})= (z_\beta, w_{\beta})\},\qquad \forall \alpha,\beta\in [\ell] ~\text{ such that }\alpha\neq \beta\,.
\]
\end{lemma}

We postpone the proof of Lemma \ref{T:Comp1} to Subsection \ref{subsec:proof_of_lemma_comp1}. 

\begin{remark}
	Let $\FM^{\nu}$ denote the matrix defined by disintegrating $\frN$. That is $\FM$ is the matrix defined by 
\[
	\FM^\nu_{\si,\tau} = \frac{1}{\pi}\prod_{\alpha\in \widehat{\SV}_{nf}(\si;\tau)} \frac{1}{(\overline{\nu}-\overline{u_{\alpha}})} \frac{1}{(\nu-v_{\sigma^{-1}(\alpha)})}\qquad \forall \si\prec \tau\,,
\]
\[
	\FM^\nu_{\si,\si} = \sum_{\alpha\in \SV(\si)}\frac{1}{(\overline{\nu}-\overline{z_\alpha})(\nu-w_{\si^{-1}(\alpha)})}\,,
\]
and $0$ otherwise. An immediate corollary from Lemma \ref{T:Comp1} is that for every choice of $\nu, \nu'\in \C$, it holds that $[\FM^{\nu}, \FM^{\nu'}]:=\FM^\nu\FM^{\nu'}-\FM^{\nu'}\FM^{\nu}=0$.

The statement is not as miraculous as it might first appear. Since the left-hand side of \eqref{eq:cond_expec_for_F_N} is a conditional expectation over the sequence of eigenvalues $(\lambda_i)$ and the sequence of eigenvalues is exchangeable the right-hand side must be symmetric in $(\lambda_i)$. This strongly suggests that already at the finite level of $N$ by $N$ matrices $[A^{\lambda}, A^{\lambda'}]=0$. In a forthcoming paper we will take a more careful look at this and use it to connect local correlations  to the representation theory of the symmetric group $\CS_{[\ell]}$.
\end{remark} 

Before turning to the proof of Lemma \ref{T:Comp1} and the main results, we provide a summary of the properties the left eigenvectors satisfy and state the analogue result for right eigenvectors. 

\paragraph{Left eigenvectors - Summary} Let $\si,\tau\in\CS$ such that $\si\trieq \tau$ and $\tau=\{\CL_k\}_{k=1}^{|\tau|}$. Then 
\begin{center}
\begin{tabular}{c|c|c}
Tensorial property$\quad$  & $\frl_{\si}(\tau; \bfz, \bfw)=\prod_{k=1}^{|\tau|} \frl_{\si|_{\SV(\CL_k)}}(\CL_k; \tbz|_{\SV(\CL_k)}, \tbw|_{\SV(\CL_k)})$ & - \\
&&\\
Recursive property & $\frl_{\si}(\tau; \bfz, \bfw)=\frl_{\texttt{I}_{\SV(\si)}}(\tau\circ \si_\tau^{-1}; \bfz, \si_\tau^{-1}(\bfw))$ & - \\
&&\\
Lifting property & $\frl_\si(\tau) = (-1)^{|\SV(\tau)|-|\SV(\pi)|}\frl_\si(\pi)$ & $\tau = U_{\SV(\pi)}^{\SV(\tau)}(\pi)$\\
& $\frl_\si(\tau)=0$ & otherwise\\
& &\\

\end{tabular}
\end{center}
Furthermore, $\frl_\si(\tau;\tbz,\tbw)$ is a rational function in $\overline{\tbz}$ and $\tbw$ such that $\mathsf{V}_{\SV(\si)}(\overline{\tbz})\mathsf{V}_{\SV(\si)}(\tbw)\frl_\si(\tau;\tbz,\tbw)$ is a homogeneous polynomial in $\overline{\tbz}$ and $\tbw$ of degree $\binom{|\SV(\si)|-1}{2}$. 

\paragraph{Right eigenvectors - Summary} Let $\si,\tau\in\CS$ such that $\si\trieq \tau$ and $\tau=\{\CL_k\}_{k=1}^{|\tau|}$. Then 
\begin{center}
\begin{tabular}{c|c|c}
Tensorial property$\quad$  & $\frr_{\tau}(\si; \bfz, \bfw)=\prod_{k=1}^{|\tau|} \frr_{\CL_k}(\si|_{\SV(\CL_k)}; \tbz|_{\SV(\CL_k)}, \tbw|_{\SV(\CL_k)})$ & - \\
&&\\
Recursive property & $\frr_{\si}(\tau; \bfz, \bfw)=\frr_{\tau\circ \si_\tau^{-1}}(\texttt{I}_{\SV(\si)}; \bfz, \si_\tau^{-1}(\bfw))$ & - \\
&&\\
Lifting property & $\frr_\tau(\si) = \frr_\tau(\pi)$ & $\pi = U_{\SV(\sigma)}^{\SV(\pi)}(\si)$\\
& $\frr_\tau(\si)=0$ & otherwise\\
& &\\

\end{tabular}
\end{center}
Furthermore, $\frr_\tau(\si;\tbz,\tbw)$ is a rational function in $\overline{\tbz}$ and $\tbw$ such that $\mathsf{V}_{\SV(\si)}(\overline{\tbz})\mathsf{V}_{\SV(\si)}(\tbw)\frr_\tau(\si;\tbz,\tbw)$ is a homogeneous polynomial in $\overline{\tbz}$ and $\tbw$ of degree $\binom{|\SV(\si)|-1}{2}$. 

\subsection{Proof of the main theorems}

	Let $\si\in\CS$ and $\tbu,\tbv\in \tbD_1^{\SV(\si)}$ such that $\mathrm{Dist}(\tbu,\tbv)>0$. Recall that 
\[
	\rho(\si;\tbu,\tbv) = \lim_{\eps \rightarrow 0} \lim_{N\rightarrow \infty}\E[\widehat{\rho}_{N,\eps}(\sigma; \tbu,\tbv)]
\]
and  
\begin{equation}\label{eq:Relaton_of_rho_and_F_2}
	\widehat{\rho}_{N, \eps}(\sigma; \tbu,\tbv)=\frac{1}{(2\pi i)^{2|\SV(\sigma)|}N^{|\sigma|}\varepsilon^{4|\SV(\sigma)|}}\int_{\caS_\varepsilon^{2\SV(\sigma)}(\tbu,\tbv)}  F_N(\sigma; \bfz, \bfw)\textrm{d}\overline{\bfz} \textrm{d}\bfw\,.
\end{equation}
Using Theorem \ref{thm:limit_exists} together with the fact that the convergence is uniform provided $\mathrm{Dist}(\tbu,\tbv)>0$, we conclude that 
\[
	\rho(\si;\tbu,\tbv)=\lim_{\varepsilon\to 0}\frac{1}{(2\pi i)^{2|\SV(\sigma)|}\varepsilon^{4|\SV(\sigma)|}}\int_{\caS_\varepsilon^{2\SV(\sigma)}(\tbu,\tbv)} (\exp(\frN))_{\emptyset,\si}\textrm{d}\overline{\bfz} \textrm{d}\bfw\,.
\]

Using the complex form of Green's theorem, we conclude that the limit $\varepsilon\to 0$ exists and 
\[
	\rho(\si,\tbu,\tbv) = \partial_{\tbu}\partial_{\overline{\tbv}}(\exp(\FN))_{\emptyset,\si}(\tbu,\tbv)\,.
\]
This completes the proof for the existence of $\rho(\si,\tbu,\tbv)$. 

In order prove the factorization property stated in Theorem \ref{T:1} as well as its representation given in Theorem \ref{T:2}, we will use the properties of the basis eigevectors $(\tbr_\pi)_{\pi\in\CS}$ and $(\tbl_\pi)_{\pi\in\CS}$. By Lemma \ref{lem:warum} and the spectral decomposition, for every $\si,\tau\in\CS$ and $t>0$ we can write
\[
	(e^{t\frN})_{\si,\tau}(\tbu,\tbv)=\sum_{\si\trieq \pi\trieq \tau} \frr_{\pi}(\si)e^{t \frh_{\pi}} \frl_{\pi}(\tau)
	=\sum_{\si\trieq \pi\trieq \tau} \frr_{\pi}(\si)\prod_{\alpha\in \SV(\pi)}\bigg(\frac{1-\overline{u_\alpha}v_{\pi^{-1}(\alpha)}}{|u_\alpha-v_{\pi^{-1}(\alpha)}|^2}\bigg)^t \frl_{\pi}(\tau)\,.
\]
and in particular
\[
	(e^{\frN})_{\emptyset,\si}(\tbu,\tbv)
	=\sum_{\pi\trieq \si} \frr_{\pi}(\emptyset)\frl_{\pi}(\si)\prod_{\alpha\in \SV(\pi)}\frac{1-\overline{u_\alpha}v_{\pi^{-1}(\alpha)}}{|u_\alpha-v_{\pi^{-1}(\alpha)}|^2} \,.
\]

From Lemma \ref{T:Comp1}, we know that $\tbr_\pi(\emptyset)$ and $\tbl_\pi(\si)$ (as functions of $\tbu$ and $\tbv$) are rational functions in $\overline{\tbu}$ and $\tbv$ alone. This points now becomes crucial, as it allows us to calculate the partial derivatives of $(e^{\frN})_{\emptyset,\si}(\tbu,\tbv)$ 
\begin{align}\label{eq:Imsotired}
	\rho(\si,\tbu,\tbv) & = 
	 \partial_{\tbu}\partial_{\overline{\tbv}}(e^{\frN})_{\emptyset,\si}(\tbu,\tbv)\nonumber \\
	& =\sum_{\pi\trieq \si} \frr_{\pi}(\emptyset)\frl_{\pi}(\si)\partial_{\tbu}\partial_{\overline{\tbv}}\bigg(\prod_{\alpha\in \SV(\pi)}\frac{1-v_{\pi^{-1}(\alpha)}\overline{u_\alpha}}{|u_\alpha-v_{\pi^{-1}(\alpha)}|^2}\bigg)\\
 	& \overset{(1)}{=}\sum_{\substack{\pi\trieq \si \\ \SV(\pi)=\SV(\si)}} \frr_{\pi}(\emptyset)\frl_{\pi}(\si)\prod_{\alpha\in \SV(\pi)}\rho_2(v_{\pi^{-1}(\alpha)},u_\alpha)\nonumber\\
 	 	& \overset{(2)}{=}\sum_{\substack{\pi\trieq \si \\ \SV(\pi)=\SV(\si)}} \frr_{\pi}(\emptyset)\frl_{\texttt{I}_{\SV(\si\circ \pi^{-1})}}(\si\circ \pi^{-1})\prod_{\alpha\in \SV(\pi)}\rho_2(v_{\pi^{-1}(\alpha)},u_\alpha)\nonumber\,,
\end{align}
where in $(1)$ we used the fact that the differentiation equals zero\footnote{the differentiation equals zero, since we differentiate with respect to at least one coordinate on which the function does not depend} whenever $\SV(\pi)\neq \SV(\si)$ and in $(2)$ we used Lemma \ref{lem:Recurrence}. Note that in second and third line one should understand $\frr_\pi(\emptyset)$ and $\frl_\pi(\si)$ as a function of $\overline{\tbu}|_{\SV(\pi)}$ and $\overline{\tbv}|_{\SV(\pi)}$ while in the last line of \eqref{eq:Imsotired}, one should understand $\frr_\pi(\emptyset)$ as a function of $\overline{\tbu}|_{\SV(\pi)}$ and $\tbv|_{\SV(\pi)}$, while $\frl_{\texttt{I}_{\SV(\si\circ \pi^{-1})}}(\si\circ \pi^{-1})$ is a function of $\overline{\tbu}|_{\SV(\pi)}$ and $\pi^{-1}(\tbv|_{\SV(\pi)})$. 

Setting 
\[
	\frQ_\pi :=\mathsf{V}_{\SV(\pi)}(\overline{\tbu})\mathsf{V}_{\SV(\pi)}({\tbv})\frl_{\texttt{I}_{\SV(\pi)}}(\pi)\qquad \text{and}\qquad\frp_\pi :=\mathsf{V}_{\SV(\pi)}(\overline{\tbu})\mathsf{V}_{\SV(\pi)}({\tbv})\frr_\pi(\emptyset)\,,
\] 
we obtain
\[
	\rho(\si;\tbu,\tbv) =\frac{1}{\mathsf{V}_{\SV(\si)}(\overline{\tbu})^2\mathsf{V}_{\SV(\si)}({\tbv})^2}\sum_{\substack{\pi\trieq \si \\ \SV(\pi)=\SV(\si)}} \frp_{\pi}\frQ_{\si\circ \pi^{-1}}\prod_{\alpha\in \SV(\pi)}\rho_2(u_\alpha,v_{\pi^{-1}(\alpha)})\,.
\]
Lemma \ref{T:Comp1} now completes the proof of Theorem \ref{T:2}.

The factorization stated in Theorem \ref{T:1} follows from \eqref{eq:Imsotired} and the tensorial property of the left and right eigenvectors, see Lemma \ref{L:Tens}. Hence the proof of Theorem \ref{T:1} is complete.

Finally, Theorem \ref{thm:T3} follows immediately from the definitions of $\frQ_\pi$ and $\frp_\pi$ together with Lemma \ref{T:2}.\qed


\subsection{Proof of Lemma \ref{T:Comp1}}
\label{subsec:proof_of_lemma_comp1}

We prove the statement via induction on $\ell$, starting with the statement that $\frl_{\texttt{I}_{[\ell]}}(C_\ell,\tbz,\tbw)$ is a rational function in $\overline{\tbz}$ and $\tbw$. Due to \eqref{eq:well_ok}, proving the statement for $\frl_{\texttt{I}_{[m]}}(C_m)$ for $m<\ell$, implies that $\frl_{\si}(\tau)$ is a rational function in $\overline{\tbz}$ and $\tbw$ for every $\si,\tau\in\CS_\ell$ such that $\si\trieq \tau\prec C_\ell$. Note that for pairs $\si,\tau\in\CS_\ell$ for which $\si\trieq\tau$ does not hold we have $\frl_\si(\tau)=0$ and hence the result holds trivially. 

Using the abbreviation $\frl(\si) = \frl_{\texttt{I}_{[\ell]}}(\si;\tbz,\tbw)$ and denoting $a(w)=(\lambda-w)^{-1}$, we can rewrite \eqref{E:fun} as
\[
\begin{aligned}
	&\frac{1}{\pi}\int_{\tbD_1}  \sum_{\alpha\in [\ell]} [a(w_\alpha)\overline{a(z_\alpha)} - a(w_{C_\ell^{-1}(\alpha)})\overline{a(z_\alpha)}]\text{d}^2\lambda\cdot \frl(C_\ell) = 	[\frh_{\mathtt{I}_{[\ell]}}-\frh_{C_\ell}]\frl(C_\ell)\\
		=& \sum_{\mathtt{I}_{[\ell]}\trieq \sigma\prec C_\ell} \frl(\si) \frN_{\sigma, C_\ell}
		=\frac{1}{\pi}\int_{\tbD_1} \sum_{\mathtt{I}_{[\ell]}\trieq \sigma\prec C_\ell} \frl(\si)\prod_{\alpha\in \SV_{nf}(\si)}a(w_{\si^{-1}(\alpha)})\overline{a(z_\alpha)}\text{d}^2\lambda.
\end{aligned}
\]

Denote $b(z) = (\kappa-\overline{z})^{-1}$ and define the rational functions $\widetilde{P}^\ell$ and $\widetilde{Q}^\ell$ in $\lambda,\kappa\in\BC$ by 
\[
	\widetilde{P}^\ell(\lambda,\kappa) = \sum_{\alpha\in [\ell]} [a(w_\alpha)b(z_\alpha) - a(w_{C_\ell^{-1}(\alpha)})b(z_\alpha)]
\]	
and
\[
	\widetilde{Q}^\ell(\lambda,\kappa) = \sum_{\mathtt{I}_{[\ell]}\trieq \sigma\prec C_\ell} \frl(\si)\prod_{\alpha\in \SV_{nf}(\si)}a(w_{\si^{-1}(\alpha)})b(z_\alpha)\,.
\]
In terms of $\widetilde{P}^\ell$ and $\widetilde{Q}^\ell$ the eigenvector equation reads
\[
	\frac{1}{\pi}\int_{\tbD_1} \widetilde{P}^\ell(\lambda,\overline{\lambda}) \frf(C_\ell) - \widetilde{Q}(\lambda,\overline{\lambda}) \text{d}^2\lambda=0\,.
\]

Assume we can prove that $\widetilde{T}^\ell:=\widetilde{Q}^\ell/\widetilde{P}^\ell$ is independent of $\lambda$ and $\kappa$. Then, $\frl(C_\ell) = \widetilde{T}^\ell(\lambda,\overline{\lambda})$ is a solution to \eqref{E:fun}, which is also the unique solution as shown in Lemma \ref{lem:exist_uniq}. Furthermore, due to the induction assumption, in this case we obtain that $\frl(C_\ell)$ is a rational function (as a quotient of two such functions) in $\overline{\tbz}$ and $\tbw$. 

We thus turn to prove that $\widetilde{T}^\ell$ is independent of $\lambda$ and $\kappa$. To this end, let us introduce some additional notation. For $A\subset [\ell]$, denote $\Fp^{\tbw}_A(\lambda) = \prod_{\alpha\in A}(\lambda- w_\alpha)$ and $\Fq^\tbz_A(\kappa) = \prod_{\alpha\in A}(\kappa - \overline{z_\alpha})$ and let
\[
	P^\ell(\lambda,\kappa) = \mathsf{V}_{[\ell]}(\overline{\tbz})\mathsf{V}_{[\ell]}(\tbw)\widetilde{P}^\ell(\lambda,\kappa)\qquad ,\qquad Q^\ell(\lambda,\kappa) = \mathsf{V}_{[\ell]}(\overline{\tbz})\mathsf{V}_{[\ell]}(\tbw)\widetilde{Q}^\ell(\lambda,\kappa)\,.
\]
Since $\widetilde{Q^\ell}/\widetilde{P^\ell} = Q^\ell/P^\ell$, it suffices to show that $Q^\ell/P^\ell$ is independent of $\lambda$ and $\kappa$. 

Note that 
\[
	P^\ell(\lambda,\kappa) = \sum_{\alpha\in [\ell]} (\Fp^\tbw_{[\ell]\setminus \alpha}(\lambda)\Fq^\tbz_{[\ell]\setminus \alpha}(\kappa) - \Fp^\tbw_{[\ell]\setminus C_\ell^{-1}(\alpha)}(\lambda)\Fq^\tbz_{[\ell]\setminus \alpha}(\kappa))
\]
is a polynomial in $\lambda$ and $\kappa$ whose degree in each of the variables is $\leq \ell -1$. In fact, from the cyclic behavior of $C_\ell$, it follows that the coefficient of $\lambda^{\ell-1}\kappa^j$ and $\lambda^j\kappa^{\ell-1}$ for $0\leq j\leq \ell-1$ all vanish, and thus that $P(\lambda,\kappa)$ is in fact polynomial in $\lambda$ and $\kappa$ whose degree in each of the variables is at most $\ell -2$. 

For $Q^\ell$, we have
\[
	Q^\ell(\lambda,\kappa) = \sum_{\texttt{I}_{[\ell]}\trieq \si\prec C_\ell} \frl(\si) \Fp^\tbw_{[\ell] \setminus \si^{-1}(\SV_{nf}(C_\ell \circ \si^{-1}))}(\lambda)\Fq^\tbz_{[\ell] \setminus \SV_{nf}(C_\ell \circ \si^{-1})}(\kappa)\,,
\]
and since the sum is over permutations satisfying $\si\prec C_\ell$, which implies $|\SV_{nf}(C_\ell \circ \si^{-1})|\geq 2$, it follows that $Q(\lambda,\kappa)$ is also a polynomial in $\lambda$ and $\kappa$, whose degree in each of the coefficients is at most $\ell-2$. 

Using the fact that $P^\ell$ and $Q^\ell$ are both polynomials of degree $\leq (\ell-2)^2$ in the variables $\lambda$ and $\kappa$, it follows that one needs to specify $(\ell-1)^2$ coefficients in order to determine $Q^\ell$ and $P^\ell$. Similarly, one need to specify $(\ell-1)^2-1=\ell(\ell-2)$ coefficients in order to determine $Q^\ell$ and $P^\ell$ up to a multiplicative constant. Thus, the claim will follow if we can show that both $P^\ell$ and $Q^\ell$ vanish on common $(\ell-1)^2-1 = \ell(\ell-2)$ points in $\BC^2$.

One can check by a direct computation that for any generic choice of $\tbz$ and $\tbw$, the polynomial $P^\ell$ vanishes on all points of the form $(\lambda,\kappa) = (w_\gamma,\overline{z_\beta})$ for $\gamma\in [\ell]$ and $\beta\in [\ell]\setminus \{\gamma,C_\ell(\gamma)\}$, which are precisely $\ell(\ell-2)$ points in $\BC^2$. We now turn to show that $Q^\ell$ also vanishes on these points as well. The proof follows an induction on $\ell$. 

For $\ell = 1,2$ there is nothing to prove and for $\ell=3$, one can verify that 
\[
\begin{aligned}
	Q^3(\lambda,\kappa) &= 1 + \sum_{\alpha\in [3]} \Fp_{\alpha}^\tbw\Fq^\tbz_{C_3^{-1}(\alpha)}\cdot \frl((\alpha,C_3(\alpha))\\
	& = 1 - \sum_{\alpha\in [3]} \frac{(\lambda-w_\alpha)(\kappa-\overline{z_{C_3(\alpha)}})}{(w_\alpha - w_{C_3(\alpha)})(\overline{z_\alpha}-\overline{z_{C_3(\alpha)}})}\,
\end{aligned}
\]
which vanishes for $\lambda=w_\gamma$ and $\kappa=\overline{z_\beta}$ with $\gamma,\beta\in [3]$ such that $\beta\notin {\gamma,C_\ell(\gamma)}$. 

Assume next that the result holds for all integers strictly smaller than $\ell$. Note that for $\mathtt{I}_{[\ell]}\trieq \si\prec C_\ell$, it holds that $\Fp^\tbw_{[\ell] \setminus \si^{-1}(\SV_{nf}(C_\ell \circ \si^{-1}))}(w_\gamma)\neq 0$ if and only if $\si^{-1}\circ C_\ell(\gamma)\neq \gamma$. Similarly, $\Fq^\tbz_{[\ell] \setminus \SV_{nf}(C_\ell \circ \si^{-1})}(z_\beta)\neq 0$ if and only if $C_\ell \circ \si^{-1}(\beta)=\beta$. 

Consequently
\[
	Q^\ell(w_\gamma,\overline{z_\beta}) = \hspace{-0.2cm}
	\sum_{\substack{\texttt{I}_{[\ell]}\trieq \si\prec C_\ell \\ \si^{-1}\circ C_\ell (\gamma)\neq \gamma,\, C_\ell\circ \si^{-1}(\beta)\neq \beta}} \hspace{-0.6cm}\frl(\si) \Fq^\tbz_{[\ell] \setminus \SV_{nf}(C_\ell \circ \si^{-1})}(\overline{z_\beta})\Fp^\tbw_{[\ell] \setminus \si^{-1}(\SV_{nf}(C_\ell \circ \si^{-1}))}(w_\gamma)\,.
\]

Denote by $J=J_{\beta,\gamma}^\ell$ the subinterval of $C_\ell$ starting in $\beta$ and ending with $\gamma$ with respect to the ordering in $C_\ell$ and denote by $C_J$ the cycle on $J_{\beta,\gamma}^\ell$ induced from $C_\ell$. Similarly, denote by $I=I_{\beta,\gamma}^\ell$ the subinterval of $C_\ell$ starting in $C_\ell(\gamma)$ and ending with $C_\ell^{-1}(\beta)$ and let $C_I$ the cycle on $I_{\beta,\gamma}^\ell$ induced from $C_\ell$. Using the tensorial property of $\frl(\si)$, see Lemma \ref{L:Tens}, and the definition of the polynomials $\Fp$ and $\Fq$, it follows that\footnote{Observe that we have $\preceq$ in both sums and not $\prec$ as in the eigenvector equation.} 
\[
\begin{aligned}
& Q^\ell(w_\gamma,\overline{z_\beta}) = \bigg(
	\sum_{\texttt{I}_{J}\trieq \tau\preceq C_J } \frl(\tau) \Fq^\tbz_{J \setminus (\SV_{nf}(C_J \circ \tau^{-1})\cup \beta)}(\overline{z_\beta})\Fp^\tbw_{J \setminus (\tau^{-1}(\SV_{nf}(C_J \circ \tau^{-1}))\cup \gamma)}(w_\gamma)\bigg)\times\\
	&\qquad \Bigg(\hspace{-0.2cm}
	\sum_{\substack{\texttt{I}_{I}\trieq \pi\preceq C_I \\ C_J\circ\pi^{-1}\circ C_\ell(\gamma)\neq C_\ell(\gamma)\\ \pi^{-1}\circ C_J\circ C_\ell^{-1}(\beta)\neq C_\ell^{-1}(\beta)}} \hspace{-0.6cm}\frl(\si) \Fq^\tbz_{I \setminus \SV_{nf}(C_I \circ \pi^{-1})\cup \{C_\ell^{-1}(\beta)\}}(\overline{z_\beta})\Fp^\tbw_{I \setminus \pi^{-1}(\SV_{nf}(C_I \circ \pi^{-1}))\cup \{C_\ell(\gamma)\}}(w_\gamma)\Bigg)\,.
\end{aligned}
\]

Denoting $|J|=j$, we observe that the first sum equals, up to a permutation on the indexes taking the interval $J$ to the interval $\{1,2,\ldots,|J|\}$, to
\[
	Q^j(w_j,\overline{z_1}) + \frl(C_j)\Fq^\tbz_{[j]}(\overline{z_1})\Fp^\tbw_{[j]}(w_j) = Q^j(w_j,\overline{z_1}) - \frl(C_j)P^j(w_j,\overline{z_1})=0,
\]
where in the last two equalities we used the fact that $2\leq j\leq \ell - 1$ due to the assumption that $\si\prec C_\ell$ and the induction assumption. 

This completes the proof of the fact that $\frl_{\texttt{I}_{[\ell]}}(C_\ell)$ is a rational function of $\overline{\tbz}$ and $\tbw$ for every $\ell\in\BN$, and hence, so are $\frl_\si(\tau)$ for every $\si,\tau\in\CS$ are rational functions in $\overline{\tbz}$ and $\tbw$. 

Next we turn to show that $\mathsf{V}_{\si}(\bar{\bfz})\mathsf{V}_{\si}(\bfw) \frl_{\textrm{I}_{[\ell]}}(\si; \bfz, \bfw)$ are  homogeneous polynomial of degree ${|\SV(\si)|-1 \choose  2}$ in $\bar{\bfz}$ and $\bfw$.

We start by introducing some necessary notation. Recall that for a permutation $\si\in\CS_{[\ell]}$ and $\tbz\in\BC^\ell$, we define $\mathsf{V}_{\si}(\tbz) = \prod_{\CL\in \si}\mathsf{V}_{\SV(\CL)}(\tbz)$, where the product is taken over all cycles in $\si$ and we used the convention that $\mathsf{V}_{i}(\tbz)=1$ for $i\in [\ell]$. Furthermore, for $\si\in \CS_{[\ell]}$, let 
\[
	\CW_\si\equiv \CW_\si(\overline{\tbz},\tbw) = \frac{\sfV_{[\ell]}(\overline{\tbz})\sfV_{[\ell]}(\tbw)}{\sfV_{\si}(\overline{\tbz})\sfV_{\si}(\tbw)}\,
\]
and
\[
	\Fe_\si\equiv \Fe_\si(\tbz,\tbw) = \sfV_\si(\overline{\tbz})\sfV_\si(\tbw) \frl_{\texttt{I}_{[\ell]}}(\si)\,.
\]

Fix $\alpha\in [\ell]$. Recalling that $\frl_{\texttt{I}_{[\ell]}}(C_\ell) = Q^\ell(\lambda,\kappa)/P^\ell(\lambda,\kappa)$ for every $\lambda,\kappa\in\BC$, by substituting $\kappa = w_\alpha$ and $\kappa=\overline{z_\alpha}$ and using the notation above, we obtain
\begin{equation}\label{eq:Recurssive_poly}
\begin{aligned}
	&\Fe_{C_\ell}\cdot \prod_{\beta\in [\ell]\setminus \alpha}(w_\alpha-w_\beta)(\overline{z_\alpha}-\overline{z_\beta}) = 
	\Fe_{C_\ell}(\tbz,\tbw) P^\ell(w_\alpha,\overline{z_\alpha}) = Q^\ell(w_\alpha,\overline{z_\alpha}) \\
	&\qquad = \hspace{-0.5cm}\sum_{\substack{\texttt{I}_{[\ell]}\trieq \si\prec C_\ell \\ C_\ell\circ \si^{-1}(\alpha)\neq \alpha,\, \si^{-1}\circ C_\ell(\alpha)\neq \alpha}} \hspace{-0.5cm}\CW_\si\Fe_\si\cdot  \Fq^\tbz_{[\ell] \setminus \SV_{nf}(C_\ell \circ \si^{-1})}(\overline{z_\alpha})\Fp^\tbw_{[\ell] \setminus \si^{-1}(\SV_{nf}(C_\ell \circ \si^{-1}))}(w_\alpha)\,,
\end{aligned}
\end{equation}
where in the last equality we restricted the sum only to those permutation for which $\alpha$ is not a fixed point of $C_\ell\circ \si^{-1}$ nor $\si^{-1}\circ C_\ell$, since for the remaining permutations either $\Fq^\tbz_{[\ell] \setminus \SV_{nf}(C_\ell \circ \si^{-1})}(\overline{z_\alpha})=0$ or $\Fp^\tbw_{[\ell] \setminus \si^{-1}(\SV_{nf}(C_\ell \circ \si^{-1}))}(w_\alpha)=0$.

Let us start by proving that $\Fe_\si$ are all polynomials in $\overline{\tbz}$ and $\tbw$. Due to the tensorial property of the eigenvectors and the definition of $\sfV_\si$, we have that $\Fe_\si = \prod_{\CL\in \si}\Fe_\CL$. Hence, it suffices to show that $\Fe_{C_\ell}$, for $\ell\geq 1$ are all polynomials. We prove this by induction on $\ell$. For $\ell=1$, we have by definition $\Fe_{C_1} = \frl_{\texttt{I}_{[1]}}(C_1) = 1$ and the result follows. Next, assume the result holds for $(\Fe_{C_m})_{m\leq \ell-1}$. Due to \eqref{eq:Recurssive_poly} and the induction assumption, it suffices to show that each summand on the right hand side is divisible by $\prod_{\beta\in [\ell]\setminus\alpha}(w_\alpha-w_\beta)(\overline{z_\alpha}-\overline{z_\beta})$. Fix $\beta\in [\ell]\setminus \alpha$ and let $\texttt{I}_{[\ell]}\trieq \si\prec C_\ell$ such that $C_\ell\circ \si^{-1}(\alpha)\neq \alpha$ and $\si^{-1}\circ C_\ell(\alpha)\neq \alpha$ . If $\beta$ is not in the same cycle of $\si$ as $\alpha$, then the numerator of $\CW_\si$ contains in its product the term $(w_\alpha-w_\beta)(\overline{z_\alpha}-\overline{z_\beta})$, while the denominator does not, hence $\CW_\si$ is divisible by $(w_\alpha-w_\beta)(\overline{z_\alpha}-\overline{z_\beta})$ and so is the summand related to $\si$. On the other hand, since $\si\prec C_\ell$, all cycles of $\si$ are of type $1$ with respect to $C_\ell$ and hence the unique vertex in the cycle of $\alpha$ in $\si$ which is not a fixed point is $\alpha$ itself. In particular, if $\beta$ and $\alpha$ are in the same cycle in $\si$, then $\beta\notin \SV_{nf}(\si,C_\ell)$ and therefore the product $\Fq^\tbz_{[\ell] \setminus \SV_{nf}(C_\ell \circ \si^{-1})}(\overline{z_\alpha})\Fp^\tbw_{[\ell] \setminus \si^{-1}(\SV_{nf}(C_\ell \circ \si^{-1}))}(w_\alpha)$ is divisible by $(w_\alpha-w_\beta)(\overline{z_\alpha}-\overline{z_\beta})$. This complete the proof that $\Fe_\si$ are all polynomials. 

Next, we turn to discuss the homogeneity of the polynomials. Going back to the representation $\frl_{\texttt{I}_{[\ell]}}(C_\ell) P^\ell(\lambda,\kappa) = Q^\ell(\lambda,\kappa)$. Taking $\kappa=\overline{\lambda}$ and equating the coefficients of $|\lambda|^{2\ell-4}$ in both sides (recall that $P^\ell$ and $Q^\ell$ are polynomials of degree $\ell-2$ is $\lambda$ and $\kappa$) we obtain
\[
	\frl_{\texttt{I}_{[\ell]}}(C_\ell)\cdot \bigg(-\sum_{\alpha\in[\ell]} \overline{z_\alpha}(w_{C_\ell^{-1}(\alpha)}-w_\alpha)\bigg) = \sum_{\substack{\texttt{I}_{[\ell]}\trieq \si\prec C_\ell \\ |\si| = 2}}\frl_{\texttt{I}_{[\ell]}}(\si)\,,
\] 
and therefore
\begin{equation}\label{eq:proving_stuff_1}
	\Fe_{C_\ell} \cdot \bigg(-\sum_{\alpha\in[\ell]} \overline{z_\alpha}(w_{C_\ell^{-1}(\alpha)}-w_\alpha)\bigg) = \sum_{\substack{\texttt{I}_{[\ell]}\trieq \si\prec C_\ell \\ |\si| = 2}}\CW_\si \Fe_\si\,.
\end{equation}

As in previous claim, the proof follows by induction on $\ell$ and the tensorial property which implies it suffices to prove the statement for the cyclic permutations $(C_\ell)_{\ell\geq 1}$. For $\ell=1$ we have $\Fe_{C_1} = 1$, which is a homogeneous of degree $0=\binom{1-1}{2}$. Assume the statement for $\Fe_{C_m}$ with $m\leq \ell-1$. Noting that each $\si$ in the sum on the right hand side of \eqref{eq:proving_stuff_1} is composed of two cycles whose total length is $\ell$, and that their length is $i$ and $\ell-i$ for $1\leq i\leq \ell-1$, respectively, it follows from the induction assumption that the corresponding term on the right hand side of \eqref{eq:proving_stuff_1} is a homogeneous polynomials in $\overline{\tbz}$ and $\tbw$ of degree $\binom{\ell}{2} - \binom{i}{2} - \binom{\ell-i}{2} + \binom{i-1}{2} + \binom{\ell-i-1}{2} = \binom{\ell-1}{2}+1$. Hence $\Fe_{C_\ell}$ is a homogeneous polynomial in $\overline{\tbz}$ and $\tbw$ of degree $\binom{\ell-1}{2}$, as required. 

Finally, we turn to prove that for $\ell\geq 3$, the polynomial $\Fe_{C_\ell}$ vanish on the hyperplanes 
\[
	\{(\bfz, \bfw)\in(\BC^{[\ell]})^2 ~:~ (z_\alpha, w_{\alpha})= (z_\beta, w_{\beta})\},\qquad \forall \alpha,\beta\in [\ell] ~\text{ such that }\alpha\neq \beta\,.
\]

To this end, fix $\alpha\neq \beta$. For $\ell=3$ we have by  a direct computation (using for example \eqref{eq:proving_stuff_1}) that the result holds. We use one last induction over $\ell$. Assuming the result holds for $\Fe_{C_m}$ with $3\leq m\leq \ell-1$. We start by proving that the sum on the right hand side of \eqref{eq:proving_stuff_1} vanishes on the hyperplane $\{(\bfz, \bfw)\in(\BC^{[\ell]})^2 ~:~ (z_\alpha, w_{\alpha})= (z_\beta, w_{\beta})\}$. Indeed, if  $\alpha, \beta\in \sigma$, then $\fre^{(|\sigma|)}_{\sigma}=0$ vanished on the hyperplane by induction. Otherwise $\alpha$ and $\beta$ are not in the same cycle, and therefore $\caW_\sigma=0$ by inspection. Finally, noting that the polynomial multiplying $\Fe_{C_\ell}$ on the left hand side of \eqref{eq:proving_stuff_1} does not vanish in a typical points in the hyperplane $\{(\bfz, \bfw)\in(\BC^{[\ell]})^2 ~:~ (z_\alpha, w_{\alpha})= (z_\beta, w_{\beta})\}$, it follows that $\fre^{(\ell)}(\bfz, \bfw)$ must vanish on the entire set $(z_\alpha, w_{\alpha})= (z_\beta, w_{\beta})$.    
\qed

\appendix

\section{Proof of Lemma \ref{L:pfe}}\label{appendix}

\begin{proof}[Proof of Lemma \ref{L:pfe}(1)] We start by proving \eqref{eq:k1}. Changing variables, one can rewrite $\frK_{i,j}$ as 
\begin{equation}\label{eq:rewriting_Frk}
	\frK_{i,j}(w,z)=\frac{1}{\pi}\int_\C  \frac{\lambda^{i-1}\overline{\lambda}^{j-1}}{\sqrt{(i-1)!(j-1)!}} (\lambda-\sqrt{N}w)^{-1}(\overline{\lambda}-\sqrt{N}\overline{z})^{-1} e^{-|\lambda|^{2}}\textrm{d}^2 \lambda \,.
\end{equation}
Denote by $B_r(z)$ the open ball of radius $r>0$ around $z\in\BC$. We split the integration over $\BC$ into three regions: $R_1:=B_{|z-w|\sqrt{N}/2}(\sqrt{N}z)$, $R_2:=B_{|z-w|\sqrt{N}/2}(\sqrt{N}w)$ and $R_3:=\BC\setminus (R_1\cup R_2)$ and turn to estimate each of the integrals separately.

In the region $R_3$, $|(\lambda-\sqrt{N}w)^{-1}(\overline{\lambda}-\sqrt{N}\overline{z})^{-1}|\leq 4/(|z-w|^2N)$ and therefore 
\begin{align*}
	&\bigg|\frac{1}{\pi}\int_{R_3}  \frac{\lambda^{i-1}\overline{\lambda}^{j-1}}{\sqrt{(i-1)!(j-1)!}} (\lambda-\sqrt{N}w)^{-1}(\overline{\lambda}-\sqrt{N}\overline{z})^{-1} e^{-|\lambda|^{2}}\textrm{d}^2 \lambda\bigg|\\
	\leq &\frac{4}{\pi|z-w|^2 N\sqrt{(i-1)!(j-1)!}}\int_{R_3}  |\lambda|^{i+j-2} e^{-|\lambda|^{2}}\textrm{d}^2 \lambda\\
	= & \frac{4}{|z-w|^2 N\sqrt{(i-1)!(j-1)!}}\int_0^{\infty}  u^{(i+j-2)/2} e^{-u}\textrm{d}u\,,
\end{align*}
where in the last step we replaced the integration over $R_3$ by integration over $\BC$ and made two changes of coordinates, first to polar coordinates and second replacing the radial coordinate $r$ by $u=r^2$. Using the Cauchy-Schwarz inequality and the fact that $\int_0^\infty u^{j-1}e^{-u} \textrm{d}u= (j-1)!$ we conclude that
\[
	\bigg|\frac{1}{\pi}\int_{R_3}  \frac{\lambda^{i-1}\overline{\lambda}^{j-1}}{\sqrt{(i-1)!(j-1)!}} (\lambda-\sqrt{N}w)^{-1}(\overline{\lambda}-\sqrt{N}\overline{z})^{-1} e^{-|\lambda|^{2}}\textrm{d}^2 \lambda\bigg| \leq \frac{4}{|z-w|^2 N}\,.
\]

Next we turn to estimate the integration over the regions $R_1$ and $R_2$. Since the estimations for both regions are similar we only consider the integration over $R_1$. Using the bound $k!\geq (k/e)^k$ one can verify that 
\[
	\bigg|\frac{\lambda^{j-1}}{\sqrt{(j-1)!}} \frac{\lambda^{l-1}}{\sqrt{(l-1)!}}e^{-|\lambda|^2}\bigg|\leq 1\,.
\]
Furthermore, in the region $R_1$ we have $|(\lambda-\sqrt{N}w)^{-1}|\leq 2/(|z-w| \sqrt{N})$ and therefore 
\begin{align*}
	& \bigg|\frac{1}{\pi}\int_{R_1}  \frac{\lambda^{i-1}\overline{\lambda}^{j-1}}{\sqrt{(i-1)!(j-1)!}} (\lambda-\sqrt{N}w)^{-1}(\overline{\lambda}-\sqrt{N}\overline{z})^{-1} e^{-|\lambda|^{2}}\textrm{d}^2 \lambda\bigg| \\
	 &	\qquad \leq\frac{2}{\pi|z-w|\sqrt{N}}\int_{R_1}  |(\overline{\lambda}-\sqrt{N}\overline{z})|^{-1} \textrm{d}^2 \lambda = \frac{2}{|z-w|\sqrt{N}} \int_0^{|z-w|\sqrt{N}/2}du=1\,,
\end{align*}
proving inequality \eqref{eq:k1}.
\end{proof}

\vspace{0.3cm}
Before turning to the proof of \eqref{eq:k2} let us introduce an auxiliary kernel and provide some preliminary estimations. Let
\[
	K(a,b)= e^{-(|a|^2+|b|^2)/2}\sum_{j=0}^{N-1}\frac{(\overline{a}b)^j}{j!}\,. 
\]

\begin{claim}\label{clm:bounds_on_K} $~$
\begin{enumerate}
\item We have
	\begin{equation}\label{eq:bound_on_K}
		|K(a,b)| \leq \begin{cases}
		e^{-(|a|-|b|)^2/2} & \text{for all }a,b\in\BC \\
		e^{-(|a|-|b|)^2/2 - |a||b|/2} & \text{for all }a,b\in\BC \text{ such that }|a|,|b|>2\sqrt{N}
		\end{cases}\,.
	\end{equation}
\item There exists a universal constant $C\in (0,\infty)$ such that for every $\tbz,\tbw\in \tbD_1$ and every $b\in\BC$
\begin{equation}\label{eq:bound_on_K_2}
	\bigg| \frac{1}{\pi}\int_{\BC} \frac{K(a,b)}{(a-\sqrt{N}w)(\overline{a}-\sqrt{N}\overline{z})}d^2a\bigg| \leq \frac{C\log N}{|z-w|^2\sqrt{N}}\,.
\end{equation}
\end{enumerate}
\end{claim}

\begin{proof}[Proof of Claim \ref{clm:bounds_on_K}] Denoting by $X_u$ a random variable distributed as a Poisson$(u)$ and observing that $\tbP(X_u\leq N-1)=e^{-u}\sum_{i=0}^{N-1}\frac{u^{i-1}}{(i-1)!}$ we can rewrite $K$ as
\[
	|K(a,b)| \leq  e^{-(|a|-|b|)^2/2} e^{-|a||b|}\sum_{j=0}^{N-1} \frac{|ab|^j}{j!} = e^{-(|a|-|b|)^2/2}\tbP(X_{|ab|}\leq N-1)\,.
\]	
Standard large deviation estimates imply $\tbP(X_{|ab|}\leq N-1)< e^{-|ab|/2}$ for sufficiently large $N$, provided $|ab|\geq 4N$.  The bound \eqref{eq:bound_on_K} follows from this estimate immediately.	

Next, we turn to the proof of \eqref{eq:bound_on_K_2}. We split the integral into two regions $|a|\geq  2\sqrt{N}$ and $|a|< 2\sqrt{N}$. Starting with the former, using \eqref{eq:bound_on_K} and the fact that $|(a-\sqrt{N}w)(\overline{a}-\sqrt{N}\overline{z})|\geq N$ for $|a|\geq 2\sqrt{N}$ we obtain 
\begin{align*}
	& \quad \bigg| \frac{1}{\pi}\int_{|a|\geq 2\sqrt{N}} \frac{K(a,b)}{(a-\sqrt{N}w)(\overline{a}-\sqrt{N}\overline{z})}d^2a\bigg|\leq \frac{1}{\pi N}\int_{|a|\geq 2\sqrt{N}}|K(a,b)|d^2a \\
	& \qquad\quad\leq \begin{cases} \frac{1}{\pi N}\int_{|a|\geq 2\sqrt{N}}e^{-(|a|-|b|)^2/2}d^2a & |b|< 2\sqrt{N} \\
	\frac{1}{\pi N}\int_{|a|\geq 2\sqrt{N}}e^{-\frac{(|a|-|b|)^2}{2}-\frac{1}{2}|a||b|}d^2a & |b|\geq 2\sqrt{N}\end{cases}
	 \leq \begin{cases} \frac{C}{\sqrt{N}} & |b|< 2\sqrt{N} \\ 
	\frac{C}{N} & |b|\geq 2\sqrt{N}\end{cases}\,
\end{align*}
for some universal constant $C$.   In the last inequality we moved to polar coordinates and used the fact that the indefinite integral of both integrands can be written explicitly. 

In order to estimate the integral over $|a|<2\sqrt{N}$ we split the integration further into three regions $R_1:=B_{|z-w|\sqrt{N}/2}(\sqrt{N}z)$, $R_2:=B_{|z-w|\sqrt{N}/2}(\sqrt{N}w)$ and $R_3:=B_{2\sqrt{N}}(0)\setminus (R_1\cup R_2)$ and estimate each part separately. For $R_3$ we have $|(a-\sqrt{N}w)(\overline{a}-\sqrt{N}\overline{z})|\geq N|z-w|^2/4$ and therefore 
\begin{align}
	 \bigg| \frac{1}{\pi}\int_{R_3} \frac{K(a,b)}{(a-\sqrt{N}w)(\overline{a}-\sqrt{N}\overline{z})}d^2a\bigg| 
	&\leq  \frac{4}{N|z-w|^2\pi}\int_{R_3} |K(a,b)|d^2a \nonumber\\
	&\label{E:name}  \leq \frac{4}{N|z-w|^2\pi}\int_{|a|<2\sqrt{N}}e^{-(|a|-|b|)^2/2}d^2a 
\end{align}
Since
\begin{align*}
	&\frac{1}{\pi N}\int_{|a|< 2\sqrt{N}}e^{-(|a|-|b|)^2/2}d^2a = \frac{2}{N}\int_0^{2\sqrt{N}} e^{-(r-|b|)^2/2}rdr \\
	=& \frac{2}{N}\bigg[e^{-|b|^2/2}-e^{-(2\sqrt{N}-|b|)^2/2}+\sqrt{\frac{\pi}{2}}|b|\Big(\mathrm{erf}(|b|/\sqrt{2})-\mathrm{erf}\Big(\frac{1}{\sqrt{2}}(|b|-2\sqrt{N}\Big)\Big)\bigg]\\
	\leq & \frac{2}{N} + \frac{\sqrt{2\pi}|b|}{N}\Big(\mathrm{erf}(|b|/\sqrt{2})-\mathrm{erf}\Big(\frac{1}{\sqrt{2}}(|b|-2\sqrt{N}\Big)\Big)\,,
\end{align*}
and the last term is bounded by $C/\sqrt{N}$ for some universal constant $C\in(0,\infty)$, provided $|b|\leq 3\sqrt{N}$ and decays exponentially for $|b|\geq 3\sqrt{N}$, \eqref{E:name} follows.

Since the regions $R_1$ and $R_2$ are dealt similarly, we only provide the details for $R_1$. Note that $|a-\sqrt{N}w|\geq |z-w|\sqrt{N}/2$ for $a\in R_1$ and therefore 
\begin{align*}
	& \bigg| \frac{1}{\pi}\int_{R_1} \frac{K(a,b)}{(a-\sqrt{N}w)(\overline{a}-\sqrt{N}\overline{z})}d^2a\bigg| 
	\leq  \frac{2}{\sqrt{N}|z-w|\pi}\int_{R_1} \frac{|K(a,b)|}{|a-\sqrt{N}z|}d^2a\\
	&\qquad \leq  \frac{2}{\sqrt{N}|z-w|\pi}\int_{|\nu|<|z-w|\sqrt{N}/2} \frac{e^{-(|\nu+\sqrt{N}\overline{z}|-|b|)^2/2}}{|\nu|}d^2\nu\\
	&\qquad \leq  \frac{2}{\sqrt{N}|z-w|\pi}\int_0^{|z-w|\sqrt{N}/2}\int_0^{2\pi} e^{-(\left|r+|z|\sqrt{N}e^{i\theta}\right|-|b|)^2/2}drd\theta\,.
\end{align*}
For $|b|\geq 3\sqrt{N}$, the integrand is bounded by $e^{-N}$ which implies the result. Hence, for the rest of the proof assume that $|b|<3\sqrt{N}$. 

For $0\leq r\leq |z-w|\sqrt{N}/2$, let $S_1=\{w\in\BC ~:~|w|=|b|\}$ and $S_2(r)=\{w\in\BC ~:~ w=r+|z| \sqrt{N} e^{i\theta} \text{ for some }\theta\in[0,2\pi)\}$ and let 
\[
	R_0=\{r\in [0, |z-w|\sqrt{N}/2]: S_1 \text{ and } S_2(r) \text{ do not intersect transversally}\}.
\]
One can now verify that 
\[
	\frac{1}{2\pi}\int_{R_0} \int_0^{2\pi}  e^{-(\left|r+|z| \sqrt{N} e^{i\theta}\right|-|b|)^2/2}\textrm{d} \theta\textrm{d} r\leq \sqrt{2\pi}\,.
\]

Let $R_1=[0, |z-w|\sqrt{N}/2]\backslash R_0$.  To estimate the integral over this region, we need to use the geometry of our problem more carefully. Fix $r\in R_1$ and let $\{\theta_0(r), -\theta_0(r)\}$ denote the two intersections between $S_1$ and $S_2(r)$. By symmetry
\[
\int_0^{2\pi} e^{-(\left|r+ |z|\sqrt{N} e^{i\theta}\right|-|b|)^2/2}\textrm{d} \theta= 2 \int_0^{\pi}  e^{-(\left|r+ |z|\sqrt{N} e^{i\theta}\right|-|b|)^2/2}\textrm{d} \theta\,.
\]

For convenience, let us assume that $\sqrt{N}|z|\leq |b|$, the other case being handled similarly. Under the last assumption $\theta_0\in [-\pi/2, \pi/2]$. 
We split the integral over $[0, \pi]$ into the pair of integrals over the regions $A_1= [0, \theta_0]$ and $A_2=[\theta_0, \pi]$.  Using linear interpolation, we see that
for $\theta\in A_1$,
\[
\left||r+ |z|\sqrt{N} e^{i\theta}|-|b|\right|\geq (r+|z|\sqrt{N}-|b|)(\theta_0-\theta)/\theta_0
\]
while for $\theta\in A_2$
\[
\left||r+|z|\sqrt{N} e^{i\theta}|-|b|\right|\geq (|b|-||z|\sqrt{N}-r|)(\theta-\theta_0)/(\pi-\theta_0).
\] 
We have the uniform bound
\beq
\label{E:uni1}
\frac{1}{2\pi}\int_0^{2\pi} \textrm{d} \theta e^{-(\left|r+|z|\sqrt{N} e^{i\theta}\right|-|b|)^2/2}\leq 1
\eeq
which, when combined to the above,  leads to the bound
\[
\frac{1}{2\pi}\int_0^{2\pi} \textrm{d} \theta e^{-(\left|r+|z|\sqrt{N} e^{i\theta}\right|-|b|)^2/2}\leq \underbrace{\sqrt{2\pi}^{-1}\bigg(\frac{\theta_0(r)}{r+|z|\sqrt{N}-|b|}+\frac{\pi-\theta_0(r)}{|b|-|z|\sqrt{N}-r|} \bigg)}_{f(r)}.
\]
Using the bounds above and 
\[
\int_{R_1}  \frac{1}{2\pi}\int_0^{2\pi} e^{-(\left|r+|z|\sqrt{N} e^{i\theta}\right|-|b|)^2/2}\textrm{d} \theta \textrm{d} r
\leq \int_{R_1}  \min\{1, f(r)\} \textrm{d} r\leq C\log N
\]
Combined with the prior estimate on the integral over $R_0$, the proof is complete.
\end{proof}

\begin{claim}\label{clm:one_more_integral_bound} There exists a universal constant $C\in (0,\infty)$ such that for  every $\tbz,\tbw\in \tbD_1$
\begin{equation}\label{eq:bound_on_K_3}
	\frac{1}{\pi}\int_{|a|\leq 2\sqrt{N}} \bigg| \frac{1}{(a-\sqrt{N}w)(\overline{a}-\sqrt{N}\overline{z})} \bigg| d^2a \leq \frac{C}{|z-w|^2}\,.
\end{equation}
\end{claim}

\begin{proof}
	As in previous cases we split the integral into the regions $R_1:=B_{|z-w|\sqrt{N}/2}(\sqrt{N}z)$, $R_2:=B_{|z-w|\sqrt{N}/2}(\sqrt{N}w)$ and $R_3:=B_{2\sqrt{N}}(0)\setminus (R_1\cup R_2)$ and estimate each part separately. For $R_3$ we have 
\[
	 \frac{1}{\pi}\int_{R_3} \bigg| \frac{1}{(a-\sqrt{N}w)(\overline{a}-\sqrt{N}\overline{z})}\bigg| d^2a	\leq \frac{4}{\pi N|z-w|^2}\int_{R_3}d^2a \leq \frac{16}{|z-w|^2}\,,
\]
where in the last step we bounded the integral over $R_3$ by the integral over $|a|<2\sqrt{N}$. 

Turning to deal with the integrations over $R_1$ and $R_2$, by symmetry it suffices to prove the result only for $R_1$. 
\begin{align*}
	& \frac{1}{\pi}\int_{R_1} \bigg| \frac{1}{(a-\sqrt{N}w)(\overline{a}-\sqrt{N}\overline{z})}\bigg| d^2a \leq \frac{2}{\pi|z-w| \sqrt{N}}\int_{|a-\sqrt{N}z|<|z-w|\sqrt{N}/2}\frac{1}{|a-\sqrt{N}z|}d^2a\\
	& = \frac{4}{|z-w| \sqrt{N}}\int_0^{|z-w|\sqrt{N}/2}dr = 2\,.
\end{align*}
\end{proof}

\begin{proof}[Proof of Lemma \ref{lem:pfe}(2)] Notice that due to \eqref{eq:rewriting_Frk}
\[
	\mathrm{Tr}(\frK(w_1,z_1)\frK(w_2,z_2)\cdots \frK(w_k,z_k)) = \frac{1}{\pi^k} \int_{\BC^k} \prod_{j=1}^k \frac{K(\lambda_j,\lambda_{j+1})}{(\lambda_j-\sqrt{N}w_j)(\overline{\lambda_j}-\sqrt{N}\overline{z_j})}\prod_{j=1}^k \mathrm{d}^2\lambda_j\,.
\]
where we use cyclic indexing $k+1=1$. 

For $1\leq j\leq k$, let $R_1(j):= \{\lambda_j: |\lambda_j|\leq 2\sqrt{N}\}$ and let $R_2(j)=\C\setminus R_1(j)$. We split the integration over $\BC^k$ into $2^k$ regions by defining for $\tba\in \{1, 2\}^{k}$ the region $R^{\tba} = R_{a_1}(1)\times R_{a_2}(2)\times \ldots \times R_{a_k}(k)$. 
Then
\[
	\mathrm{Tr}(\frK(w_1,z_1)\frK(w_2,z_2)\cdots \frK(w_k,z_k)) =\hspace{-0.3cm} \sum_{\tba\in \{1,2\}^k} \frac{1}{\pi^k} \int_{R^{\tba}} \prod_{j=1}^k \frac{K(\lambda_j,\lambda_{j+1})}{(\lambda_j-\sqrt{N}w_j)(\overline{\lambda_j}-\sqrt{N}\overline{z_j})}\prod_{j=1}^k \mathrm{d}^2\lambda_j\,.
\] 

We distinguish between two types of regions $R^{\tba}$: Either $a_j=2$ for all $1\leq j \leq k$ or $a_j=1$ for some $1\leq j\leq k$. Starting with the former, note that in this case $|(\lambda_j-\sqrt{N}w_j)(\overline{\lambda_j}-\sqrt{N}\overline{z_j})|\geq N$. Furthermore, it follows from \eqref{eq:bound_on_K} that $|K(a,b)|\leq e^{-(|a|^2+|b|^2)/4}$ for $a,b\in\BC$ such that $|a|,|b|>2\sqrt{N}$, and therefore 
\begin{equation}\label{eq:bound_on_K_1}
\begin{aligned}
	&\bigg|\frac{1}{\pi^k} \int_{R^{\textbf{2}}} \prod_{j=1}^k \frac{K(\lambda_j,\lambda_{j+1})}{(\lambda_j-\sqrt{N}w_j)(\overline{\lambda_j}-\sqrt{N}\overline{z_j})}\prod_{j=1}^k \mathrm{d}^2\lambda_j\bigg|\\
	&\qquad \qquad \leq \bigg(\frac{1}{\pi N}\int_{|a|\geq 2\sqrt{N}} e^{-|a|^2/2}d^2a\bigg)^k =\bigg(\frac{4e^{-2N}}{N}\bigg)^k\,.
\end{aligned}
\end{equation}

Next, we estimate the second type of region, namely, ones in which $a_j=1$ for some $1\leq j\leq k$. Without loss of generality assume that $a_k=1$. Using \eqref{eq:bound_on_K_2}, \eqref{eq:bound_on_K_3} and the bound $|K(\lambda_k,\lambda_1)|\leq 1$ which follows from \eqref{eq:bound_on_K}, we can integrate the variables $\lambda_1,\lambda_2,\ldots,\lambda_k$ one by one. Each of the first $k-1$ integrals is bounded by $C\log(N)/(\mathrm{dist}(\tbz,\tbw)\sqrt{N})$ by \eqref{eq:bound_on_K_2}, while the last integral (over $\lambda_k$) is bounded by $C/\mathrm{dist}(\tbz,\tbw)$ due to \eqref{eq:bound_on_K_3}. Combining all of the above we conclude that in each region as above 
\[
	\bigg|\frac{1}{\pi^k} \int_{R^{\tba}} \prod_{j=1}^k \frac{K(\lambda_j,\lambda_{j+1})}{(\lambda_j-\sqrt{N}w_j)(\overline{\lambda_j}-\sqrt{N}\overline{z_j})}\prod_{j=1}^k \mathrm{d}^2\lambda_j\bigg|\leq \frac{C}{\mathrm{dist}(\tbz,\tbw)^2}\bigg(\frac{C\log N}{\mathrm{dist}(\tbz,\tbw)^2\sqrt{N}}\bigg)^{k-1}\,.
\]
Using \eqref{eq:bound_on_K_1} to bound the integral over $R^{\textbf{2}}$ and the last estimation to bound all other regions, we conclude that 
\[
	|\mathrm{Tr}(\frK(w_1,z_1)\frK(w_2,z_2)\cdots \frK(w_k,z_k))| \leq \bigg(\frac{C}{\mathrm{dist}(\tbz,\tbw)^2}\bigg)^k\cdot \bigg(\frac{\log N}{\sqrt{N}}\bigg)^{k-1}\,,
\]
as required. 
\end{proof}

Finally, we turn to prove \eqref{eq:k3}. 

\begin{proof}[Proof of Lemma \ref{lem:pfe}(3)]
From \eqref{eq:rewriting_Frk}
\[
	\mathrm{Tr}(\frK(w,z))=\frac{1}{\pi}\int_\C  \sum_{i=0}^{N-1}\frac{|\lambda|^{2i-2}}{(i-1)!} (\lambda-\sqrt{N}w)^{-1}(\overline{\lambda}-\sqrt{N}\overline{z})^{-1} e^{-|\lambda|^{2}}\textrm{d}^2 \lambda \,.
\]
Similarly to the proof of \eqref{eq:k1} we split the integral into regions by considering $R_1:=B_{\sqrt{N-N^{1/2+\delta}} }(0)$ and $R_2:=\BC\setminus R_1$.

Starting with $R_2$ and using the fact that $z$ and $w$ are in the open disc $\tbD_1$, we note that $|(\lambda-\sqrt{N}w)^{-1}(\overline{\lambda}-\sqrt{N}\overline{z})^{-1}|\leq 4(1-|w|)^{-1}(1-|z|)^{-1}N^{-1}$ for all $\lambda\in R_2$ and sufficiently large $N$. Hence
\begin{align}\label{eq:proof_of_pfe_2}
	& \bigg|\frac{1}{\pi}\int_{R_2} \sum_{i=0}^{N-1}\frac{|\lambda|^{2i-2}}{(i-1)!} (\lambda-\sqrt{N}w)^{-1}(\overline{\lambda}-\sqrt{N}\overline{z})^{-1} e^{-|\lambda|^{2}}\textrm{d}^2 \lambda\bigg|\nonumber\\
	 & \qquad \leq\frac{4}{\pi N(1-|z|)(1-|w|)}\int_{R_2}  \sum_{i=0}^{N-1}\frac{|\lambda|^{2i-2}}{(i-1)!}  e^{-|\lambda|^{2}}\textrm{d}^2 \lambda \nonumber\\
	 & \qquad = \frac{4}{N(1-|z|)(1-|w|)}\int_{N-N^{1/2+\delta}}^\infty  \sum_{i=0}^{N-1}\frac{u^{i-1}}{(i-1)!}  e^{-u}\textrm{d}u\,. 
\end{align}
If  $X_u$ denotes  a random variable distributed as a Poisson$(u)$ we can rewrite the last bound as 
\begin{align*}
	&\frac{4}{N(1-|z|)(1-|w|)}\int_{N-N^{1/2+\delta}}^\infty \tbP(X_u\leq N-1)du \\
	&\qquad = \frac{4}{N(1-|z|)(1-|w|)}\bigg[\int_{N-N^{1/2+\delta}}^{N+N^{1/2+\delta}} \tbP(X_u\leq N-1)du
	+\int_{N+N^{1/2+\delta}}^\infty \tbP(X_u\leq N-1)du\bigg]\\
	&\qquad \leq \frac{4}{N(1-|z|)(1-|w|)}\cdot \bigg[2N^{1/2+\delta} + \int_{N+N^{1/2+\delta}}^\infty e^{N-u+N\log(u/N)}du\bigg]\leq \frac{10}{N^{1/2-\delta}}\,,
\end{align*}
where in the one before last inequality we bounded the probability by $1$ in the first integral and used Chernoff's inequality to bound the probability in the second. 

We now estimate the integral over $R_1$, which can also be expressed in terms of $X_u$ as 
\[
	\bigg|\frac{1}{\pi}\int_{R_1} \frac{\tbP(X_{|\lambda|^2}\leq N-1)}{(\lambda-\sqrt{N}w)(\overline{\lambda}-\sqrt{N}\overline{z})} \textrm{d}^2 \lambda\bigg|
\]
By Chernoff's inequality once again, $\forall\lambda\in R_1$,
\[
	|1-\tbP(X_{|\lambda|^2	}\leq N-1)| = \tbP(X_{|\lambda|^2	}\geq N) \leq e^{N-|\lambda|^2-N\log N+N\log(|\lambda|^2)} \leq e^{-N^{2\delta}}.
\]
Using the last estimation together with a similar argument to the one in Claim \ref{clm:computing_expectation_2}, we conclude that it suffices to prove
\begin{equation}\label{eq:contour1}
	\bigg|\frac{1}{\pi}\int_{R_1} \frac{1}{(\lambda-\sqrt{N}w)(\overline{\lambda}-\sqrt{N}\overline{z})} \textrm{d}^2 \lambda-h(w,z)\bigg|\leq \frac{C}{N^{1/2-\delta}}\,,
\end{equation}
for some universal constant $C\in (0,\infty)$. Using the change of coordinate $ \lambda= \nu \cdot \sqrt{N-N^{1/2+\delta}}$ we can rewrite the integral on the left hand side of \eqref{eq:contour1} as
\begin{equation}\label{eq:contour2}
	\frac{1}{\pi}\int_{|\nu|\leq 1} \frac{1}{(\nu-w_N)(\overline{\nu}-\overline{z_N})} \textrm{d}^2 \nu\,,
\end{equation}
where we denote $w_N=w/\sqrt{1-N^{-1/2+\delta}}$ and $z_N=z/\sqrt{1-N^{-1/2+\delta}}$.

To evaluate this integral, we use the complex form of Green's theorem. Let $\Gamma_\varepsilon$ be the contour along the $\partial \tbD_1$ with two punctures of width $\varepsilon$ surrounding $z_N$ and $w_N$ respectively. More precisely, let $\Gamma_{\epsilon}$ be the contour as schematically represented in Figure \ref{F:contour}. This contour has $5$ distinct parts. $\gamma_1$ denotes the part of the contour on the unit circle, $\gamma_3$ is the circle of radius $\epsilon$ around $z_N$ and $\gamma_5$ is the circle of radius $\epsilon$ around $w_N$.  Finally $\gamma_2$, respectively $\gamma_4$, is a pair of parallel straight line segments at distance $\epsilon^2$ of each other, connecting $\gamma_1$ to $\gamma_3$, respectively to $\gamma_5$.  

Then 
\[
	\eqref{eq:contour2}=\lim_{\epsilon \rightarrow 0} \frac{1}{2\pi\textrm{i}}\oint_{\Gamma_\epsilon} \omega\,,
\]
where $\omega=\omega_1(\lambda, \overline{\lambda}) \textrm{d} \lambda-\omega_2(\lambda, \overline{\lambda}) \textrm{d} \overline{\lambda}$
and $\omega_1, \omega _2 $ are the multi-valued functions 
\begin{equation}
\omega_1(\lambda, \overline{\lambda})= \frac 12 \frac{\log(\overline{\lambda}-\overline{z})}{\lambda-w_N}\qquad \text{and}\qquad 
\omega_2(\lambda, \overline{\lambda})=\frac 12 \frac{\log({\lambda}- {w})}{\overline{\lambda}-\overline{z_N}}\,.
\end{equation}

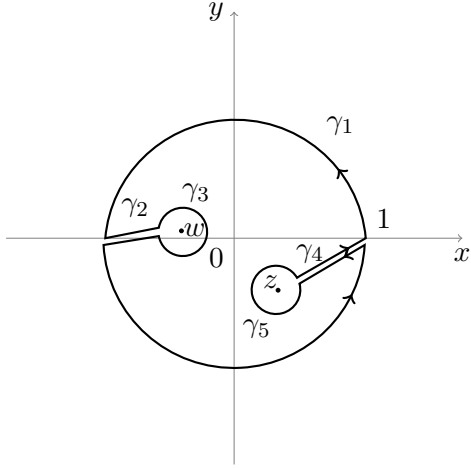
\begin{figure}
\centering
\begin{tikzpicture}
[decoration={markings,
mark=at position 0.75cm with {\arrow[line width=1pt]{>}},
mark=at position 2cm with {\arrow[line width=1pt]{>}},
mark=at position 14cm with {\arrow[line width=1pt]{>}},
mark=at position 15cm with {\arrow[line width=1pt]{>}}
}, dot/.style={draw,circle,minimum size=.5mm,inner sep=0pt,outer sep=0pt,fill=black}
]
\draw[help lines,->] (-3,0) -- (3,0) coordinate (xaxis);
\draw[help lines,->] (0,-3) -- (0,3) coordinate (yaxis);

\path[draw,line width=0.8pt,postaction=decorate] (-30:1)  -- +(.866,.5) node[above right] {$1$} arc (5:175:1.72) -- (172:1) arc (170:-170:.32)-- (183:1.72) arc (183:357:1.72)--(-35:1.05) arc (15:-330: .32)--cycle;

\node[below] at (xaxis) {$x$};
\node[left] at (yaxis) {$y$};
\node[below left] {$0$};
\node at (-1.3,.4) {$\gamma_{2}$};
\node at (-.5,.6) {$\gamma_{3}$};
\node at (1.4,1.5) {$\gamma_{1}$};
\node at (1,-.2) {$\gamma_{4}$};
\node at (.3,-1.2) {$\gamma_{5}$};
\node[dot] at  (172:.7) {};
\node[dot] at  (-50:.9) {};
\node at (-50:.75) {$z$};
\node at   (168:.53) {$w$} ;
\end{tikzpicture}
\caption{The contour $\Gamma_{\epsilon}$.}
\label{F:contour}
\end{figure}

Let us consider the integration of $\omega_1$ over the various parts of $\Gamma_\epsilon$.
First, we have 
\[
\lim_{\epsilon \rightarrow 0} \int_{\gamma_2} \omega_1(\lambda, \overline{\lambda}) \textrm{d} \lambda= \lim_{\epsilon \rightarrow 0} \int_{\gamma_5} \omega_1(\lambda, \overline{\lambda}) \textrm{d} \lambda=0
\]
and 
\[
\lim_{\epsilon \rightarrow 0} \frac{1}{2\pi\textrm{i}}\int_{\gamma_3} \omega_1(\lambda, \overline{\lambda}) \textrm{d} \lambda=-\frac{1}{2}\log(\overline{w_N}-\overline{z_N})\,.
\]
Also since $\log(\overline{\lambda}-\overline{z})$ increases by $2\pi \textrm{i}$ as we go around $\gamma_5$,
\[
\lim_{\epsilon \rightarrow 0}  \frac{1}{2\pi\textrm{i}}\int_{\gamma_4} \omega_1(\lambda, \overline{\lambda}) \textrm{d} \lambda=\frac 12 \log(1-w_N)-\frac 12\log(z_N-w_N)\,.
\]

Finally, 
\[\lim_{\epsilon \rightarrow 0}  \frac{1}{2\pi\textrm{i}}\int_{\gamma_4} \omega_1(\lambda, \overline{\lambda}) \textrm{d} \lambda=\frac{1}{2 \pi}\int_{0}^{2\pi} \frac{e^{\textrm{i}\theta} \log(e^{-\textrm{i} \theta}-\overline{z})}{e^{\textrm{i} \theta}-w_N} \textrm{d} \theta=\frac 12 \log(1-\overline{z_N}{w_N})-\frac 12\log(1-w_N)\,.
\]
Collecting the various pieces together gives
\[
\lim_{\epsilon\rightarrow 0}\frac{1}{2\pi\textrm{i}}\oint_{\Gamma_\epsilon} \omega_1 \textrm{d} \lambda= \frac 12\log\bigg(\frac{1-w_N\overline{z_N}}{|z_N-w_N|^2}\bigg)\,.
\]
By complex conjugation symmetry
\[
-\lim_{\epsilon\rightarrow 0}\frac{1}{2\pi\textrm{i}}\oint_{\Gamma_\epsilon} \omega_2\overline{ \textrm{d} \lambda}= \frac 12\log\bigg(\frac{1-w_N\overline{z_N}}{|z_N-w_N|^2}\bigg)\,,
\]
and therefore
\[
	\eqref{eq:contour2} = \log\bigg(\frac{1-w_N\overline{z_N}}{|z_N-w_N|^2}\bigg)\,.
\]
The result now follows since 
\[
\begin{aligned}
	\bigg|\log\bigg(\frac{1-w_N\overline{z_N}}{|z_N-w_N|^2}\bigg)-\log\bigg(\frac{1-w\overline{z}}{|z-w|^2}\bigg)\bigg| \\
	&\qquad = \log\bigg(1+ \frac{N^{-1/2+\delta}}{(1-w\overline{z})(1-N^{-1/2+\delta})}\bigg) \\
	\leq \frac{C}{N^{1/2-\delta}|1-w\overline{z}|}\,,
\end{aligned}
\]
for some universal constant $C$. Since $|1-w\overline{z}|\geq \max\{1-|w|,1-|z|\}$ the result follows. 
\end{proof}

\bibliography{Biblio}

\newcommand{\etalchar}[1]{$^{#1}$}
\begin{thebibliography}{BGN{\etalchar{+}}14}

\bibitem[AEK{\etalchar{+}}18]{AEK18}
Johannes Alt, L{\'a}szl{\'o} Erd{\H{o}}s, Torben Kr{\"u}ger, et~al.
\newblock Local inhomogeneous circular law.
\newblock {\em The Annals of Applied Probability}, 28(1):148--203, 2018.

\bibitem[Bai97]{bai1997}
Zhidong.~D. Bai.
\newblock Circular law.
\newblock {\em Ann. Probab.}, 25(1):494--529, 1997.

\bibitem[BD18]{BD}
Paul Bourgade and Guillaume Dubach.
\newblock The distribution of overlaps between eigenvectors of {G}inibre
  matrices.
\newblock {\em arXiv preprint arXiv:1801.01219}, 2018.

\bibitem[BGN{\etalchar{+}}14]{DBM}
Zdzislaw Burda, Jacek Grela, Maciej~A. Nowak, Wojciech Tarnowski, and Piotr
  Warcho\l{}.
\newblock Dysonian dynamics of the {G}inibre ensemble.
\newblock {\em Phys. Rev. Lett.}, 113:104102, Sep 2014.

\bibitem[BYY14a]{BYY14}
Paul Bourgade, Horng-Tzer Yau, and Jun Yin.
\newblock Local circular law for random matrices.
\newblock {\em Probability Theory and Related Fields}, 159(3-4):545--595, 2014.

\bibitem[BYY14b]{BYY14b}
Paul Bourgade, Horng-Tzer Yau, and Jun Yin.
\newblock The local circular law {II}: the edge case.
\newblock {\em Probab. Theory Related Fields}, 159(3-4):619--660, 2014.

\bibitem[CM98]{CM1}
John~T. Chalker and Bernhard Mehlig.
\newblock Eigenvector statistics in non-hermitian random matrix ensembles.
\newblock {\em Phys. Rev. Lett.}, 81:3367--3370, Oct 1998.

\bibitem[CR18]{CR2}
Nick Crawford and Ron Rosenthal.
\newblock Local eigenvector correlations in the complex {G}inibre ensemble.
\newblock {\em In preparation}, 2018.

\bibitem[FS12]{FS12}
Yan~V Fyodorov and Dmitry~V Savin.
\newblock Statistics of resonance width shifts as a signature of eigenfunction
  nonorthogonality.
\newblock {\em Physical review letters}, 108(18):184101, 2012.

\bibitem[Fyo17]{F}
Yan~V. Fyodorov.
\newblock On statistics of bi-orthogonal eigenvectors in real and complex
  {G}inibre ensembles: combining partial schur decomposition with
  supersymmetry.
\newblock {\em arXiv preprint arXiv:1710.04699}, 2017.

\bibitem[Gin65]{Gin}
Jean Ginibre.
\newblock Statistical ensembles of complex, quaternion, and real matrices.
\newblock {\em J. Mathematical Phys.}, 6:440--449, 1965.

\bibitem[Gir84]{Gir}
Vyacheslav~L. Girko.
\newblock The circular law.
\newblock {\em Teor. Veroyatnost. i Primenen.}, 29(4):669--679, 1984.

\bibitem[Gir94]{Gir94}
Vyacheslav~L. Girko.
\newblock The circular law: ten years later.
\newblock {\em Random Oper. Stochastic Equations}, 2(3):235--276, 1994.

\bibitem[GK98]{Goldshield}
Ilya~Ya. Goldsheid and Boris~A. Khoruzhenko.
\newblock Distribution of eigenvalues in non-{H}ermitian {A}nderson models.
\newblock {\em Phys. Rev. Lett.}, 80:2897--2900, Mar 1998.

\bibitem[GKL{\etalchar{+}}14]{GKL14}
J.-B. Gros, U.~Kuhl, O.~Legrand, F.~Mortessagne, E.~Richalot, and D.~V. Savin.
\newblock Experimental width shift distribution: A test of nonorthogonality for
  local and global perturbations.
\newblock {\em Phys. Rev. Lett.}, 113:224101, Nov 2014.

\bibitem[Grc11]{Grcar}
Joseph~F. Grcar.
\newblock John von {N}eumann's analysis of {G}aussian elimination and the
  origins of modern numerical analysis.
\newblock {\em SIAM Rev.}, 53(4):607--682, 2011.

\bibitem[GT10]{gotze}
Friedrich G\"otze and Alexander Tikhomirov.
\newblock The circular law for random matrices.
\newblock {\em Ann. Probab.}, 38(4):1444--1491, 2010.

\bibitem[HN96]{HN}
Naomichi Hatano and David~R. Nelson.
\newblock Localization transitions in non-{H}ermitian quantum mechanics.
\newblock {\em Phys. Rev. Lett.}, 77:570--573, Jul 1996.

\bibitem[JNN{\etalchar{+}}99]{Poles1}
Romuald~A. Janik, Wolfgang N\"orenberg, Maciej~A. Nowak, G\'abor Papp, and
  Ismail Zahed.
\newblock Correlations of eigenvectors for non-hermitian random-matrix models.
\newblock {\em Phys. Rev. E}, 60:2699--2705, Sep 1999.

\bibitem[MC98]{CM3}
Bernhard Mehlig and John~T. Chalker.
\newblock Eigenvector correlations in non-{H}ermitian random matrix ensembles.
\newblock {\em Ann. Phys.}, 7(5-6):427--436, 1998.

\bibitem[MC00]{CM2}
Bernhard Mehlig and John~T. Chalker.
\newblock Statistical properties of eigenvectors in non-{H}ermitian {G}aussian
  random matrix ensembles.
\newblock {\em J. Math. Phys.}, 41(5):3233--3256, 2000.

\bibitem[Meh04]{Mehta}
Madan~Lal Mehta.
\newblock {\em Random matrices}, volume 142 of {\em Pure and Applied
  Mathematics (Amsterdam)}.
\newblock Elsevier/Academic Press, Amsterdam, third edition, 2004.

\bibitem[Sta97]{Sta97}
Richard~P. Stanley.
\newblock {\em Enumerative combinatorics. {V}ol. 1}, volume~49 of {\em
  Cambridge Studies in Advanced Mathematics}.
\newblock Cambridge University Press, Cambridge, 1997.
\newblock With a foreword by Gian-Carlo Rota, Corrected reprint of the 1986
  original.

\bibitem[TE05]{Tref3}
Lloyd~N. Trefethen and Mark Embree.
\newblock {\em Spectra and pseudospectra}.
\newblock Princeton University Press, Princeton, NJ, 2005.
\newblock The behavior of nonnormal matrices and operators.

\bibitem[Tre97]{Tref2}
Lloyd~N. Trefethen.
\newblock Pseudospectra of linear operators.
\newblock {\em SIAM Rev.}, 39(3):383--406, 1997.

\bibitem[TTRD93]{Tref1}
Lloyd~N. Trefethen, Anne~E. Trefethen, Satish~C. Reddy, and Tobin~A. Driscoll.
\newblock Hydrodynamic stability without eigenvalues.
\newblock {\em Science}, 261(5121):578--584, 1993.

\bibitem[TV08]{TV}
Terence Tao and Van Vu.
\newblock Random matrices: the circular law.
\newblock {\em Commun. Contemp. Math.}, 10(2):261--307, 2008.

\bibitem[TV10]{TV10}
Terence Tao and Van Vu.
\newblock Random matrices: universality of {ESD}s and the circular law.
\newblock {\em Ann. Probab.}, 38(5):2023--2065, 2010.
\newblock With an appendix by Manjunath Krishnapur.

\bibitem[Wil89]{Wilk}
Michael Wilkinson.
\newblock Statistics of multiple avoided crossings.
\newblock {\em Journal of Physics A: Mathematical and General}, 22(14):2795,
  1989.

\bibitem[WS15]{WS}
Meg Walters and Shannon Starr.
\newblock A note on mixed matrix moments for the complex {G}inibre ensemble.
\newblock {\em J. Math. Phys.}, 56(1):013301, 20, 2015.

\bibitem[Yin14]{Y14}
Jun Yin.
\newblock The local circular law {III}: general case.
\newblock {\em Probab. Theory Related Fields}, 160(3-4):679--732, 2014.

\end{thebibliography}
\bibliographystyle{alpha}

\medskip{}

$~$\\
Department of mathematics,\\
Technion, Haifa,\\
Israel.\\
E-mail: nickc@technion.ac.il\\
E-mail: ron.ro@technion.ac.il

\end{document}